\documentclass[11pt]{amsart}
\usepackage{mathrsfs}
\usepackage{geometry}
\usepackage{graphicx}
\usepackage{amssymb}
\usepackage{epstopdf}
\usepackage{amsmath,amscd}
\usepackage{amsthm}
\usepackage{enumerate}
\usepackage{url,verbatim}
\usepackage{subfig}
\usepackage{enumitem}

\usepackage[normalem]{ulem}
\usepackage{fixme}

\RequirePackage[colorlinks,citecolor=blue,urlcolor=blue]{hyperref}
\usepackage{breakurl}
\theoremstyle{plain}
\calclayout
\newtheorem{theorem}{Theorem}
\newtheorem{definition}[theorem]{Definition}
\newtheorem{lemma}[theorem]{Lemma}
\newtheorem{proposition}[theorem]{Proposition}
\newtheorem{corollary}[theorem]{Corollary}
\newtheorem{example}[theorem]{Example}
\newtheorem{assumption}[theorem]{Assumption}

\newcommand\ol{\overline}

\newcommand\EE{{\mathbb E}}

\newcommand\RR{{\mathbb R}}
\newcommand\ZZ{{\mathbb Z}}
\newcommand\NN{{\mathbb N}}

\newcommand\YY{{\mathbb {Y}}}

\newcommand\g{\gamma}

\newcommand\si{\sigma}

\newcommand\q{\quad}

\newcommand\Si{\Sigma}

\renewcommand\ell{l}

\newcommand\CC{\mathbb{C}}
\newcommand\bm{\mathbf{m}}

\newcounter{mycount}

\numberwithin{equation}{section}
\numberwithin{theorem}{section}
\numberwithin{figure}{section}

\title{Limit shape of perfect matchings on rail-yard graphs}

\author{Zhongyang Li}

\address{Department of Mathematics, University of Connecticut, Storrs, CT, 06268}

\email{zhongyang.li@uconn.edu}

\begin{document}


\maketitle

\begin{abstract}
We obtain limit shape of perfect matchings on a large class of rail-yard graphs with right boundary condition given by the empty partition, and left boundary condition given by either by a staircase partition with constant density or a piecewise partition with densities either 1 or 0. We prove the parametric equations for the frozen boundary, and find conditions under which the frozen boundary is a cloud curve, or a union of disjoint cloud curves.
\end{abstract}

\section{Introduction}

A dimer cover, or a perfect matching of a graph is a subset of edges in which each vertex is incident to exactly one edge. The dimer model is a probability measure on the sample space consisting of all the perfect matchings of a graph (See \cite{RK09}).
The dimer model has origins from the structures of matter; for example, each perfect matching on a hexagonal lattice corresponds to a double-bond configuration of a graphite molecule; the dimer model on a Fisher graph has a measure-preserving correspondence with the 2D Ising model (see \cite{Fi66,MW73,ZL12}).

Partially inspired by the fact that in the structure of matter the probability of a certain molecule configuration depends on the underlying energy, mathematically we define a probability measure on the set of all perfect matchings of a graph depending on the energy of the dimer configuration, quantified as the product of weights of present edges in the configuration.
The asymptotical behavior and phase transition of the dimer model has been an interesting topic for mathematicians and physicists for a long time. A combinatorial argument shows that the total number of perfect matchings on any finite planar graph can be computed by the Pfaffian of the corresponding weighted adjacency matrix (\cite{Kas61,TF61}). The local statistics can be computed by the inverse of the weighted adjacency matrix (\cite{Ken01}); a complete picture of phase transitions was obtained in \cite{KOS}. Empirical results shows that in large graphs, there are certain regions that the ``randomness" of the configuration is completely lost, i.e. only one type of edges can occur in the dimer configuration. These are called ``frozen regions", and their boundary are called ``frozen boundary".
When the mesh size of the graph goes to 0 such that the graph approximates a simply-connected region in the plane, the limit shape of height functions can be obtained by a variational principle (\cite{ckp}), and the frozen boundary are proved to be an algebraic curve of a specific type (\cite{KO07}). It is also known that the fluctuations of (unrescaled) dimer heights converges to GFF in distribution when the boundary satisfies certain conditions (\cite{Ken01,Li13}).

In this paper, we investigate perfect matchings on a general class of bipartite graphs called rail-yard graph. This type of graphs were defined in \cite{bbccr}, and the formula to compute the partition function of pure dimer coverings on such graphs was also proved in \cite{bbccr}. The major goal of the paper is to study the limit shape of dimer coverings on such graphs whose left boundary conditions are given by either staircase or piecewise partitions and right boundary condition is given by an empty partition.
Special cases of rail-yard graphs include the Aztec diamond (\cite{EKLP92a,EKLP92b,KJ05,bk}), the pyramid partition (\cite{BY09}), the steep tiling (\cite{BCC17}), the tower graph (\cite{BF15}), the contracting square-hexagon lattice (\cite{BL17,ZL18,ZL20,Li182}) and the contracting bipartite graph (\cite{ZL202}).
Dimer coverings on rail-yard graphs are in one-to-one correspondence with sequences of partitions whose distributions are Schur processes. The main technique used in this paper to study limit shape of such dimer coverings is to analyze asymptotics of Schur polynomials as the size of the partition goes to infinity with general value of variables. The asymptotics of dimer coverings on rail-yard graphs with both the left boundary conditions and right boundary conditions given by the empty partition and modified weights were studied in \cite{LV21} by Macdonald processes.

The organization of the paper is as follows. In Section \ref{sect:bk}, we define the rail-yard graph, the perfect matching and the height function, and review related technical facts. In Section \ref{sect:sc}, we obtain the limit shape of dimer coverings on a rail-yard graph whose left boundary condition is given by a stair case partition whose counting measure has constant density.  We then discuss the frozen boundary in some special cases. The main theorems proved in this section are Theorems \ref{t33}, \ref{t38}. In Section \ref{sect:pw},  we obtain the limit shape of dimer coverings on a rail-yard graph whose left boundary condition is given by a piecewise partition with densities either 1 or 0.  We then discuss the cases when the frozen boundary have multiple connected components, each of which is a cloud curve (an algebraic curve of a certain type). The main theorems proved in this section are Theorems \ref{t412}, \ref{t419} and \ref{t421}. In Section \ref{sect:AA}, we review some technical results.

\section{Background}\label{sect:bk}

In this section, we define the rail-yard graph, the perfect matching and the height function, and review related technical facts.

\subsection{Weighted rail-yard graphs}

Let $l,r\in\ZZ$ such that $l\leq r$. Let
\begin{eqnarray*}
[l..r]:=[l,r]\cap\ZZ,
\end{eqnarray*}
i.e., $[l..r]$ is the set of integers between $l$ and $r$. For a positive integer $m$, we use
\begin{eqnarray*}
[m]:=\{1,2,\ldots,m\}.
\end{eqnarray*}
Consider two binary sequences indexed by integers in $[l..r]$ 
\begin{itemize}
\item the $LR$ sequence $\underline{a}=\{a_l,a_{l+1},\ldots,a_r\}\in\{L,R\}^{[l..r]}$;
\item the sign sequence $\underline{b}=(b_l,b_{l+1},\ldots,b_r)\in\{+,-\}^{[l..r]}$.
\end{itemize}
The rail-yard graph $RYG(l,r,\underline{a},\underline{b})$ with respect to integers $l$ and $r$, the $LR$ sequence $\underline{a}$ and the sign sequence $\underline{b}$, is the bipartite graph with vertex set $[2l-1..2r+1]\times \left\{\ZZ+\frac{1}{2}\right\}$. A vertex is called even (resp.\ odd) if its abscissa is an even (resp.\ odd) integer. Each even vertex $(2m,y)$, $m\in[l..r]$ is incident to 3 edges, two horizontal edges joining it to the odd vertices $(2m-1,y)$ and $(2m+1,y)$ and one diagonal edge joining it to
\begin{itemize}
\item the odd vertex $(2m-1,y+1)$ if $(a_m,b_m)=(L,+)$;
\item the odd vertex $(2m-1,y-1)$ if $(a_m,b_m)=(L,-)$;
\item the odd vertex $(2m+1,y+1)$ if $(a_m,b_m)=(R,+)$;
\item the odd vertex $(2m+1,y-1)$ if $(a_m,b_m)=(R,-)$.
\end{itemize} 

See Figure \ref{fig:rye} for an example of a rail-yard graph.
\begin{figure}
\includegraphics[width=.8\textwidth]{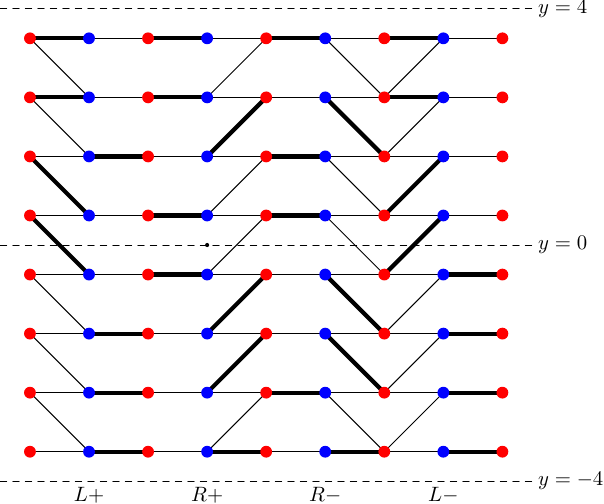}
\caption{A rail yard graph with LR sequence $\underline{a}=\{L,R,R,L\}$and sign sequence $\underline{b}=\{+,+,-,-\}$. Odd vertices are represented by red points, and even vertices are represented by blue points. Dark lines represent a pure dimer covering. Assume that above the horizontal line $y=4$, only horizontal edges with an odd vertex on the left are present in the dimer configuration; and below the horizontal line $y=-4$, only horizontal edges with an even vertex on the left are present in the dimer configuration. }\label{fig:rye}
\end{figure}

The left boundary (resp.\ right boundary) of $RYG(l,r,\underline{a},\underline{b})$ consists of all odd vertices with abscissa $2l-1$ (resp.\ $2r+1$). Vertices which do not belong to the boundaries are called inner. A face of $RYG(l,r,\underline{a},\underline{b})$ is called an inner face if it contains only inner vertices.

We assign edge weights to a rail yard graph $RYG(l,r,\underline{a},\underline{b})$ as follows:
\begin{itemize}
    \item all the horizontal edges have weight 1; and
    \item each diagonal edge adjacent to a vertex with abscissa $2i$ has weight $x_i$.
\end{itemize}

\subsection{Dimer coverings and pure dimer coverings}

\begin{definition}
A dimer covering is a subset of edges of $RYG(l,r,\underline{a},\underline{b})$ such that
\begin{enumerate}
\item each inner vertex of $RYG(l,r,\underline{a},\underline{b})$ is incident to exactly one edge in the subset;
\item each left boundary vertex or right boundary vertex is incident to at most one edge in the subset;
\item only a finite number of diagonal edges are present in the subset.
\end{enumerate}
Note that the set of admissible dimer coverings defined as in \cite{bbccr} is a proper subset of the set of all dimer coverings defined here.

A pure dimer covering of $RYG(l,r,\underline{a},\underline{b})$ is a dimer covering of $RYG(l,r,\underline{a},\underline{b})$ satisfying the following two additional conditions
\begin{itemize}
\item each left boundary vertex $(2l-1,y)$ is incident to exactly one edge (resp.\ no edges) in the subset if $y>0$ (resp.\ $y<0$).
\item each right boundary vertex $(2r+1,y)$ is incident to exactly one edge (resp.\ no edges) in the subset if $y<0$ (resp.\ $y>0$).
\end{itemize}
\end{definition}

See Figure \ref{fig:rye} for an example of pure dimer coverings on a rail yard graph.

For a dimer covering $M$ on the rail-yard graph $RYG(l,r,\underline{a},\underline{b})$, define the associated height function $h_{M}$ on faces of $RYG(l,r,\underline{a},\underline{b})$ as follows. We first define a preliminary height function $\overline{h}_M$ on faces of $RYG(l,r,\underline{a},\underline{b})$. Note that there exists a positive integer $N>0$, such that when $y<-N$, only horizontal edges with even vertices on the left are present.  Fix a face $f_0$ of $RYG(l,r,\underline{a},\underline{b})$ such that the midpoint of $f_0$ is on the horizontal line $y=-N$, and define
\begin{eqnarray*}
\overline{h}_M(f_0)=0.
\end{eqnarray*}
For any two adjacent faces $f_1$ and $f_2$ sharing at least one edge, 
\begin{itemize}
\item If moving from $f_1$ to $f_2$ crosses a present (resp.\ absent) horizontal edge in $M$ with odd vertex on the left, then $\ol{h}_M(f_2)-\ol{h}_M(f_1)=1$ (resp.\ $\ol{h}_M(f_2)-\ol{h}_M(f_1)=-1$).
\item If moving from $f_1$ to $f_2$ crosses a present (resp.\ absent) diagonal edge in $M$ with odd vertex on the left, then $\ol{h}_M(f_2)-\ol{h}_M(f_1)=2$ (resp.\ $\ol{h}_M(f_2)-\ol{h}_M(f_1)=0$).
\end{itemize}

Let $\ol{h}_0$ be the preliminary height function associated to the dimer configuration satisfying 
\begin{itemize}
\item no diagonal edge is present; and
\item each present edge is horizontal with an even vertex on the left.
\end{itemize}
The height function $h_{M}$ associated to $M$ is then defined by 
\begin{eqnarray}
h_M=\ol{h}_M-\ol{h}_0.\label{dhm}
\end{eqnarray}

Let $m\in[l..r]$. Let $x=2m-\frac{1}{2}$ be a vertical line such that all the horizontal edges and diagonal edges of $RYG(l,r,\underline{a},\underline{b})$ crossed by $x=2m-\frac{1}{2}$ have odd vertices on the left. Then for each point $\left(2m-\frac{1}{2},y\right)$ in a face of $RYG(l,r,\underline{a},\underline{b})$, we have
\begin{eqnarray}
h_{M}\left(2m-\frac{1}{2},y\right)=2\left[N_{h,M}^{-}\left(2m-\frac{1}{2},y\right)+N_{d,M}^{-}\left(2m-\frac{1}{2},y\right)\right];\label{hm1}
\end{eqnarray}
where $N_{h,M}^{-}\left(2m-\frac{1}{2},y\right)$ is the total number of present horizontal edges in $M$ crossed by $x=2m-\frac{1}{2}$ below $y$, and $N_{d,M}^{-}\left(2m-\frac{1}{2},y\right)$ is the total number of present diagonal edges in $M$ crossed by $x=2m-\frac{1}{2}$ below $y$. From the definition of a pure dimer covering we can see that both $N_{h,M}^{-}\left(2m-\frac{1}{2},y\right)$ and $N_{d,M}^{-}\left(2m-\frac{1}{2},y\right)$ are finite for each finite $y$.

Note also that $x=2m+\frac{1}{2}$ is a vertical line such that all the horizontal edges and diagonal edges of $RYG(l,r,\underline{a},\underline{b})$ crossed by $x=2m+\frac{1}{2}$ have even vertices on the left. Then for each point $\left(2m+\frac{1}{2},y\right)$ in a face of $RYG(l,r,\underline{a},\underline{b})$, we have
\begin{eqnarray}
h_{M}\left(2m+\frac{1}{2},y\right)=2\left[J_{h,M}^{-}\left(2m+\frac{1}{2},y\right)-N_{d,M}^{-}\left(2m+\frac{1}{2},y\right)\right];\label{hm2}
\end{eqnarray}
where $J_{h,M}^{-}\left(2m+\frac{1}{2},y\right)$ is the total number of absent horizontal edges in $M$ crossed by $x=2m+\frac{1}{2}$ below $y$, and $N_{d,M}^{-}\left(2m+\frac{1}{2},y\right)$ is the total number of present diagonal edges in $M$ crossed by $x=2m+\frac{1}{2}$ below $y$. From the definition of a pure dimer covering we can also see that both $J_{h,M}^{-}\left(2m+\frac{1}{2},y\right)$ and $N_{d,M}^{-}\left(2m+\frac{1}{2},y\right)$ are finite for each finite $y$.

\subsection{Partitions}

Let $N$ be a non-negative integer. A length-$N$ partition is a non-increasing sequence $\lambda=(\lambda_i)_{i\geq 1}^{N}$ of non-negative integers. Let $\mathbb{Y}_N$ be the set of all the length-$N$ partitions. The size of a partition $\lambda\in \YY_N$ is defined by
\begin{eqnarray*}
|\lambda|=\sum_{i=1}^{N}\lambda_i.
\end{eqnarray*}
We denote the length of a partition by $l(\lambda)$. In particular, if $\lambda\in \YY_N$, then $l(\lambda)=N$.
Two partitions $\lambda\in \YY_N$ and $\mu\in\YY_{N-1}$ are called interlaced, and written by $\lambda\succ\mu$ or $\mu\prec \lambda$ if
\begin{eqnarray*}
\lambda_1\geq \mu_1\geq\lambda_2\geq \mu_2\geq \lambda_3\geq \ldots\geq \mu_{N-1}\geq \lambda_N.
\end{eqnarray*}

A length-$\infty$ partition is a non-increasing sequence $\lambda=(\lambda_i)_{i\geq 0}^{N}$ of non-negative integers which vanishes eventually. Each length $N$-partition can be naturally extended to an length-$\infty$ partition by adding infinitely many $0$'s at the end of the sequence.

When representing partitions by Young diagrams, this means $\lambda/\mu$ is a horizontal strip. The conjugate partition $\lambda'$ of $\lambda$ is a partition whose Young diagram $Y_{\lambda'}$ is the image of the Young diagram $Y_{\lambda}$ of $\lambda$ by the reflection along the main diagonal. More precisely
\begin{eqnarray*}
\lambda_i':=\left|\{j\geq 0: \lambda_j\geq i\}\right|,\qquad \forall i\geq 1.
\end{eqnarray*}

The skew Schur functions are defined in Section I.5 of \cite{IGM15}.

\begin{definition}\label{dss}Let $\lambda$, $\mu$ be partitions of finite length. Define the skew Schur functions as follows
\begin{eqnarray*}
&&s_{\lambda/\mu}=\det\left(h_{\lambda_i-\mu_j-i+j}\right)_{i,j=1}^{l(\lambda)}\\
\end{eqnarray*} 
Here for each $r\geq 0$, $h_r$ is the $r$th complete symmetric function defined by the sum of all monomials of total degree $r$ in the variables $x_1,x_2,\ldots$. More precisely,
\begin{eqnarray*}
h_r=\sum_{1\leq i_1\leq i_2\leq \ldots\leq i_r} x_{i_1}x_{i_2}\cdots x_{i_r}
\end{eqnarray*}
If $r<0$, $h_r=0$.


Define the Schur function as follows
\begin{eqnarray*}
s_{\lambda}=s_{\lambda/\emptyset}.
\end{eqnarray*}
\end{definition}

For a dimer covering $M$ of $RYG(l,r,\underline{a},\underline{b})$, we associate a particle-hole configuration to each odd vertex of $RYG(l,r,\underline{a},\underline{b})$ as follows: let $m\in[l..(r+1)]$ and $k\in\ZZ$: if the odd endpoint $\left(2m-1,k+\frac{1}{2}\right)$ is incident to a present edge in $M$ on its right (resp.\ left), then associate a hole (resp.\ particle) to the odd endpoint $\left(2m-1,k+\frac{1}{2}\right)$. When $M$ is a pure dimer covering, it is not hard to check that there exists $N>0$, such that when $y>N$, only holes exist and when $y<-N$, only particles exist.

We associate a partition $\lambda^{(M,m)}$ to the column indexed by $m$ of particle-hole configurations, which corresponds to a pure dimer covering $M$ adjacent to odd vertices with abscissa $(2m-1)$ as follows. Assume
\begin{eqnarray*}
\lambda^{(M,m)}=(\lambda^{(M,m)}_1,\lambda^{(M,m)}_2,\ldots),
\end{eqnarray*}
Then for $i\geq 1$, $\lambda^{(M,m)}_i$ is the total number of holes in $M$ along the vertical line $x=2m-1$ below the $i$th highest particles. Let $l(\lambda^{(M,m)})$ be the total number of nonzero parts in the partition $\lambda^{(M,m)}$.

\begin{figure}
\includegraphics[width=.6\textwidth]{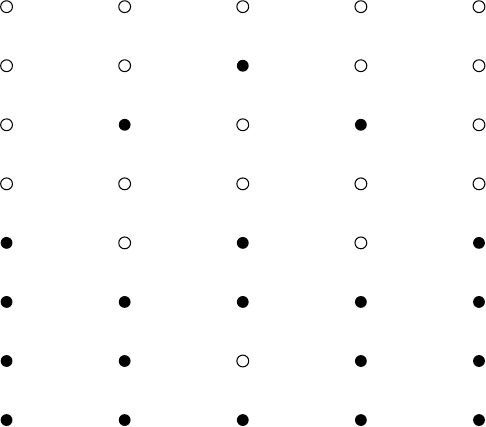}
\caption{Particle-hole configuration corresponding to the dimer covering in Figure \ref{fig:rye}. Particles are represented by black dots, while holes are represented by circles. Each particle/hole represents the configuration of an odd vertex (represented by red points)  in the same location of Figure \ref{fig:rye}. The corresponding sequence of partitions (from the left to the right) is given by $\emptyset\prec(2,0,\ldots)\prec' (3,1,1,\ldots)\succ'(2,0,\ldots)\succ \emptyset$.}\label{fig:rye1}
\end{figure}

We define the charge $c^{(M,m)}$ on column $(2m-1)$ for the configuration $M$ as follows:
\begin{eqnarray}
c^{(M,m)}&=&\mathrm{number\ of\ particles\ on\ column\ }(2m-1)\ \mathrm{in\ the\ upper\ half\ plane}\label{dcg}\\
&&-\mathrm{number\ of\ holes\ on\ column\ }(2m-1)\ \mathrm{in\ the\ lower\ half\ plane}\notag
\end{eqnarray}

The weight of a dimer covering $M$ of $RYG(l,r,\underline{a},\underline{b})$ is defined as follows
\begin{eqnarray*}
w(M):=\prod_{i=l}^{r}x_i^{d_i(M)},
\end{eqnarray*}
where $d_i(M)$ is the total number of present diagonal edges of $M$ incident to an even vertex with abscissa $2i$. 

Let $\lambda^{(l)},\lambda^{(r+1)}$ be two partitions. The partition function $Z_{\lambda^{(l)},\lambda^{(r+1)}}(G,\underline{x})$ of dimer coverings on $RYG(l,r,\underline{a},\underline{b})$ whose configurations on the left (resp.\ right) boundary correspond to partition $\lambda^{(l)}$ (resp.\ $\lambda^{(r+1)}$) is the sum of weights of all such dimer coverings on the graph. Given the left and right boundary conditions $\lambda^{(l)}$ and $\lambda^{(r+1)}$, respectively, the probability of a dimer covering $M$ is then defined by
\begin{eqnarray}
\mathrm{Pr}(M|\lambda^{(l)},\lambda^{(r+1)}):=\frac{w(M)}{Z_{\lambda^{(l)},\lambda^{(r+1)}}(G,\underline{x})}.\label{ppd}
\end{eqnarray}
Note that pure dimer coverings have left and right boundary conditions given by
\begin{eqnarray}
\lambda^{(l)}=\lambda^{(r+1)}=\emptyset;\label{pbc}
\end{eqnarray}
respectively.

The bosonic Fock space $\mathcal{B}$ is the infinite dimensional Hilbert space spanned by the orthonormal basis vectors $|\lambda\rangle$, where $\lambda$ runs over all the partitions. Let $\langle \lambda |$ denote the dual basis vector. Let $x$ be a formal or a complex variable. Introduce the operators $\Gamma_{L+}(x)$, $\Gamma_{L-}(x)$, $\Gamma_{R+}(x)$, $\Gamma_{R-}(x)$ from $\mathcal{B}$  to $\mathcal{B}$ as follows
\begin{eqnarray*}
\Gamma_{L+}(x)|\lambda\rangle=\sum_{\mu\prec \lambda}x^{|\lambda|-|\mu|}|\mu\rangle;\qquad \Gamma_{R+}(x)|\lambda\rangle=\sum_{\mu'\prec \lambda'}x^{|\lambda|-|\mu|}|\mu\rangle;\\
\Gamma_{L-}(x)|\lambda\rangle=\sum_{\mu\succ \lambda}x^{|\mu|-|\lambda|}|\mu\rangle;\qquad \Gamma_{R-}(x)|\lambda\rangle=\sum_{\mu'\succ \lambda'}x^{|\mu|-|\lambda|}|\mu\rangle;
\end{eqnarray*}
Such operators were first studied in \cite{oko01} for random partitions.

\begin{lemma}\label{l12}Let $a_1,a_2\in \{L,R\}$. We have the following commutation relations for the operators $\Gamma_{a_1,\pm}$, $\Gamma_{a_2,\pm}$.
\begin{eqnarray*}
\Gamma_{a_1,+}(x_1)\Gamma_{a_2,-}(x_2)=\begin{cases}\frac{\Gamma_{a_2,-}(x_2)\Gamma_{a_1,+}(x_1)}{1-x_1x_2}&\mathrm{if}\ a_1=a_2\\(1+x_1x_2)\Gamma_{a_2,-}(x_2)\Gamma_{a_1,+}(x_1)&\mathrm{if}\ a_1\neq a_2.\end{cases}.
\end{eqnarray*}
Moreover,
\begin{eqnarray*}
\Gamma_{a_1,b}(x_1)\Gamma_{a_2,b}(x_2)=\Gamma_{a_2,b}(x_2)\Gamma_{a_1,b}(x_1);
\end{eqnarray*}
for all $a_1,a_2\in\{L,R\}$ and $b\in\{+,-\}$.
\end{lemma}

\begin{proof}See Proposition 7 of \cite{bbccr}; see also \cite{you10,bbb14}.
\end{proof}

Given the definitions of the operators $\Gamma_{a,b}(x)$ with $a\in\{L,R\}$, $b\in\{+,-\}$, it is straightforward to check the following lemma.

\begin{lemma}\label{l13}The partition function of dimer coverings on a rail yard graph $G=RYG(l,r,\underline{a},\underline{b})$ with left and right boundary conditions given by $\lambda^{(l)},\lambda^{(r+1)}$, respectively,  is 
\begin{eqnarray}
Z_{\lambda^{(l)},\lambda^{(r+1)}}(G;\underline{x})=\langle\lambda^{(l)}| \Gamma_{a_lb_l}(x_l)\Gamma_{a_{l+1}b_{l+1}}(x_{l+1})\cdots \Gamma_{a_rb_r}(x_r)|\lambda^{(r+1)} \rangle\label{bf}
\end{eqnarray}
\end{lemma}

\begin{corollary}\label{lc13}
\begin{enumerate}
\item Assume for all
$i\in[l,r]$,
\begin{eqnarray}
(a_i,b_i)\neq (R,-).\label{c151}
\end{eqnarray}
Let $g_l$ be the number of nonzero entries in $\lambda^{(l)}$, then
\begin{eqnarray}
g_l\leq \left|\{i\in[l..r]:a_i=L,b_i=-\}\right|.\label{glu}
\end{eqnarray}
The partition function of dimer coverings on a rail yard graph $G=RYG(l,r,\underline{a},\underline{b})$ with left and right boundary conditions given by $\lambda^{(l)},\emptyset$, respectively,  can be computed as follows:
\begin{eqnarray}
Z_{\lambda^{(l)},\emptyset}(G;\underline{x})=s_{\lambda^{(l)}}\left(\underline{x}^{(L,-)}\right)\prod_{l\leq i<j\leq r;b_i=+,b_j=-}z_{i,j}\label{fp}
\end{eqnarray}
where
\begin{eqnarray}
z_{ij}=\begin{cases}1+x_ix_j&\mathrm{if}\ a_i\neq a_j\\\frac{1}{1-x_ix_j}&\mathrm{if}\ a_i=a_j\end{cases}\label{dzij}
\end{eqnarray}
and
\begin{eqnarray*}
\underline{x}^{(L,-)}:=\{x_i: (a_i,b_i)=(L,-)\}.
\end{eqnarray*}
\item Assume for all
$i\in[l,r]$,
\begin{eqnarray}
(a_i,b_i)\neq (L,-).\label{c152}
\end{eqnarray}
Let $g_l'$ be the number of nonzero entries in $\lambda^{(l)}$, then
\begin{eqnarray*}
g_l'\leq \left|\{i\in[l..r]:a_i=R,b_i=-\}\right|.
\end{eqnarray*}
The partition function of dimer coverings on a rail yard graph $G=RYG(l,r,\underline{a},\underline{b})$ with left and right boundary conditions given by $\lambda^{(l)},\emptyset$, respectively,  can be computed as follows:
\begin{eqnarray}
Z_{\lambda^{(l)},\emptyset}(G;\underline{x})=s_{[\lambda^{(l)}]'}\left(\underline{x}^{(R,-)}\right)\prod_{l\leq i<j\leq r;b_i=+,b_j=-}z_{i,j}\label{fp}
\end{eqnarray}
where
\begin{eqnarray*}
\underline{x}^{(R,-)}:=\{x_i: (a_i,b_i)=(R,-)\}.
\end{eqnarray*}
\end{enumerate}
\end{corollary}

\begin{proof}We prove part (1) here; part (2) can be proved using the same technique.
Let
\begin{eqnarray*}
Z_0:=\prod_{l\leq i<j\leq r;b_i=+,b_j=-}z_{i,j}
\end{eqnarray*}
Let $t_{l},t_{l+1},\ldots,t_{r}$ be a permutation of $l,l+1,\ldots,r$ such that there exists $m\in[l-1..r]$ satisfying
\begin{eqnarray*}
b_{t_i}=\begin{cases}-&\mathrm{if}\ i\leq m\\ +&\mathrm{if}\ i\geq m\end{cases};
\end{eqnarray*}
for all $i\in[l..r]$, where
\begin{eqnarray*}
m-l+1=|\{i:a_i=L,b_i=-\}|.
\end{eqnarray*}
Then by (\ref{bf}), we obtain

\begin{eqnarray}
Z_{\lambda^{(l)},\emptyset}(G;\underline{x})&=&\langle\lambda^{(l)}| \Gamma_{a_{t_l}b_{t_l}}(x_{t_l})\Gamma_{a_{t_{l+1}}b_{t_{l+1}}}(x_{t_{l+1}})\cdots \Gamma_{a_{t_r}b_{t_r}}(x_{t_r})|\emptyset \rangle Z_0
\\
&=&\sum_{\mu}\langle\lambda^{(l)}| \Gamma_{a_{t_l}b_{t_l}}(x_{t_l})\ldots\Gamma_{a_{t_m},b_{t_m}}(x_{t_m})
|\mu\rangle\notag\\
&&\times\langle \mu|
\Gamma_{a_{t_{m+1}},b_{t_{m+1}}}(x_{t_{m+1}})\cdots \Gamma_{a_{t_r},b_{t_r}}(x_{t_r})|\emptyset \rangle
Z_0\notag
\end{eqnarray}

Note that 
\begin{eqnarray*}
\langle \mu|
\Gamma_{a_{t_{m+1}},b_{t_{m+1}}}(x_{t_{m+1}})\cdots \Gamma_{a_{t_r},b_{t_r}}(x_{t_r})|\emptyset \rangle
=\begin{cases}1&\mathrm{if}\ \mu=\emptyset\\ 0&\mathrm{otherwise}\end{cases}
\end{eqnarray*}
Hence when (\ref{c151}) holds, we have
\begin{eqnarray*}
Z_{\lambda^{(l)},\emptyset}(G;\underline{x})
&=&\langle\lambda^{(l)}| \Gamma_{a_{t_l}b_{t_l}}(x_{t_l})\ldots\Gamma_{a_{t_m},b_{t_m}}(x_{t_m})
|\emptyset\rangle Z_0
\end{eqnarray*}
In order that 
\begin{eqnarray*}
\langle\lambda^{(l)}| \Gamma_{a_{t_l}b_{t_l}}(x_{t_l})\ldots\Gamma_{a_{t_m},b_{t_m}}(x_{t_m})
|\emptyset\rangle\neq 0,
\end{eqnarray*}
there exists a sequence of partitions
\begin{eqnarray*}
\mu^{(0)}:=\lambda^{(l)}\succ \mu^{(1)}\succ \mu^{(2)}\succ\ldots\succ\mu^{(m-l)}\succ\emptyset
\end{eqnarray*}
The total number of partitions in the sequence is $m-l+1$. Since $\mu^{(i)}/\mu^{(i+1)}$ is a horizontal strip, the number of nonzero parts in $\mu^{(i)}$ is at most one more than that of $\mu^{(i+1)}$, then we obtain (\ref{glu}). Hence we may construct a partition of length $m-l+1$ from $\lambda^{(l)}$ by adding zeros in the end if necessary, and still use $\lambda^{(l)}$ to denote this new partition. Then we obtain
\begin{eqnarray*}
Z_{\lambda^{(l)},\emptyset}(G;\underline{x})
&=& Z_0 s_{\lambda^{(l)}}\left(\underline{x}^{(L,-)}\right).
\end{eqnarray*}
Then part (1) of the lemma follows.
\end{proof}

\noindent{\textbf{Remark.}} Corollary \ref{lc13} generalizes Proposition 2.5 in \cite{BL17}. In particular, if (\ref{c151}) holds, and if whenever $b_i\neq b_j$ we have $a_i\neq a_j$, we obtain exactly Proposition 2.5 in \cite{BL17}.

\bigskip
\noindent{\textbf{Remark.}} The partition function $Z(G;\underline{x})$ is always well-defined as a power series in $\underline{x}$. When we consider the edge weights $x_i$'s to be positive numbers, to make sure the convergence of the power series representing the partition function, we need to assume that
\begin{assumption}\label{ap14}For any $i,j\in[l..r]$, $i<j$, $a_i=a_j$ and $b_i=+$, $b_j=-$ we have
\begin{eqnarray*}
x_ix_j<1.
\end{eqnarray*}
\end{assumption}

\begin{definition}
  \label{df21}
  Let $\mathbf{X}=(x_1,\ldots,x_N)\in\CC^{N}$. Let $\rho$ be a
  probability measure on $\YY_N$. The \emph{Schur generating function}
  $\mathcal{S}_{\rho,\mathbf{X}}(u_1,\ldots,u_N)$ with respect to parameters
  $\mathbf{X}$ is the symmetric Laurent series in $(u_1,\ldots,u_N)$ given by
  \begin{equation*}
    \mathcal{S}_{\rho,\mathbf{X}}(u_1,\ldots,u_N)=
    \sum_{\lambda\in \YY} \rho(\lambda)
    \frac{s_{\lambda}(u_1,\ldots,u_N)}{s_{\lambda}(x_1,\ldots,x_N)}.
  \end{equation*}
\end{definition}

\begin{lemma}
  \label{lm212}
  Let $t\in[l..r]$. Assume (\ref{c151}) holds. Let
\begin{eqnarray*}
\underline{x}^{(L,-,> t)}&=&\{x_i:i\in[t+1...r],a_i=L,b_i=-\}.\\
\underline{x}^{(L,-,\leq t)}&=&\{x_i:i\in[l...t],a_i=L,b_i=-\}
\end{eqnarray*}
Let
\begin{eqnarray*}
\underline{u}^{(L,-,> t)}&=&\{u_i:i\in[t+1...r],a_i=L,b_i=-\}.
\end{eqnarray*}
Let $\rho^t$ be the probability measure for the partitions corresponding to dimer configurations adjacent to the column of odd vertices labeled by $2t-1$ in $RYG(l,r,\underline{a},\underline{b})$, conditional on the left and right boundary condition $\lambda^{(l)}$ and $\emptyset$, respectively.

For $i\in[l..r]$, let
  \begin{eqnarray*}
  w_i=\begin{cases}u_i&\mathrm{if}\ i\in[t+1..r],a_i=L,b_i=-\\ x_i&\mathrm{otherwise}\end{cases}
  \end{eqnarray*}
  and 
  \begin{eqnarray*}
\xi_{ij}=\begin{cases}1+w_iw_j&\mathrm{if}\ a_i\neq a_j\\\frac{1}{1-w_iw_j}&\mathrm{if}\ a_i=a_j\end{cases}.
\end{eqnarray*}
Let
\begin{eqnarray*}
\underline{w}^{(L,-)}=\{w_i:i\in[l..r],a_i=L,b_i=-\}.
\end{eqnarray*}
  Then the generating Schur function $\mathcal{S}_{\rho^t,X^{(N-t)}}$ is given by:
\begin{equation*}
\mathcal{S}_{\rho^t,\underline{x}^{(L,-,> t)}}(\underline{u}^{(L,-,> t)})=
\frac{s_{\lambda^{(l)}}\left(\underline{u}^{(L,-,> t)},x^{(L,-,\leq t)}\right)}{s_{\lambda^{(l)}}(\underline{x}^{(L,-)})}
\prod_{i\in[l..t],b_i=+;}
\prod_{j\in[t+1..r],a_j=L,b_j=-.}\frac{\xi_{ij}}{z_{ij}}.
\end{equation*}
where $z_{ij}$ is defined as in (\ref{dzij}).
\end{lemma}

\begin{proof}By Corollary \ref{lc13}(1), we obtain that for $\lambda\in \YY$
\begin{eqnarray*}
\rho^t(\lambda)=\frac{\langle\lambda^{(l)}| \Gamma_{a_lb_l}(x_l)\cdots \Gamma_{a_tb_t}(x_t)|\lambda \rangle
s_{\lambda}\left(\underline{x}^{(L,-,> t)}\right)\prod_{t+1\leq i<j\leq r;b_i=+,b_j=-}z_{i,j}}{s_{\lambda^{(l)}}\left(\underline{x}^{(L,-)}\right)\prod_{l\leq i<j\leq r;b_i=+,b_j=-}z_{i,j}}
\end{eqnarray*}
By Definition \ref{df21}, we have
\begin{eqnarray}
    &&\mathcal{S}_{\rho^t,\underline{x}^{(L,-,> t)}}(\underline{u}^{(L,-,> t)})\notag=
    \sum_{\lambda\in \YY} \rho^t(\lambda)
    \frac{s_{\lambda}(\underline{u}^{(L,-,> t)})}{s_{\lambda}(\underline{x}^{(L,-,>t)})}\notag\\
    &=&\sum_{\lambda\in \YY} 
    \frac{\langle\lambda^{(l)}| \Gamma_{a_lb_l}(x_l)\cdots \Gamma_{a_tb_t}(x_t)|\lambda \rangle
s_{\lambda}(\underline{u}^{(L,-,> t)})\prod_{t+1\leq i<j\leq r;b_i=+,b_j=-}z_{i,j}}{s_{\lambda^{(l)}}\left(\underline{x}^{(L,-)}\right)\prod_{l\leq i<j\leq r;b_i=+,b_j=-}z_{i,j}}\label{sgp11}
  \end{eqnarray}
  Note that
  \begin{eqnarray}
  &&\sum_{\lambda\in \YY} 
    \langle\lambda^{(l)}| \Gamma_{a_lb_l}(x_l)\cdots \Gamma_{a_tb_t}(x_t)|\lambda \rangle
s_{\lambda}(\underline{u}^{(L,-,> t)})\prod_{t+1\leq i<j\leq r;b_i=+,b_j=-}z_{i,j}\notag\\
&=&\langle\lambda^{(l)}| \Gamma_{a_lb_l}(w_l)\cdots \Gamma_{a_rb_r}(w_t)|\emptyset \rangle
\frac{\prod_{t+1\leq i<j\leq r;b_i=+,b_j=-}z_{i,j}}{\prod_{t+1\leq i<j\leq r;b_i=+,b_j=-}\xi_{i,j}}\notag\\
&=&s_{\lambda^{(l)}}(\underline{w}^{(L,-)})
\prod_{l\leq i<j\leq r;b_i=+,b_j=-}\xi_{i,j}
\frac{\prod_{t+1\leq i<j\leq r;b_i=+,b_j=-}z_{i,j}}{\prod_{t+1\leq i<j\leq r;b_i=+,b_j=-}\xi_{i,j}}\label{sgp2}
  \end{eqnarray}
  Then the lemma follows from plugging (\ref{sgp2}) to (\ref{sgp11}).
\end{proof}

\begin{assumption}\label{ap41}Let $\epsilon>0$ be a small positive parameter. 
\begin{enumerate}
\item Let $l^{(\epsilon)}<r^{(\epsilon)}$ be the  integers representing the left and right boundary of the Rail-yard graph depending on $\epsilon$ such that
\begin{eqnarray*}
\lim_{\epsilon\rightarrow 0}\epsilon l^{(\epsilon)}=l^{(0)}<r^{(0)}=\lim_{\epsilon\rightarrow 0}\epsilon r^{(\epsilon)}
\end{eqnarray*}
so that the scaling limit of the sequence of rail-yard graphs $\{\epsilon RYG(l^{(\epsilon)},l^{(\epsilon)},\underline{a}^{(\epsilon)},\underline{b}^{(\epsilon)})\}_{\epsilon>0}$, as $\epsilon\rightarrow 0$, has left boundary given by $x=2l^{(0)}$ and right boundary given by $x=2r^{(0)}$.
\item Assume for each $\epsilon>0$, there exist
\begin{eqnarray*}
l^{(\epsilon)}=v_0^{(\epsilon)}<v_1^{(\epsilon)}<\ldots<v_m^{(\epsilon)}=r^{(\epsilon)};
\end{eqnarray*}
such that for each $p\in[m]$, the sequences $\underline{a}^{(\epsilon)}$ and 
$\underline{b}^{(\epsilon)}$ with indices in $\left(v_{p-1}^{(\epsilon)},v_{p}^{(\epsilon)}\right)$ are $n_p$-periodic. More precisely,
for any $v_{p-1}^{(\epsilon)}< i<j \leq v_p^{(\epsilon)}$, if 
\begin{eqnarray*}
[(i-j)\mod n_p]=0,
\end{eqnarray*}
then
\begin{eqnarray*}
a_i^{(\epsilon)}=a_j^{(\epsilon)}=a_{[i\mod n_p],\ p};\\
b_i^{(\epsilon)}=b_j^{(\epsilon)}=b_{[j\mod n_p],\ p}; 
\end{eqnarray*}
and
\begin{eqnarray*}
x_i=x_j=x_{[i\mod {n_p}]}^{(p)}.
\end{eqnarray*}
where $n_p$ is a positive integer depending on $p$; $a_{[i\mod n_p],\ p}\in\{L,R\}$, $b_{[j\mod n_p],\ p}\in\{+,-\}$, $x_{[i\mod {n_p}]}^{(p)}\in[0,\infty)$ are independent of $\epsilon$. Here for simplicity, we make the range of $[i\mod n_p]$ to be $[n_p]$, i.e. the integers from $1$ to $n_p$.
\item There exist
\begin{eqnarray*}
l^{(0)}=V_0<V_1<\ldots<V_m=r^{(0)},
\end{eqnarray*}
such that
\begin{eqnarray*}
\lim_{\epsilon\rightarrow 0}\epsilon v_p^{(\epsilon)}=V_p,\ \forall p\in[m].
\end{eqnarray*}
\item For each $p\in [m]$, $j\in[n_p]$, $\mathbf{a}\in\{L,R\}$, $\mathbf{b}\in\{+,-\}$
\begin{eqnarray*}
\lim_{\epsilon\rightarrow 0}\frac{\left|\left\{u\in [v_{p-1}^{(\epsilon)}+1..v_p^{(\epsilon)}]\cap \{n\ZZ+j\}: a_u=\mathbf{a},b_u=\mathbf{b}\right\}\right|}{v_p^{(\epsilon)}-v_{p-1}^{(\epsilon)}}&=&\zeta_{\mathbf{a},\mathbf{b},j,p}
\end{eqnarray*}
\end{enumerate}
\end{assumption}

\subsection{Differential operator}
Let 
\begin{equation*}
V(u_1,\ldots, u_N)=\prod_{1\leq i<j\leq N}(u_i-u_j)
\end{equation*}
be the Vandermonde determinant with respect to variables $u_1,\ldots,u_N$.
Introduce the family $(\mathcal{D}_k)$ of differential operators acting on
symmetric functions $f$ with variables $u_1,\ldots, u_N$ as follows:
\begin{equation}
\mathcal{D}_k f=\frac{1}{V}\left(\sum_{i=1}^{N}\left(u_i\frac{\partial}{\partial
u_i}\right)^k\right)(V\cdot f)\label{dk}.
\end{equation}

\subsection{Counting measure}Let $\lambda$ be a length-$N$ partition. We define the counting measure $m(\lambda)$ as a probability measure on $\RR$ as follows:
\begin{eqnarray*}
m(\lambda)=\frac{1}{N}\sum_{i=1}^{N}\delta\left(\frac{\lambda_i+N-i}{N}\right).
\end{eqnarray*}
If $\lambda$ is random, then we can define the corresponding random counting measure.

Let $\rho$ be a probability measure on $\YY_N$ with Schur generating function given by $\mathcal{S}_{\rho,\mathbf{X}}(u_1,\ldots,u_N)$. Let $\mathbf{m}_{\rho}$ be the random counting measure for a random partition $\lambda\in \YY_N$ with distribution $\rho$.
Then explicit computations show that
\begin{eqnarray}
\EE\left(\int_{\RR}x^k\mathbf{m}_{\rho}(dx)\right)^m=
    \frac{1}{N^{m(k+1)}}
    (\mathcal{D}_k)^m\mathcal{S}_{\rho,\mathbf{X}}(u_1,\ldots,u_N)|_{(u_1,\ldots,u_N)=(x_1,\ldots,x_N)}.\label{emt}
\end{eqnarray}

\section{Limit shape for staircase boundary condition}\label{sect:sc}
In this section, we obtain the limit shape of dimer coverings on a rail-yard graph whose left boundary condition is given by a staircase partition whose counting measure has constant density $\frac{1}{M}$. The main theorem proved in this section is Theorem \ref{t33}. We then discuss the frozen boundary in the special case when $M=1$ or $M=2$.

To make precise the staircase boundary condition, 
 we make the following assumption about the left boundary condition $\lambda^{(l)}$.
 \begin{assumption}\label{ap110}
 Assume the
 configuration on left is given by the following very specific partition:
 \begin{equation}
 \lambda^{(l)}=((M-1)(N-1),(M-1)(N-2),\ldots,(M-1),0),\label{scb}
 \end{equation}
 where $M\geq 1$ is a positive integer and 
 \begin{eqnarray}
N:= N_{L,-}=|\{i\in[l..r]:a_i=L,b_i=-\}|.\label{dln}
 \end{eqnarray}
 In other words, the length of the partition corresponding to the configuration on the left boundary is the same as the number
 of columns of the rail-yard graph with paramenters satisfying $(a_i,b_i)=(L,-)$; between each pair of non-adjacent nearest
particles on the left boundary, there are $(M-1)$ removed vertices.
 \end{assumption}

See Figure \ref{fig:ryy} for an example of a dimer covering on a rail yard graph with left boundary condition given by $(4,2,0)$ and right boundary condition given by the empty partition $\emptyset$.
 
  \begin{figure}
\includegraphics[width=.6\textwidth]{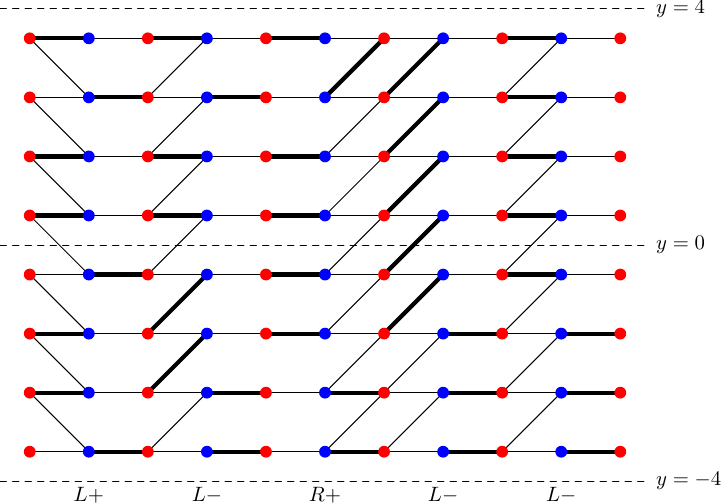}
\caption{A dimer covering on a rail yard graph with left boundary condition given by $(4,2,0)$ and right boundary condition given by the empty partition. Here $N=N_{L,-}=3$. Assume that above $y=4$, the dimer configuration at each odd vertex (represented by red points) corresponds to a hole; below $y=-4$, the dimer configuration at each odd vertex corresponds to a particle.}\label{fig:ryy}
\end{figure}
 
 This distribution in the limit as $N_{L,-}$ goes to infinity converges to a uniform
 measure on the whole interval $[0,M]$. By~\cite[Example~1.3.7]{IGM15}, we have
 \begin{equation*}
 s_{\lambda^{(l)}}\left(\underline{x}^{(L,-)}\right)=\prod_{i,j\in(L,-):i<j}\frac{x_i^M-x_j^M}{x_i-x_j}
 \end{equation*}
 when all the $x_i$'s are distinct, and extended by continuity when some $x_i$'s
 are equal.

 \begin{lemma}\label{l22}
 Assume (\ref{c151}) holds. Let $\rho_{\epsilon}^t$ be the probability measure for the partitions corresponding to dimer configurations adjacent to the column of odd vertices labeled by $2t-1$ in $RYG(l^{(\epsilon)},r^{(\epsilon)},\underline{a}^{(\epsilon)},\underline{b}^{(\epsilon)})$, conditional on the left and right boundary condition $\lambda^{(l)}$ and $\emptyset$, respectively.
 Let $k$ be a positive integer and let
 \begin{eqnarray*}
 K\subset\{i\in[t+1..r^{(\epsilon)}]:a_i^{(\epsilon)}=L,b_i^{(\epsilon)}=-\}\ \mathrm{s.t.}\ |K|=k.
 \end{eqnarray*}
 and let
 \begin{eqnarray*}
 &&\underline{u}^K=\{u_i,i\in K\};\\
 &&\underline{x}^{(L,-,>t)\setminus K}=
 \{x_i:i\in [t+1..r^{(\epsilon)}]\setminus K, a_i^{(\epsilon)}=L,b_i^{(\epsilon)}=-\}.
 \end{eqnarray*}
 Assume that the finite set $K$ is independent of $\epsilon$.
 Assume $t\in\left[v_{p_t-1}^{(\epsilon)},v_{p_t}^{(\epsilon)}\right]$ such that
 \begin{eqnarray}
\lim_{\epsilon\rightarrow 0} \frac{t-v_{p_t-1}^{(\epsilon)}}{v_{p_t}^{(\epsilon)}-v_{p_t-1}^{(\epsilon)}}=\alpha_t\label{lmt}
 \end{eqnarray}
Under Assumptions \ref{ap14}, \ref{ap41}, \ref{ap110} we have
 \begin{eqnarray*}
\lim_{\epsilon\rightarrow 0} \frac{1}{|r^{(\epsilon)}-l^{(\epsilon)}|}\log \mathcal{S}_{\rho^t_{\epsilon},\underline{x}^{(L,-,>t)}}\left(\underline{u}^K,\underline{x}^{(L,-,>t)\setminus K}\right)=\sum_{i\in K}\left[Q_{p_t,\alpha_t}(u_i)-Q_{p_t,\alpha_t}(x_i)\right]
 \end{eqnarray*}
 where
 \begin{small}
 \begin{eqnarray}\label{dqp}
&& Q_{p_t,\alpha_t}(u)=\sum_{p=1}^m\sum_{s=1}^{n_p}\frac{(V_p-V_{p-1})\zeta_{L,-,s,p}}{V_m-V_0}\log\left(\frac{u^M-[x_s^{(p)}]^M}{u-x_s^{(p)}}\right)\\
&&+\sum_{p=1}^{p_t-1}\sum_{s=1}^{n_p}\frac{(V_p-V_{p-1})\zeta_{R,+,s,p}}{V_m-V_0}\log(1+ux_{s}^{(p)})+\sum_{s=1}^{n_{p_t}}\frac{\alpha_t(V_{p_t}-V_{{p_t}-1})\zeta_{R,+,s,p_t}}{V_m-V_0}\log(1+ux_{s}^{(p_t)})\notag\\
&&-\sum_{p=1}^{p_t-1}\sum_{s=1}^{n_p}\frac{(V_p-V_{p-1})\zeta_{L,+,s,p}}{V_m-V_0}\log(1-ux_{s}^{(p)})-\sum_{s=1}^{n_{p_t}}\frac{\alpha_t(V_{p_t}-V_{{p_t}-1})\zeta_{L,+,s,p_t}}{V_m-V_0}\log(1-ux_{s}^{(p_t)}).\notag
 \end{eqnarray}
 \end{small}
 Moreover, the convergence is uniform when each $u_i$ is in a small complex neighborhood of $x_i$ for $i\in K$.
 \end{lemma}
  
\begin{proof}Let
 \begin{eqnarray*}
 \underline{x}^{(L,-)\setminus K}=
 \{x_i:i\in [l..r]\setminus K, a_i=L,b_i=-\}.
 \end{eqnarray*}

By Lemma \ref{lm212}, we have
\begin{eqnarray*}
&&\frac{1}{|r^{(\epsilon)}-l^{(\epsilon)}|}
\log \mathcal{S}_{\rho^t,\underline{x}^{(L,-,>t)}}\left(\underline{u}^K,\underline{x}^{(L,-,>t)\setminus K}\right)\\
&=&\frac{1}{|r^{(\epsilon)}-l^{(\epsilon)}|}
\log \left(\frac{s_{\lambda^{(l)}}\left(\underline{u}^K,\underline{x}^{(L,-)\setminus K}\right)}{s_{\lambda^{(l)}}(\underline{x}^{(L,-)})}
\prod_{i\in[l..t],b_i=+;}
\prod_{j\in[t+1..r],a_j=L,b_j=-}\frac{\psi_{ij}}{z_{ij}}\right).
\end{eqnarray*}
where 
\begin{eqnarray*}
\psi_{ij}=\begin{cases}1+\psi_i\psi_j\ &\mathrm{if}\ a_i\neq a_j.\\\frac{1}{1-\psi_i\psi_j}\ &\mathrm{if}\ a_i=a_j.\end{cases}
\end{eqnarray*}
and for $i\in[l..r]$,
\begin{eqnarray*}
\psi_i=\begin{cases}u_i, &\mathrm{if}\ i\in K,\\ x_i, &\mathrm{otherwise}.\end{cases}
\end{eqnarray*}
Under Assumption \ref{ap14}, \ref{ap110} we obtain
\begin{small}
\begin{eqnarray*}
&&\frac{1}{|r^{(\epsilon)}-l^{(\epsilon)}|}
\log \mathcal{S}_{\rho^t,\underline{x}^{(L,-,>t)}}\left(\underline{u}^K,\underline{x}^{(L,-,>t)\setminus K}\right)\\
&=&\frac{1}{|r^{(\epsilon)}-l^{(\epsilon)}|}
\left[\log \left(\prod_{i,j\in K,i<j}\frac{u_i^M-u_j^M}{x_i^M-x_j^M}\right)
+\log \left(\prod_{i,j\in K,i<j}\frac{x_i-x_j}{u_i-u_j}\right)\right.\\
&&+\left.\log \left(\prod_{i\in K,j\in (L,-)\setminus K}\frac{u_i^M-x_j^M}{x_i^M-x_j^M}\right)+
\log \left(\prod_{i\in K,j\in (L,-)\setminus K}\frac{x_i-x_j}{u_i-x_j}\right)
\right.\\
&&
\left.+\log\left(
\prod_{i\in[l..t],b_i=+;}
\prod_{j\in K,a_j\neq a_i}\frac{1+x_iu_j}{1+x_ix_j}\right)
+\log\left(
\prod_{i\in[l..t],b_i=+;}
\prod_{j\in K,a_j= a_i}\frac{1-x_ix_j}{1-x_iu_j}\right)\right]
.
\end{eqnarray*}
\end{small}
Under Assumption \ref{ap41} we obtain
\begin{small}
\begin{eqnarray*}
&&\lim_{\epsilon\rightarrow 0}\frac{1}{|r^{(\epsilon)}-l^{(\epsilon)}|}
\log \mathcal{S}_{\rho^t,\underline{x}^{(L,-,>t)}}\left(\underline{u}^K,\underline{x}^{(L,-,>t)\setminus K}\right)\\
&=&\sum_{i\in K}\sum_{p=1}^m\sum_{s=1}^{n_p}\frac{(V_p-V_{p-1})\zeta_{L,-,s,p}}{V_m-V_0}
\left[\log\left(\frac{u_i^M-[x_s^{(p)}]^M}{u_i-x_s^{(p)}}\right)-\log\left(\frac{x_i^M-[x_s^{(p)}]^M}{x_i-x_s^{(p)}}\right)\right]\\
&&+\sum_{i\in K}\sum_{p=1}^{p_t-1}\sum_{s=1}^{n_p}\frac{(V_p-V_{p-1})\zeta_{R,+,s,p}}{V_m-V_0}\left(\log(1+u_ix_{s}^{(p)})-\log(1+x_ix_{s}^{(p)})\right)\\
&&+\sum_{i\in K}\sum_{s=1}^{n_{p_t}}\frac{\alpha_t(V_{p_t}-V_{{p_t}-1})\zeta_{R,+,s,p_t}}{V_m-V_0}\left(\log(1+u_ix_{s}^{(p_t)})-\log(1+x_ix_{s}^{(p_t)})\right)\\
&&+\sum_{i\in K}\sum_{p=1}^{p_t-1}\sum_{s=1}^{n_p}\frac{(V_p-V_{p-1})\zeta_{L,+,s,p}}{V_m-V_0}\left(\log(1-x_ix_{s}^{(p)})-\log(1-u_ix_{s}^{(p)})\right)\\
&&+\sum_{i\in K}\sum_{s=1}^{n_{p_t}}\frac{\alpha_t(V_{p_t}-V_{{p_t}-1})\zeta_{L,+,s,p_t}}{V_m-V_0}\left(\log(1-x_ix_{s}^{(p_t)})-\log(1-u_ix_{s}^{(p_t)})\right)
\end{eqnarray*}
\end{small}
\end{proof}

\subsection{Limit Shape}

Let $k\geq 1$ be a positive integer. Let $t\in\left[v_{p_t-1}^{(\epsilon)},v_{p_t}^{(\epsilon)}\right]$ satisfying (\ref{lmt}). Let $\rho^t_{\epsilon}$ be the probability measure for the partitions corresponding to dimer configurations adjacent to the column of odd vertices labeled by $2t-1$ in $RYG(l^{(\epsilon)},r^{(\epsilon)},\underline{a}^{(\epsilon)},\underline{b}^{(\epsilon)})$, conditional on the left and right boundary condition $\lambda^{(l)}$ and $\emptyset$, respectively.
Let $\mathbf{m}_{\rho^t_{\epsilon}}$ be the random counting measure on $\RR$ regarding a random partition $\lambda\in \YY_N$ with distribution $\rho_{\epsilon}^t$.
We shall compute the moments of the measure $\mathbf{m}_{\rho^t_{\epsilon}}$ via (\ref{emt}) by applying the differential operator to the Schur generating function.

For simplicity, write
\begin{align}
&N^{(\epsilon)}=r^{(\epsilon)}-l^{(\epsilon)}.\label{dne}\\
&N^{(L,-,>t)}=
\left| \{i\in [t+1..r^{(\epsilon)}]: a_i^{(\epsilon)}=L,b_i^{(\epsilon)}=-\}\right|.\notag
\end{align}
Let $K$ be a finite set independent of $\epsilon$ as defined in Lemma \ref{l22}. Define
\begin{eqnarray*}
 \underline{u}^{(L,-,>t)}=
 \{u_i:i\in [t+1..r^{(\epsilon)}], a_i^{(\epsilon)}=L,b_i^{(\epsilon)}=-\}.
 \end{eqnarray*}

By Lemma \ref{l22}, we may write
\begin{align}
&\mathcal{S}_{\rho^t_{\epsilon},\underline{x}^{(L,-,>t)}}\left(\underline{u}^{(L,-,>t)}\right)
=e^{N^{(\epsilon)}\left(\sum_{i\in (L,-,>t)}\left[Q_{p_t,\alpha_t}(u_i)-Q_{p_t,\alpha_t}(x_i)\right]\right)}T_{N^{(\epsilon)}}\left(\underline{u}^{(L,-,>t)}\right)\label{rst}
\end{align}
such that
\begin{eqnarray}
\lim_{\epsilon\rightarrow 0}\frac{1}{N^{(\epsilon)}}\log  T_{N^{(\epsilon)}}\left(\underline{u}^K,\underline{x}^{(L,-,>t)\setminus K}\right)=0.\label{llt}
\end{eqnarray}
and
\begin{eqnarray}
T_{N^{(\epsilon)}}\left(\underline{x}^{(L,-,>t)}\right)=1;\label{tnx}
\end{eqnarray}
and the convergence is uniform when each $u_i$ is in a small complex neighborhood of $x_i$ for $i\in K$.

Then by (\ref{emt}) (\ref{rst}),
\begin{small}
\begin{eqnarray*}
&&\EE\left(\int_{\RR}x^k\mathbf{m}_{\rho^t_{\epsilon}}(dx)\right)^m=
\frac{1}{[N^{(L,-,>t)}]^{m(k+1)}}
    (\mathcal{D}_k)^m\mathcal{S}_{\rho^t_{\epsilon},\underline{x}^{(L,-,>t)}}\left(\underline{u}^{(L,-,>t)}\right)|_{\underline{u}^{(L,-,>t)}=\underline{x}^{(L,-,>t)}}
    \\
   &=&\left(\frac{N^{(\epsilon)}}{N^{(L,-,>t)}}\right)^{m(k+1)}\frac{1}{[N^{(\epsilon)}]^{m(k+1)}}
  \left.\left[T_{N^{(\epsilon)}}\left(\underline{u}^{(L,-,>t)}\right)(\mathcal{D}_k)^me^{N^{(\epsilon)}\left(\sum_{i\in (L,-,>t)}\left[Q_{p_t,\alpha_t}(u_i)-Q_{p_t,\alpha_t}(x_i)\right]\right)}\right|_{\underline{u}^{(L,-,>t)}=\underline{x}^{(L,-,>t)}}\right.\\
   &&+\left.R\right]
\end{eqnarray*}
\end{small}
where $R$ is the terms in $(\mathcal{D}_k)^m\mathcal{S}_{\rho^t_{\epsilon},\underline{x}^{(L,-,>t)}}\left(\underline{u}^{(L,-,>t)}\right)|_{\underline{u}^{(L,-,>t)}=\underline{x}^{(L,-,>t)}}$ obtained when the differential operator $(\mathcal{D}_k)^m$ acts on $T_{N^{(\epsilon)}}\left(\underline{u}^{(L,-,>t)}\right)$ as well.
From (\ref{llt}) we see that the leading term of $\EE\int_{\RR}x^k\mathbf{m}_{\rho^t_{\epsilon}}(dx)$ as $\epsilon\rightarrow 0$ is the same as that of 
\begin{small}
\begin{eqnarray}
&&\frac{1}{[N^{(\epsilon)}]^{m(k+1)}}
   T_{N^{(\epsilon)}}\left(\underline{u}^{(L,-,>t)}\right)\mathcal({D}_k)^m\left.e^{N^{(\epsilon)}\left(\sum_{i\in (L,-,>t)}\left[Q_{p_t,\alpha_t}(u_i)-Q_{p_t,\alpha_t}(x_i)\right]\right)}\right|_{\underline{u}^{(L,-,>t)}=\underline{x}^{(L,-,>t)}}\label{ltm1}\\
  &=&\left. \frac{1}{[N^{(\epsilon)}]^{m(k+1)}}
   \mathcal({D}_k)^m e^{N^{(\epsilon)}\left(\sum_{i\in (L,-,>t)}\left[Q_{p_t,\alpha_t}(u_i)-Q_{p_t,\alpha_t}(x_i)\right]\right)}\right|_{\underline{u}^{(L,-,>t)}=\underline{x}^{(L,-,>t)}}\notag
\end{eqnarray}
\end{small}
where the last identity follows from (\ref{tnx}).

When $m=1$, (\ref{ltm1}) can be computed as follows 
\begin{eqnarray*}
&&\frac{1}{[N^{(\epsilon)}]^{k+1}}
   \frac{1}{\prod_{i,j\in(L,-,>t):i<j}(u_i-u_j)}
   \sum_{s\in (L,-,>t)}\left(u_s\frac{\partial}{\partial u_s}\right)^k\\
   &&\left.\left[e^{N^{(\epsilon)}\left(\sum_{i\in (L,-,>t)}\left[Q_{p_t,\alpha_t}(u_i)-Q_{p_t,\alpha_t}(x_i)\right]\right)}\prod_{i,j\in(L,-,>t):i<j}(u_i-u_j)\right]\right|_{\underline{u}^{(L,-,>t)}=\underline{x}^{(L,-,>t)}}
\end{eqnarray*}
whose leading term is the same as that of 
\begin{small}
\begin{equation*}
\mathcal{M}_{k,\epsilon,t}:=
\lim_{\underline{u}^{(L,-,>t)}\rightarrow\underline{x}^{(L,-,>t)}}\sum_{s\in(L,-,>t)}\sum_{g=0}^{k}\left[N^{(\epsilon)}\right]^{-g-1}
\binom{k}{g}
u_s^k[Q'_{p_t,\alpha_t}(u_s)]^{k-g}\left(\sum_{j\in
  (L,-,>t)\setminus\{s\}}\frac{1}{u_s-u_j}\right)^{g}.
\end{equation*}
\end{small}
 First we assume that the edge weights $\left\{x_i^{(p)}:p\in[m], i\in [n_p]\right\}$ are pairwise distinct. For each $s\in (L,-,>t)$, let
\begin{equation*}
S_{\epsilon}(s)=\{j\in (L,-,>t): x_j=x_{s}\}.
\end{equation*}
Then
  \begin{multline*}
\mathcal{M}_{k,\epsilon,t}=\lim_{\underline{u}^{(L,-,>t)}\rightarrow\underline{x}^{(L,-,>t)}}\sum_{s\in(L,-,>t)}\sum_{g=0}^{k}\left[N^{(\epsilon)}\right]^{-g-1}\left(\begin{array}{c}k\\g\end{array}\right)u_s^k[Q'_{p_t,\alpha_t}(u_s)]^{k-g}\\
\times\left[\sum_{q=0}^{l}\left(\begin{array}{c}g\\q\end{array}\right)\left(\sum_{j\in(L,-,>t)\setminus
S_{\epsilon}(s)}\frac{1}{u_s-u_j}\right)^{g-q}\left(\sum_{j\in
S_{\epsilon}(s)\setminus\{s\}}\frac{1}{u_s-u_j}\right)^{q}\right],
\end{multline*}
By Assumption \ref{ap41}, we obtain
\begin{eqnarray*}
&&\mathcal{M}_{k,\epsilon,t}
=\lim_{\underline{u}^{(L,-,>t)}\rightarrow\underline{x}^{(L,-,>t)}}\left[N^{(\epsilon)}\right]^{-1}\sum_{p\in[m]}\sum_{i\in [n_p]}\sum_{s\in(L,-,>t):x_s=x_i^{(p)}}\sum_{g=0}^{k}\left(\begin{array}{c}k\\g\end{array}\right)u_s^k[Q'_{p_t,\alpha_t}(u_s)]^{k-g}\\
&&\times\left[\sum_{q=0}^{g}\left(\begin{array}{c}g\\q\end{array}\right)\left[A_{p,i}(u_s)
\right]^{g-q}\left(\frac{1}{N^{(\epsilon)}}\sum_{j\in
S_{\epsilon}(s)\setminus\{s\}}\frac{1}{u_s-u_j}\right)^{q}\right]
\end{eqnarray*}
where for each $p\in[m]$ and $i\in[n_p]$, we have
\begin{enumerate}
\item If $p\neq p_t$, i.e. if $p\in [p_t+1..m]$,
\begin{small}
\begin{eqnarray*}
A_{p,i}(u):&=&\sum_{d\in [p_t+1..m],d\neq p}\sum_{j\in [n_d],a_{j,d}=L,b_{j,d}=-}\frac{V_d-V_{d-1}}{(V_m-V_0)n_d}\frac{1}{u-x_j^{(d)}}\\
&&+
\sum_{j\in [n_{p}],j\neq i,a_{j,{p}}=L,b_{j,{p}}=-}\frac{V_p-V_{p-1}}{(V_m-V_0)n_p}\frac{1}{u-x_j^{(p)}}\\
&&+
\sum_{j\in [n_{p_t}],a_{j,{p_t}}=L,b_{j,{p_t}}=-}\frac{(V_{p_t}-V_{p_t-1})(1-\alpha_t)}{(V_m-V_0)n_{p_t}}\frac{1}{u-x_j^{(p_t)}}.
\end{eqnarray*}
\end{small}
\item If $p=p_t$,
\begin{eqnarray*}
A_{p,i}(u):&=&\sum_{d\in [p_t+1..m]}\sum_{j\in [n_d],a_{j,d}=L,b_{j,d}=-}\frac{V_d-V_{d-1}}{(V_m-V_0)n_d}\frac{1}{u-x_j^{(d)}}\\
&&+
\sum_{j\in [n_{p_t}],j\neq i,a_{j,{p_t}}=L,b_{j,{p_t}}=-}\frac{(V_{p_t}-V_{p_t-1})(1-\alpha_t)}{(V_m-V_0)n_{p_t}}\frac{1}{u-x_j^{(p_t)}}.
\end{eqnarray*}
\end{enumerate}
Hence we have
\begin{eqnarray*}
&&\mathcal{M}_{k,\epsilon,t}
=\lim_{\underline{u}^{(L,-,>t)}\rightarrow\underline{x}^{(L,-,>t)}}\left[N^{(\epsilon)}\right]^{-1}\sum_{p\in[m]}\sum_{i\in [n_p]}\sum_{s\in(L,-,>t):x_s=x_i^{(p)}}\sum_{q=0}^{k}{k\choose q}\\&&\left(\frac{1}{N^{(\epsilon)}}\sum_{j\in
S_{\epsilon}(s)\setminus\{s\}}\frac{1}{u_s-u_j}\right)^{q}
u_s^k[Q'_{p_t,\alpha_t}(u_s)+A_{p,i}(u_s)]^{k-q}
\end{eqnarray*}

By Lemma \ref{la1}, we obtain
\begin{eqnarray*}
\mathcal{M}_{k,\epsilon,t}&=&\lim_{\underline{u}^{(L,-,>t)}\rightarrow\underline{x}^{(L,-,>t)}}\sum_{p\in[p_t+1..m]}\sum_{i\in[n_p]:a_{i,p}=L,b_{i,p}=-}\sum_{q=0}^{k}{k\choose q}\frac{1}{(q+1)!}\left[\frac{(V_p-V_{p-1})\alpha_t}{(V_m-V_0)n_p}\right]^{q+1}\\
&&\times
\left.
\frac{
\partial^q
}{
\partial u^q
}
\left[
  u^k\left(Q_{p_t,\alpha_t}'(u)+
  A_{p,i}(u)\right)^{k-q}
\right]
\right|_{u=x_i^{(p)}}\\
&&+\lim_{\underline{u}^{(L,-,>t)}\rightarrow\underline{x}^{(L,-,>t)}}\sum_{i\in[n_{p_t}]:a_{i,{p_t}}=L,b_{i,p_t}=-}\sum_{q=0}^{k}{k\choose q}\frac{1}{(q+1)!}\left[\frac{(V_{p_t}-V_{p_t-1})\alpha_t}{(V_m-V_0)n_{p_t}}\right]^{q+1}\\
&&\times
\left.
\frac{
\partial^q
}{
\partial u^q
}
\left[
  u^k\left(Q_{p_t,\alpha_t}'(u)+
  A_{p,i}(u)\right)^{k-q}
\right]
\right|_{u=x_i^{(p_t)}}
\end{eqnarray*}


Note that by (\ref{dqp}), we obtain
\begin{small}
\begin{eqnarray}
&& Q_{p_t,\alpha_t}'(u)=\sum_{p=1}^m\sum_{s=1}^{n_p}\frac{(V_p-V_{p-1})\zeta_{L,-,s,p}}{V_m-V_0}\left(\frac{Mu^{M-1}}{u^M-[x_s^{(p)}]^M}-\frac{1}{u-x_s^{(p)}}\right)\label{dqp1}\\
&&+\sum_{p=1}^{p_t-1}\sum_{s=1}^{n_p}\frac{(V_p-V_{p-1})\zeta_{R,+,s,p}}{V_m-V_0}\frac{x_s^{(p)}}{1+ux_{s}^{(p)}}+\sum_{s=1}^{n_{p_t}}\frac{\alpha_t(V_{p_t}-V_{{p_t}-1})\zeta_{R,+,s,p_t}}{V_m-V_0}\frac{x_s^{(p_t)}}{1+u_ix_{s}^{(p_t)}}\notag\\
&&+\sum_{p=1}^{p_t-1}\sum_{s=1}^{n_p}\frac{(V_p-V_{p-1})\zeta_{L,+,s,p}}{V_m-V_0}\frac{x_s^{(p)}}{1-ux_{s}^{(p)}}+\sum_{s=1}^{n_{p_t}}\frac{\alpha_t(V_{p_t}-V_{{p_t}-1})\zeta_{L,+,s,p_t}}{V_m-V_0}\frac{x_s^{(p_t)}}{1-ux_{s}^{(p_t)}}.\notag
\end{eqnarray}
\end{small}
Let
\begin{small}
\begin{eqnarray}
W_{p_t,\alpha_t}(u):&=&\sum_{d\in [p_t+1..m]}\sum_{j\in [n_d],a_{j,d}=L,b_{j,d}=-}\frac{V_d-V_{d-1}}{(V_m-V_0)n_d}\frac{1}{u-x_j^{(d)}}\label{dwu}\\
&&+
\sum_{j\in [n_{p_t}],a_{j,{p_t}}=L,b_{j,{p_t}}=-}\frac{(V_{p_t}-V_{p_t-1})(1-\alpha_t)}{(V_m-V_0)n_{p_t}}\frac{1}{u-x_j^{(p_t)}}.\notag
\end{eqnarray}
\end{small}
i.e. $W_{p_t,\alpha_t}(u)$ is exactly $A_{p,i}(u)$ plus the corresponding term for $(p,i)$.

It is straightforward to check the following lemma:
\begin{lemma}\label{l333}
\begin{align*}
\lim_{\epsilon\rightarrow 0}\frac{N^{(L,-,>t)}}{N^{(\epsilon)}}=\phi(x)
\end{align*}
where $\phi$ is a continuous, piecewise linear function in $x\in[0,1]$ with  
\begin{itemize}
\item $\phi(1)=0$; and
\item $x=x(p_t,\alpha_t):=\frac{V_{p_t-1}-V_0+\alpha_t(V_{p_t}-V_{p_t-1})}{V_m-V_0}$; and
\item the slope of $\phi$ in $(V_{p-1},V_p)$ is the negative of the quotient of total number of even vertices with label $(L,-)$ divided by the total number of even vertices in $(V_{p-1},V_p)$.
\end{itemize}
\end{lemma}
Using residue and following similar computations, we have
\begin{equation*}
\lim_{\epsilon\rightarrow 0}\mathbb{E}\int_{\RR}x^{k}\textbf{m}_{\rho_{\epsilon}^t}(dx)=
\frac{1}{2(k+1)[\phi(x(p_t,\alpha_t))]^{k+1}\pi \mathbf{i}}\oint_{C_{p_t,L,-}}\frac{dz}{z}\left[z\left(Q_{p_t,\alpha_t}'(z)+W_{p_t,\alpha_t}(z)\right)\right]^{k+1},
\end{equation*}
where $C_{p_t,L,-}$ is a simple, closed, positively oriented, contour containing only the poles $\{x_i^{(p)}\}_{p\in[p_t..m],i\in[n_p]:a_{i,p}=L,b_{i,p}=-}$ of the integrand, and no other singularities.

In the case that some of the edge weights $\{x_i^{(p)}\}_{p\in[p_t..m],i\in[n_p]:a_{i,p}=L,b_{i,p}=-}$ may be equal, it is straightforward to check that  the expression for the limit of the moment of
$\mathbf{m}_{\rho_{\epsilon}^t}$ as $\epsilon\rightarrow 0$ is the same.

Now let $m=2$ in (\ref{ltm1}), we obtain as $\epsilon\rightarrow 0$
\begin{eqnarray*}
\EE\left(\int_{\RR}x^k\mathbf{m}_{\rho^t_{\epsilon}}(dx)\right)^2\approx\left. \frac{1}{[N^{(L,-,>t)}]^{2(k+1)}}
   \mathcal({D}_k)^2e^{N^{(\epsilon)}\left(\sum_{i\in (L,-,>t)}\left[Q_{p_t,\alpha_t}(u_i)-Q_{p_t,\alpha_t}(x_i)\right]\right)}\right|_{\underline{u}^{(L,-,>t)}=\underline{x}^{(L,-,>t)}}
\end{eqnarray*}
Explicit computations show that 
$\EE\left(\int_{\RR}x^k\mathbf{m}_{\rho^t_{\epsilon}}(dx)\right)^2$ and $\left(\EE\int_{\RR}x^k\mathbf{m}_{\rho^t_{\epsilon}}(dx)\right)^2$ have the same leading term as $\epsilon\rightarrow 0$. Hence we have

\begin{theorem}\label{t33}Let $k\geq 1$ be a positive integer.
Let $\rho^t_{\epsilon}$ be the probability measure for the partitions corresponding to dimer configurations adjacent to the column of odd vertices labeled by $2t-1$ in $RYG(l^{(\epsilon)},r^{(\epsilon)},\underline{a}^{(\epsilon)},\underline{b}^{(\epsilon)})$, conditional on the left and right boundary condition $\lambda^{(l)}$ and $\emptyset$, respectively.
Let $\mathbf{m}_{\rho^t_{\epsilon}}$ be the random counting measure on $\RR$ regarding a random partition $\lambda\in \YY_N$ with distribution $\rho_{\epsilon}^t$. Suppose that (\ref{c151}) and Assumptions \ref{ap14}, \ref{ap41}, \ref{ap110} hold.
\begin{eqnarray*}
\lim_{\epsilon\rightarrow 0}\int_{\RR}x^{k}\textbf{m}_{\rho_{\epsilon}^t}(dx)=
\frac{1}{2(k+1)\pi \mathbf{i}}\oint_{C_{p_t,L,-}}\frac{dz}{z}\left[\frac{z\left(Q_{p_t,\alpha_t}'(z)+W_{p_t,\alpha_t}(z)\right)}{\phi_{p_t,\alpha_t}}\right]^{k+1},
\end{eqnarray*}
where $Q_{p_t,\alpha_t}'(z)$ and $W_{p_t,\alpha_t}(z)$ are defined in (\ref{dqp}) and (\ref{dwu}) respectively; $\phi_{p_t,\alpha_t}=\phi(x(p_t,\alpha_t))$ is defined as in Lemma \ref{l333}.
and $C_{p_t,L,-}$ is a simple, closed, positively oriented, contour containing only the poles $\{x_i^{(p)}\}_{p\in[p_t..m],i\in[n_p]:a_{i,p}=L,b_{i,p}=-}$ of the integrand, and no other singularities.
\end{theorem}

\subsection{Frozen boundary}
\begin{definition}
  \label{df41}
  Let $\mathbf{m}_{p_t,\alpha_t}$ be the limit of $\mathbf{m}_{\rho_{\epsilon}^t}$ as $\epsilon\rightarrow 0$. Let 
  \begin{eqnarray}
  \chi:=V_{p_{t-1}}+\alpha_t(V_{p_t}-V_{p_t-1})\label{dfc}
  \end{eqnarray}
  Let $\mathcal{L}$ be the set of $(\chi,\kappa)$ inside $\mathcal{R}$ such that
  the density $\frac{d\mathbf{m}_{p_t,\alpha_t}\left(\kappa\right)}{d\kappa}$ is not
  equal to 0 or 1. Then $\mathcal{L}$ is called the \emph{liquid region}. Its boundary
  $\partial \mathcal{L}$ is called the \emph{frozen boundary}.
\end{definition}

In order to identify the frozen boundary, we need to compute the density $\frac{d\mathbf{m}_{p_t,\alpha_t}\left(\kappa\right)}{d\kappa}$ of the limit counting measure, which can be computed by the Stieljes transform $\mathrm{St}_{\mathbf{m}_{p_t,\alpha_t}}(x)$ of $\mathbf{m}_{p_t,\alpha_t}$.

Define 
\begin{eqnarray*}
F_{p_t,\alpha_t,M}(z)=\frac{z\left[Q'_{p_t,\alpha_t}(z)+W_{p_t,\alpha_t}(z)\right]}{\phi_{p_t,\alpha_t}}.
\end{eqnarray*}
Then when
$x$ is in a neighborhood of infinity, we have
\begin{align*}
  \mathrm{St}_{\mathbf{m}_{p_t,\alpha_t}}(x)
  &=\sum_{j=0}^{\infty}x^{-(j+1)} \int_{\RR}y^j\mathbf{m}_{p_t,\alpha_t}(dy)\\
&=\sum_{j=0}^{\infty}
\frac{1}{2(j+1)\pi\mathbf{i}}\oint_{C_{p_t,L,-}}\left(\frac{F_{p_t,\alpha_t,M}(z)}{x}\right)^{j+1}\frac{dz}{z}\\
&=-\frac{1}{2\pi\mathbf{i}}\oint_{C_{p_t,L,-}}\log\left(1-\frac{F_{p_t,\alpha_t,M}(z)}{x}\right)\frac{dz}{z}.
\end{align*}
Integration by parts gives
\begin{equation*}
\mathrm{St}_{\mathbf{m}_{p_t,\alpha_t}}(x)=
\frac{1}{2\pi\mathrm{i}}\left[\oint_{C_{p_t,L,-}}\log
z\frac{\frac{d}{dz}\left(1-\frac{F_{p_t,\alpha_t,M}(z)}{x}\right)}{1-\frac{F_{p_t,\alpha_t,M}(z)}{x}}dz
-\oint_{C_{p_t,L,-}} 
d\left(\log z\log \left(1-\frac{F_{p_t,\alpha_t,M}(z)}{x}\right)\right)\right].
\end{equation*}

Because for each $p\in[p_t..m]$, and $j\in[n_p]$ $F_{\kappa,p_t,\alpha_t}(z)$ has a Laurent series expansion in a neighborhood of $x_j^{(p)}$ given by
\begin{equation*}
F_{p_t,\alpha_t,M}(z)=\frac{(V_p-V_{p-1})\beta_{t,p}x_j^{(p)}}{(V_m-V_0)n_p\left(z-x_j^{(p)}\right)}+\sum_{k=0}^{\infty}c_k\left(z-x_j^{(p)}\right)^k,
\end{equation*}
where
\begin{eqnarray*}
\beta_{t,p}=\begin{cases}\alpha_t,&\mathrm{if}\ p=p_t;\\1,&\mathrm{otherwise}.\end{cases}
\end{eqnarray*}
$F_{p_t,\alpha_t,M}(z)=x$ has exactly one root in a neighborhood of $x_j^{(p)}$ when $x$ is in a neighborhood of $\infty$, and thus, we can can find a unique composite inverse Laurent series
given by
\begin{equation*}
G_{p_t,\alpha_t,M,j,p}(w)=x_j^{(p)}+\sum_{i=1}^{\infty}\frac{\beta_i^{(j)}}{w^i},
\end{equation*}
such that $F_{p_t,\alpha_t,M}(G_{p_t,\alpha_t,M,j}(w))=w$ when $w$ is in a neighborhood of
infinity. Then
\begin{equation}
  \label{eq:rootFG}
z_{j,p}(x)=G_{p_t,\alpha_t,M,j}(x)
\end{equation}
is the unique root of $F_{p_t,\alpha_t,M}(z)=x$ in a neighborhood of $x_i$.

Since $1-\frac{F_{p_t,\alpha_t,M}}{x}$ has exactly one zero $z_{j,p}(x)$ and one pole $x_j^{(p)}$ in a neighborhood of $x_j^{(p)}$, we have
\begin{equation*}
\oint_{x_j^{(p)}}d\left(\log z\log \left(1-\frac{F_{p_t,\alpha_t,M}(z)}{x}\right)\right)=0;
\end{equation*}
and therefore
\begin{equation*}
\mathrm{St}_{\mathbf{m}_{p_t,\alpha_t}}(x)=\sum_{p=p_t}^{m}\sum_{j=1}^{n_p}\left[\log(z_{j,p}(x))-\log x_j^{(p)}\right]\label{sjl}
\end{equation*}
when $x$ is in a neighborhood of infinity. By the complex analyticity of both
sides of \eqref{sjl}, we infer that \eqref{sjl} holds whenever $x$ is outside
the support of $\mathbf{m}_{p_t,\alpha_t}$.

\begin{lemma}\label{l25}
  Assume the liquid region is nonempty, and assume that for any $x\in \RR$,
  $F_{p_t,\alpha_t, M}(z)=x$ has at most one pair of nonreal conjugate roots.   Let $\chi$ be given by (\ref{dfc}) and $\kappa\in \RR$. Then for
  any point $(\chi,\kappa)$ lying on the frozen boundary, the equation
  \begin{align}
  F_{p_t,\alpha_t,M}=\frac{\kappa-(1-\phi_{p_t,\alpha_t})}{\phi_{p_t,\alpha_t}}\label{ceq}
 \end{align} 
  has double roots.
\end{lemma}
\begin{proof}
  The continuous density $f_{\mathbf{m}_{p_t,\alpha_t}}(x)$ of the
  measure $\mathbf{m}_{\kappa}(x)$ with respect to the Lebesgue measure is given
  by
  \begin{equation*}
    f_{\mathbf{m}_{p_t,\alpha_t}}(x)=-\lim_{\epsilon\rightarrow 0+}\frac{1}{\pi}\Im
    [\mathrm{St}_{\mathbf{m}_{p_t,\alpha_t}}(x+\mathbf{i}\epsilon)]
  \end{equation*}
  By~\eqref{sjl}, we have
  \begin{equation*}
    f_{\mathbf{m}_{p_t,\alpha_t}}(x)=-\lim_{\epsilon\rightarrow 0+}\frac{1}{\pi}\mathrm{Arg}\left(\prod_{p=p_t}^m\prod_{j=1}^{n_p} z_{j,p}(x+\mathbf{i}\epsilon)\right).
  \end{equation*}
  If the liquid region is nonempty, and for any $x\in \RR$, $F_{p_t,\alpha_t,M}=x$ has
  at most one pair of complex conjugate roots. Hence for each point $(\chi,\kappa)$
  in the liquid region, there is exactly one pair of roots
  $z_{j,p}\left(x+\mathbf{i}\epsilon\right)$
  from~\eqref{eq:rootFG} converging to non-real roots of
  $F_{p_t,\alpha_t,M}(z)=x$ as $\epsilon\rightarrow 0$; and all the others converge to real
  roots of $F_{\kappa,M}(z)=x$.
   We may consider the partition corresponding to dimer 
  configurations at each column of the rail-yard graph as a partition of length $N^{(\epsilon)}$ by adding 0's in the end if necessary, 
  then the particle corresponding to the minimal part of each partition are in the same horizontal level. A rescaling shows that
  \begin{align*}
  x\phi_{p_t,\alpha_t}+(1-\phi_{p_t,\alpha_t})=\kappa
  \end{align*}
Then, the lemma follows from
  the complex analyticity of the density of the limit measure with respect to
  $(\chi,\kappa)$.
\end{proof}

\subsection{The case $M=1$.}

When $M=1$, the equation (\ref{ceq}) becomes
\begin{small}
\begin{eqnarray}
&&\kappa-1+\phi_{p_t,\alpha_t}=\sum_{d\in [p_t+1..m]}\sum_{j\in [n_d],a_{j,d}=L,b_{j,d}=-}\frac{V_d-V_{d-1}}{(V_m-V_0)n_d}\left(1+\frac{x_j^{(d)}}{u-x_j^{(d)}}\right)\label{m1k}\\
&&+\sum_{j\in [n_{p_t}],a_{j,{p_t}}=L,b_{j,{p_t}}=-}\frac{(V_{p_t}-V_{p_t-1})(1-\alpha_t)}{(V_m-V_0)n_{p_t}}\left(1+\frac{x_j^{(p_t)}}{u-x_j^{(p_t)}}\right)\notag\\
&&+\sum_{p=1}^{p_t-1}\sum_{s=1}^{n_p}\frac{(V_p-V_{p-1})\zeta_{R,+,s,p}}{V_m-V_0}\left(1-\frac{1}{1+ux_{s}^{(p)}}\right)\notag\\
&&+\sum_{s=1}^{n_{p_t}}\frac{\alpha_t(V_{p_t}-V_{{p_t}-1})\zeta_{R,+,s,p_t}}{V_m-V_0}\left(1-\frac{1}{1+ux_{s}^{(p_t)}}\right)\notag\\
&&+\sum_{p=1}^{p_t-1}\sum_{s=1}^{n_p}\frac{(V_p-V_{p-1})\zeta_{L,+,s,p}}{V_m-V_0}\left(\frac{1}{1-ux_{s}^{(p)}}-1\right)\notag\\
&&+\sum_{s=1}^{n_{p_t}}\frac{\alpha_t(V_{p_t}-V_{{p_t}-1})\zeta_{L,+,s,p_t}}{V_m-V_0}\left(\frac{1}{1-ux_{s}^{(p_t)}}-1\right)\notag\\
&&:=H_{p_t,\alpha_t}(u)\notag
\end{eqnarray}
\end{small}

The expression in the righthand side of (\ref{m1k}) has singular points
\begin{small}
\begin{eqnarray}
&&\label{sgp1}\left\{-\frac{1}{x_s^{(p)}}\right\}_{p\in[p_t],s\in[n_p],a_{s,p}=R,b_{s,p}=+};\qquad
\left\{\frac{1}{x_s^{(p)}}\right\}_{p\in[p_t],s\in[n_p],a_{s,p}=L,b_{s,p}=+};\\
&&\{x_s^{(p)}\}_{p\in[p_t..m],s\in[n_p],a_{s,p}=L,b_{s,p}=-}.\notag
\end{eqnarray}
\end{small}

Note that under Assumption \ref{ap14},
\begin{enumerate}
\item Between each pair of nearest points in 
$\{x_s^{(p)}\}_{p\in[p_t..m],s\in[n_p],a_{s,p}=L,b_{s,p}=-}$, (\ref{m1k}) has at least one real solution;
\item Between each pair of nearest points in 
$\left\{-\frac{1}{x_s^{(p)}}\right\}_{p\in[p_t],s\in[n_p],a_{s,p}=R,b_{s,p}=+}$, (\ref{m1k}) has at least one real solution;
\item Between each pair of nearest points in 
$\left\{\frac{1}{x_s^{(p)}}\right\}_{p\in[p_t],s\in[n_p],a_{s,p}=L,b_{s,p}=+}$;, (\ref{m1k}) has at least one real solution;
\end{enumerate}
Let $N$ be the degree of (\ref{m1k}); items (1)(2)(3) above give at least $N-3$ real solutions for (\ref{m1k}); hence (\ref{m1k}) has at most one pair of complex conjugate roots.

By Lemma \ref{l25}, the frozen boundary is given by the equation (\ref{m1k}) has double roots.
When (\ref{m1k}) has double real roots, we have
\begin{eqnarray*}
\frac{dF_{p_t,\alpha_t,1}(u)}{du}=0;
\end{eqnarray*}
which gives 
\begin{small}
\begin{eqnarray}
&&0=\sum_{d\in [p_t+1..m]}\sum_{j\in [n_d],a_{j,d}=L,b_{j,d}=-}\frac{V_d-V_{d-1}}{(V_m-V_0)n_d}\left(-\frac{x_j^{(d)}}{(u-x_j^{(d)})^2}\right)\label{m2k}\\
&&+\sum_{j\in [n_{p_t}],a_{j,{p_t}}=L,b_{j,{p_t}}=-}\frac{(V_{p_t}-V_{p_t-1})(1-\alpha_t)}{(V_m-V_0)n_{p_t}}\left(-\frac{x_j^{(p_t)}}{(u-x_j^{(p_t)})^2}\right)\notag\\
&&+\sum_{p=1}^{p_t-1}\sum_{s=1}^{n_p}\frac{(V_p-V_{p-1})\zeta_{R,+,s,p}}{V_m-V_0}\left(\frac{x_s^{(p)}}{(1+ux_{s}^{(p)})^2}\right)\notag\\
&&+\sum_{s=1}^{n_{p_t}}\frac{\alpha_t(V_{p_t}-V_{{p_t}-1})\zeta_{R,+,s,p_t}}{V_m-V_0}\left(\frac{x_s^{(p_t)}}{(1+ux_{s}^{(p_t)})^2}\right)\notag\\
&&+\sum_{p=1}^{p_t-1}\sum_{s=1}^{n_p}\frac{(V_p-V_{p-1})\zeta_{L,+,s,p}}{V_m-V_0}\left(\frac{x_s^{(p)}}{(1-ux_{s}^{(p)})^2}\right)\notag\\
&&+\sum_{s=1}^{n_{p_t}}\frac{\alpha_t(V_{p_t}-V_{{p_t}-1})\zeta_{L,+,s,p_t}}{V_m-V_0}\left(\frac{x_s^{(p_t)}}{(1-ux_{s}^{(p_t)})^2}\right)\notag\\
&&:=G_{p_t,\alpha_t}(u)\notag
\end{eqnarray}
\end{small}

\subsection{The case $M=1$ and $m=1$.} When $m=1$, assume that 
\begin{itemize}
\item $n_1=n$; and
\item for $j\in[n]$, 
$
a_{j,1}=a_j;\qquad
b_{j,1}=b_j; \qquad 
x_j^{(1)}=x_j
$; and
\item 
$
V_0=0,\qquad V_1=1;\qquad \alpha_t=\chi.
$
\item For $\mathbf{a}\in\{L,R\}$, and $\mathbf{b}\in\{+,-\}$ $j\in[n]$,
$\zeta_{\mathbf{a},\mathbf{b},j,1}=\zeta_{\mathbf{a},\mathbf{b},j}$.
\end{itemize}
Then
\begin{align*}
\phi_{p_t,\alpha_t}=\phi_{1,\chi}=-\frac{\sum_{s\in[n]}\zeta_{L,-,s}}{n}(\chi-1)
\end{align*}
and (\ref{m1k}) becomes 
\begin{small}
\begin{align*}
&\kappa-\left[1-\frac{\sum_{s\in[n]}\zeta_{L,-,s}}{n}(1-\chi)\right]\\
&=\sum_{j\in [n],a_{j}=L,b_{j}=-}\frac{1-\chi}{n}\left(1+\frac{x_j}{u-x_j}\right)+\sum_{s=1}^{n}\chi\zeta_{R,+,s}\left(1-\frac{1}{1+ux_{s}}\right)+\sum_{s=1}^{n}\chi\zeta_{L,+s}\left(\frac{1}{1-ux_{s}}-1\right)
\end{align*}
\end{small}
and (\ref{m2k}) becomes 
\begin{small}
\begin{eqnarray*}
&&0=-\sum_{j\in [n],a_{j}=L,b_{j}=-}\frac{x_j(1-\chi)}{n(u-x_j)^2}+\sum_{s=1}^{n}\frac{\chi\zeta_{R,+,s}x_s}{(1+ux_{s})^2}+\sum_{s=1}^{n}\frac{\chi\zeta_{L,+,s}x_s}{(1-ux_{s})^2}
\end{eqnarray*}
\end{small}

Let 
\begin{eqnarray}
U(u):&=&\sum_{s=1}^{n}\zeta_{R,+,s}\left(1-\frac{1}{1+ux_{s}}\right)+\sum_{s=1}^{n}\zeta_{L,+s}\left(\frac{1}{1-ux_{s}}-1\right)\label{duu}\\
V(u):&=&\sum_{j\in [n],a_{j}=L,b_{j}=-}\frac{1}{n}\frac{x_j}{u-x_j}\label{dvu}
\end{eqnarray}

Then we may write the frozen boundary in parametric form as follows
\begin{eqnarray*}
\mathcal{FB}=(\chi(u),\kappa(u));
\end{eqnarray*}
where
\begin{eqnarray*}
\chi(u)=\frac{V'(u)}{V'(u)-U'(u)}
\end{eqnarray*}
and
\begin{small}
\begin{eqnarray*}
&&\kappa(u)=\chi(u)U(u)+(1-\chi(u))V(u)+1;
\end{eqnarray*}
\end{small}
The dual curve $\mathcal{FB}^{\vee}=(\chi^{\vee}(u),\kappa^{\vee}(u))$ can be obtained by the formula below
\begin{eqnarray*}
\chi^{\vee}(u)&=&\frac{\frac{d\kappa(u)}{du}}{\kappa(u)\frac{d\chi(u)}{du}-\chi(u)\frac{d\kappa(u)}{du}}\\
\kappa^{\vee}(u)&=&\frac{\frac{d\chi(u)}{du}}{\chi(u)\frac{d\kappa(u)}{du}-\kappa(u)\frac{d\chi(u)}{du}}
\end{eqnarray*}
Then the dual curve has the following parametric equation
\begin{eqnarray}
(\chi^{\vee}(u),\kappa^{\vee}(u))=\left(\frac{U(u)-V(u)}{V(u)+1},-\frac{1}{V(u)+1}\right)\label{pexy}
\end{eqnarray}

\begin{definition}(\cite{KO07})A degree-$d$ real algebraic curve $C\subset\mathbb{RP}^2$ is winding if the following two conditions hold:
\begin{enumerate}
\item it intersects every line $L\subset\mathbb{RP}^2$ in at least $d-2$ points counting multiplicity,
\item there exists a point $p_0\in\mathbb{RP}^2$ called center, such that every line
through $p_0$ intersects $C$ in $d$ points.
\end{enumerate}
The dual curve of a winding curve is called a cloud curve.
\end{definition}

The next theorem gives a description of the frozen boundary when $M=m=1$, shows that the frozen boundary is a cloud curve in this case, and explicitly computes its rank. The fact that the frozen boundaries for a large class of simply-connected liquid regions are cloud curves were also been proved in \cite{ADPZ20}. 

\begin{theorem}\label{t38}Assume $M=m=1$. Then the frozen boundary $\mathcal{FB}$ of perfect matchings on rail-yard graph is a cloud curve of rank $n'$, where $n'$ is the total number of distinct elements in the set
\begin{eqnarray}
\left\{\{x_j\}_{j\in[n],a_j=L,b_j=-},\left\{-\frac{1}{x_s}\right\}_{s\in[n],a_s=R,b_s=+},\left\{\frac{1}{x_k}\right\}_{k\in[n],a_k=L,b_k=+}\right\}\label{spt}
\end{eqnarray}
Moreover,
\begin{enumerate}[label=(\alph*)]
    \item $\mathcal{FB}$ is tangent to $\chi=0$, with number of tangent points equal to
    \begin{eqnarray}
\left|\left\{-\frac{1}{x_s}\right\}_{s\in[n],a_s=R,b_s=+},\left\{\frac{1}{x_k}\right\}_{k\in[n],a_k=L,b_k=+}\right|.\label{tn0}
\end{eqnarray}
\item $\mathcal{FB}$ is tangent to $\chi=1$, with number of tangent points equal to
\begin{eqnarray}
\left|\{x_j\}_{j\in[n],a_j=L,b_j=-}\right|.\label{tn1}
\end{eqnarray}
\item 
      If $u_i$ satisfies
\begin{eqnarray}
U(u_i)=V(u_i);\qquad -\frac{1}{V(u_i)+1}=c_i\in \RR,\label{uvci}
\end{eqnarray}
Then $\mathcal{FB}$ is tangent to the line $\kappa=V(u_i)+1$.
\end{enumerate}
\end{theorem}

\begin{proof}It suffices to show that the dual curve $\mathcal{FB}^{\vee}$ is a winding curve of degree $n'$.

From the parametric equation (\ref{pexy}) which we can read that its degree is $n'$. To show that ${\mathcal{FB}}^{\vee}$ is winding, we need to look at real intersections with straight lines.

From (\ref{pexy}) we also obtain that
  \begin{equation*}
    \chi^{\vee}(u)=\kappa^{\vee}(u)\left(V(u)-U(u)\right)
  \end{equation*}

  The points of intersection $(\chi^\vee(u),\kappa^\vee(u))$ of the dual curve with a straight line
  of the form $\chi^\vee=c\kappa^\vee+d$ have a parameter $u$ satisfying:
  \begin{equation}
    c=\Theta(u):=(d+1)V(u)-U(u)+d.
    \label{eq:intersect_dual}
  \end{equation}
  Note that $\Theta(u)$ satisfies the following conditions:
  \begin{enumerate}
      \item All the singularities of $\Theta(u)$ are given by (\ref{spt}).
      \item Between any two consecutive singularities of $\Theta(u)$, $\Theta(u)$ decreases from $\infty$ to $-\infty$.
      \item Let $t_2$ be the maximal singular point of $\Theta(u)$, then in $(t_2,\infty)$, $\Theta(u)$ decreases from $\infty$ to
      \begin{eqnarray}
      \xi_{\infty}:=d
     -\sum_{s=1}^{n}\zeta_{R,+,s}+\sum_{s=1}^{n}\zeta_{L,+s}\label{sif}
      \end{eqnarray}
      \item Let $t_1$ be the minimal singular point of $\Theta(u)$, then in $(-\infty,t_1)$, $\Theta(u)$ decreases from (\ref{sif}) to $-\infty$.
      \end{enumerate}
  
 Therefore for each $c\in \RR\setminus \xi_{\infty}$, (\ref{eq:intersect_dual}) has exactly $n'$ solutions in $u$. If $c=\xi_{\infty}$, (\ref{eq:intersect_dual}) has exactly $n'-1$ solutions in $u$ (and exactly $n'$ solutions if $\infty$ is counted). This implies that the dual curve $\mathcal{FB}^{\vee}$ intersects any straight line of the form $\chi^{\vee}=c\kappa^{\vee}+d$ at least $n'-1$ points, counting multiplicities.

  To consider the vertical lines $\kappa^\vee=c$, we rewrite the equations in homogeneous
  coordinates $[\chi^\vee:\kappa^\vee:z^\vee]$. We obtain that the line $\kappa^\vee=cz^\vee$ intersects the curve at the
  point $[1:0:0]$ with multiplicity at least $n'$, where
    with corresponding $u$'s given by the poles of $U(u)$ or the poles of $V(u)$. 
 This completes the proof that $\mathcal{FB}$ is a cloud curve of rank $n'$.

Recall that each point on the dual curve ${\mathcal{FB}}^{\vee}$ corresponds to a tangent
line of $\mathcal{FB}$. Then we have
\begin{enumerate}
\item The point
$(\chi^\vee,\kappa^\vee)=(-1,0)\in \mathcal{FB}^{\vee}$, with corresponding $u$'s poles of $V(u)$,  corresponds to the tangent line $\chi=1$ of $\mathcal{FB}$. The number of tangent points of $\chi=1$ to $\mathcal{FB}$, is the number of poles of $V(u)$; given by (\ref{tn1}).
\item The point
$\chi^\vee=\infty\in \mathcal{FB}^{\vee}$, with corresponding $u$'s poles of $U(u)$,  corresponds to the tangent line $\chi=0$ of $\mathcal{FB}$. The number of tangent points of $\chi=0$ to $\mathcal{FB}$, is the number of poles of $U(u)$ given by (\ref{tn0}).
\item If $u_i$ satisfies (\ref{uvci})
Then the point $(0,c_i)\in \mathcal{FB}^{\vee}$ correspond to tangent line $c_i\kappa+1=0$ of $\mathcal{FB}$.
\end{enumerate}
\end{proof}

\begin{example}\label{e39}Consider the case when
\begin{eqnarray*}
M=1,\qquad m=1,\qquad n_1=3,\qquad V_0=0,\qquad V_1=1.
\end{eqnarray*}
and
\begin{eqnarray*}
(a_1,b_1)=(L,-),\qquad (a_2,b_2)=(R,+),\qquad (a_3,b_3)=(L,+).
\end{eqnarray*}
and
\begin{eqnarray*}
x_1^{(1)}=\frac{1}{3},\qquad x_2^{(1)}=\frac{1}{2},\qquad x_3^{(1)}=1;
\end{eqnarray*}
then
\begin{align*}
\phi_{1,\chi}=\frac{1-\chi}{3}.
\end{align*}
 the frozen boundary has parametric equation given by
\begin{eqnarray*}
\chi(u)=\frac{\frac{1}{(3u-1)^2}}
{\frac{1}{(3u-1)^2}+\frac{2}{3(u+2)^2}+\frac{1}{3(1-u)^2}}
\end{eqnarray*}
and
\begin{small}
\begin{eqnarray*}
&&\kappa(u)=(1-\chi(u))\frac{1}{3(3u-1)}+\chi(u)\left[\frac{1}{3}\frac{u}{u+2}+\frac{1}{3}\frac{u}{1-u}\right]+1.
\end{eqnarray*}
\end{small}
See Figure \ref{fig:e39}.
\begin{figure}
\includegraphics[width=.8\textwidth]{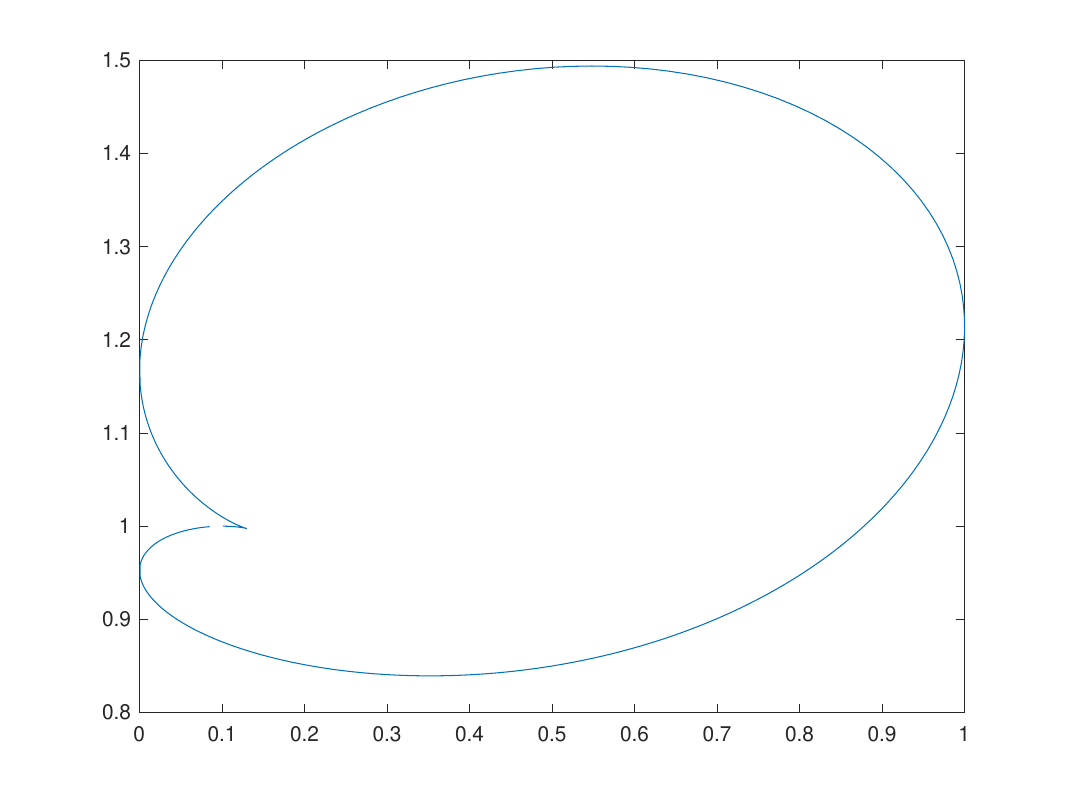}
\caption{Frozen boundary of Example \ref{e39}}\label{fig:e39}
\end{figure}
\end{example}

\subsection{The case $M=1$ and $m>1$.}
\begin{example}\label{e310}Consider the case when
\begin{eqnarray*}
M=1,\qquad m=2,\qquad n_1=2,\qquad n_2=3;
\end{eqnarray*}
and
\begin{eqnarray*}
\qquad V_0=0,\qquad V_1=0.3,\qquad V_2=1.
\end{eqnarray*}
and
\begin{eqnarray*}
(a_1^{(1)},b_1^{(1)})=(L,-),\qquad (a_2^{(1)},b_2^{(1)})=(R,+).
\end{eqnarray*}
and
\begin{eqnarray*}
(a_1^{(2)},b_1^{(2)})=(L,-),\qquad (a_2^{(2)},b_2^{(2)})=(R,+),\qquad (a_3^{(2)},b_3^{(2)})=(L,+).
\end{eqnarray*}
and
\begin{eqnarray*}
x_1^{(1)}=\frac{1}{3},\qquad x_2^{(1)}=\frac{1}{2},\qquad x_1^{(2)}=1;\qquad x_2^{(2)}=\frac{1}{6},\qquad x_3^{(2)}=\frac{1}{5},
\end{eqnarray*}
Then
\begin{eqnarray*}
\chi=\begin{cases}0.3\alpha_t,&\mathrm{if}\ \chi\in(0,0.3)\\
0.3+0.7\alpha_t,&\mathrm{if}\ \chi\in(0.3,1).
\end{cases}
\end{eqnarray*}
and
\begin{align*}
\phi_{p_t,\alpha_t}=\begin{cases}\frac{23}{60}-\frac{\chi}{2};&\mathrm{if}\ \chi\in(0,0.3)\\ \frac{1-\chi}{3}.&\mathrm{if}\ \chi\in(0.3,1)\end{cases}
\end{align*}
Then the frozen boundary has parametric equation given by 
\begin{enumerate}
\item If $\chi\in(0,0.3)$, then
\begin{eqnarray*}
\chi(u)&=&\frac{\frac{7}{30(u-1)^2}+\frac{9}{20(3u-1)^2}}{\frac{3}{2(3u-1)^2}+\frac{1}{(2+u)^2}};\\
\kappa(u)&=&\frac{7}{30}\frac{u}{u-1}+\frac{9}{20}\frac{u}{3u-1}+\left(\frac{u}{2(2+u)}-\frac{3u}{2(3u-1)}\right)\chi+\frac{37}{60}+\frac{\chi}{2}.
\end{eqnarray*}
\item If $\chi\in(0.3,1)$, then
\begin{small}
\begin{align*}
\chi(u)&=\frac{\frac{1}{3(u-1)^2}-\frac{3}{10(2+u)^2}+\frac{3}{5(6+u)^2}+\frac{1}{2(5-u)^2}}{\frac{1}{3(u-1)^2}+\frac{2}{(6+u)^2}+\frac{5}{3(5-u)^2}}
\\
\kappa(u)&=\frac{1}{3}\frac{u}{u-1}+\frac{3}{20}\frac{u}{2+u}-\frac{u}{10(u+6)}-\frac{u}{10(5-u)}
+\left(\frac{u}{3(6+u)}+\frac{u}{3(5-u)}-\frac{u}{3(u-1)}\right)\chi+\frac{2}{3}+\frac{\chi}{3}
\end{align*}
\end{small}
\end{enumerate}
See Figure \ref{fig:ex37} for the graph of frozen boundary; and see Figure \ref{fig:ex37s} for a simulation using the Markov Chain Monte Carlo.
\begin{figure}
\includegraphics[width=.8\textwidth]{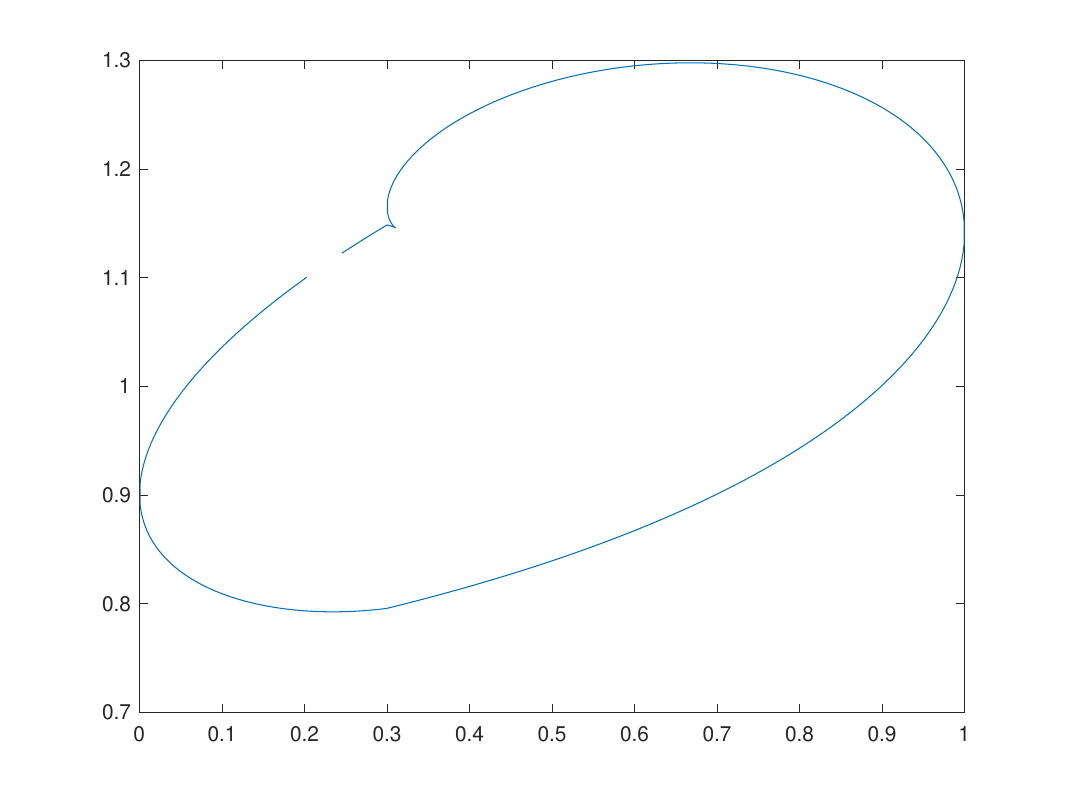}
\caption{Frozen boundary of Example \ref{e310}}\label{fig:ex37}
\end{figure}
\begin{figure}
\includegraphics[width=.8\textwidth]{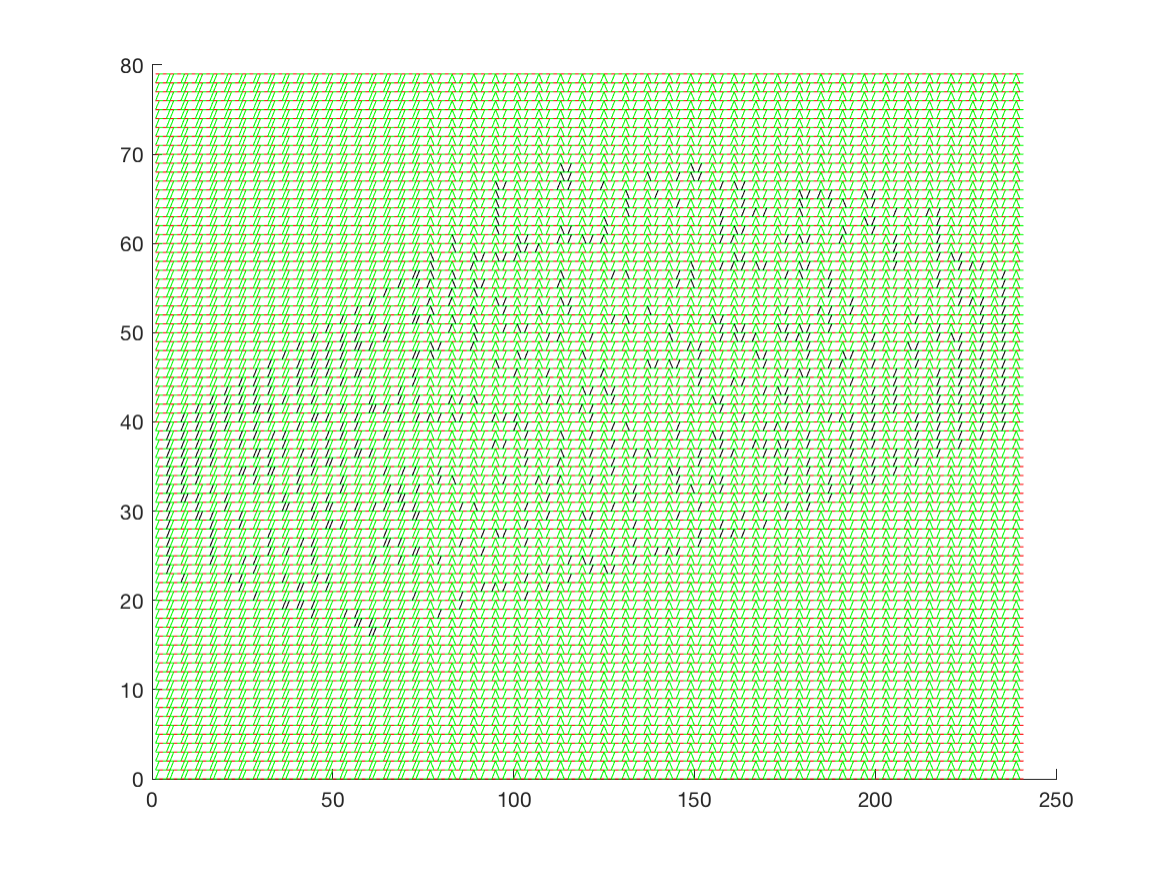}
\caption{Simulation of dimer coverings in Example \ref{e310}. Absent edges are represented by green lines; present horizontal edges are represented by red lines; present diagonal edges are represented by black lines.}\label{fig:ex37s}
\end{figure}
\end{example}

\subsection{The case $M=2$.}

When $M=2$, the equation (\ref{ceq}) becomes
\begin{small}
\begin{eqnarray}
\kappa-1+\phi_{p_t,\alpha_t}&=& \sum_{p=1}^m\sum_{s=1}^{n_p}\frac{(V_p-V_{p-1})\zeta_{L,-,s,p}}{V_m-V_0}\frac{u}{u+x_s^{(p)}}+H_{p_t,\alpha_t}(u)
\end{eqnarray}
\end{small}

The expression on the right hand side of (\ref{mk2}) has singular points given by (\ref{sgp1}) and
\begin{eqnarray*}
\{-x_s^{(p)}\}_{p\in[m],s\in[n_p],a_{s,p}=L,b_{s,p}=-}
\end{eqnarray*}

Again under Assumption \ref{ap14}, (1)(2)(3) holds as in the case $M=1$ with (\ref{m1k}) replaced by (\ref{mk2}); moreover,
\begin{enumerate}[label=(\alph*)]
\item Between each pair of nearest points in 
, (\ref{mk2}) has at least one real solution;
\item In the set
\begin{eqnarray*}
\left(-\infty,\min\left\{-\frac{1}{x_s^{(p)}}\right\}_{p\in[p_t],s\in[n_p],a_{s,p}=R,b_{s,p}=+}\right)\cup\left(\qquad
\max\left\{\frac{1}{x_s^{(p)}}\right\}_{p\in[p_t],s\in[n_p],a_{s,p}=L,b_{s,p}=+},\infty\right);
\end{eqnarray*}
(\ref{mk2}) has at least one real solution;
\end{enumerate}

Let $N_1$ be the degree of (\ref{mk2}); items (1)(2)(3)(a)(b) above give at least $N_1-3$ real solutions for (\ref{mk2}); hence (\ref{mk2}) has at most one pair of complex conjugate roots.

By Lemma \ref{l25}, the frozen boundary is given by the equation (\ref{m1k}) has double roots.
When (\ref{m1k}) has double real roots, we have
\begin{small}
\begin{eqnarray}
0&=& \sum_{p=1}^m\sum_{s=1}^{n_p}\frac{(V_p-V_{p-1})\zeta_{L,-,s,p}}{V_m-V_0}\frac{x_s^{(p)}}{\left(u+x_s^{(p)}\right)^2}+G_{p_t,\alpha_t}(u)\label{mk2}
\end{eqnarray}
\end{small}

\begin{example}\label{ep38}Consider the case when
\begin{eqnarray*}
M=2,\qquad m=1,\qquad n_1=3,\qquad V_0=0,\qquad V_1=1.
\end{eqnarray*}
and
\begin{eqnarray*}
(a_1,b_1)=(L,-),\qquad (a_2,b_2)=(R,+),\qquad (a_3,b_3)=(L,+).
\end{eqnarray*}
and
\begin{align*}
x_1^{(1)}=\frac{1}{3},\qquad x_2^{(1)}=\frac{1}{2},\qquad x_3^{(1)}=1;
\end{align*}
\begin{align*}
\phi_{1,\chi}=\frac{1-\chi}{3}
\end{align*}
then the frozen boundary has parametric equation given by
\begin{eqnarray*}
\chi(u)=\frac{\frac{1}{(3u-1)^2}-\frac{1}{(3u+1)^2}}
{\frac{1}{(3u-1)^2}+\frac{2}{3(u+2)^2}+\frac{1}{3(1-u)^2}}
\end{eqnarray*}
and
\begin{small}
\begin{eqnarray*}
&&\kappa(u)=(1-\chi(u))\frac{u}{3u-1}+\chi(u)\left[\frac{1}{3}\frac{u}{u+2}+\frac{1}{3}\frac{u}{1-u}\right]+\frac{u}{3u+1}+\frac{2}{3}+\frac{\chi(u)}{3}.
\end{eqnarray*}
\end{small}
See Figure \ref{fig:e311} for the frozen boundary, and Figure \ref{fig:e311sim} for a simulation.
\begin{figure}
\includegraphics[width=.8\textwidth]{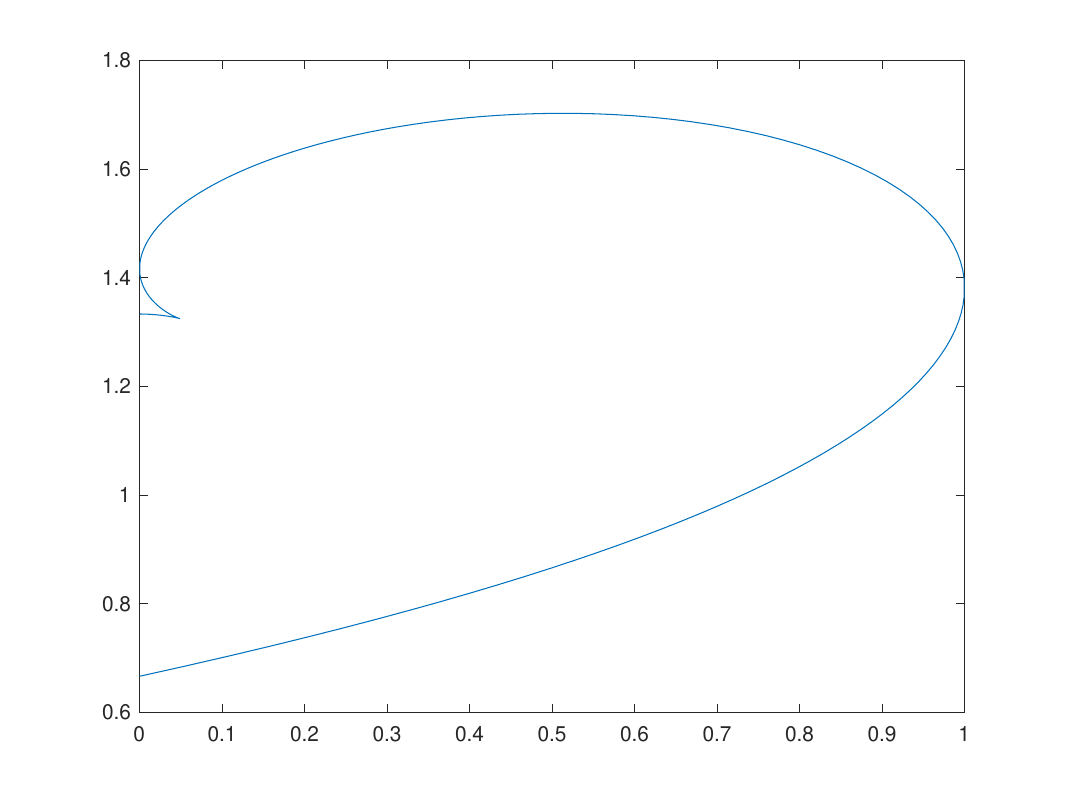}
\caption{Frozen boundary of Example \ref{ep38}}\label{fig:e311}
\end{figure}
\begin{figure}
\includegraphics[width=.8\textwidth]{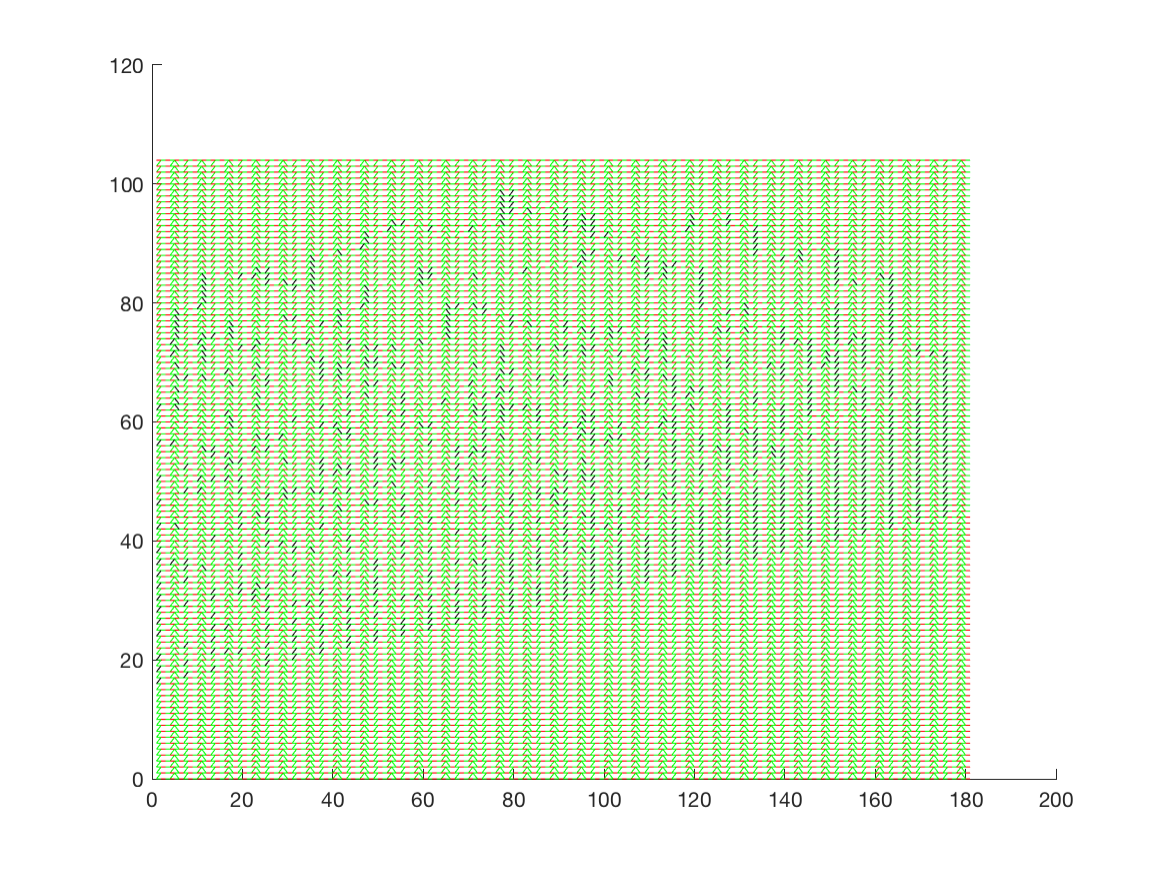}
\caption{Simulation of dimer coverings of Example \ref{ep38}: absent edges are represented by green lines, present horizontal edges are represented by red lines, present diagonal edges are represented by black lines.}\label{fig:e311sim}
\end{figure}
\end{example}

\section{Limit shape for piecewise boundary condition}\label{sect:pw}

In this section,  we obtain the limit shape of dimer coverings on a rail-yard graph whose left boundary condition is given by a piecewise partition with densities either 1 or 0. The main theorem proved in this section is Theorem \ref{t412}.

\subsection{A formula to compute Schur functions}

Let $\lambda(N)\in \YY_N$. Let $\Sigma_N$ be the permutation group of $N$ elements and let $\sigma\in \Sigma_N$. Let 
\begin{eqnarray*}
X=(x_1,\ldots,x_N).
\end{eqnarray*}
Assume that there exists a positive integer $n\in[N]$ such that $x_1,...,x_n$ are pairwise distinct and $\{x_1,...,x_n\}=\{x_1,...,x_N\}$. For $j\in[N]$, let
\begin{eqnarray}
\eta_j^{\sigma}(N)=|\{k:k>j,x_{\sigma(k)}\neq x_{\sigma(j)}\}|.\label{et}
\end{eqnarray}
For $1\leq i\leq n$, let
\begin{eqnarray}
\Phi^{(i,\sigma)}(N)=\{\lambda_j(N)+\eta_j^{\sigma}(N):x_{\sigma(j)}=x_i\}\label{pis}
\end{eqnarray}
and let $\phi^{(i,\sigma)}(N)$ be the partition with length $|\{1\leq j\leq N: x_j=x_i\}|$ obtained by decreasingly ordering all the elements in $\Phi^{(i,\sigma)}(N)$. Let $\Sigma_N^{X}$ be the subgroup $\Sigma_N$ that preserves the value of $X$; more precisely
\begin{eqnarray*}
\Sigma_N^{X}=\{\sigma\in \Sigma_N: x_{\sigma(i)}=x_i,\ \mathrm{for}\ i\in[N]\}.
\end{eqnarray*}
Let $[\Sigma/\Sigma_N^X]^r$ be the collection of all the right cosets of $\Sigma_N^X$ in $\Sigma_N$. More precisely,
\begin{eqnarray*}
[\Sigma/\Sigma_N^X]^r=\{\Sigma_N^X\sigma:\sigma\in \Sigma_N\},
\end{eqnarray*}
where for each $\sigma\in \Sigma_N$
\begin{eqnarray*}
\Sigma_N^X\sigma=\{\xi\sigma:\xi\in \Sigma_N^X\}
\end{eqnarray*}
and $\xi\sigma\in \Sigma_N$ is defined by
\begin{eqnarray*}
\xi\sigma(k)=\xi(\sigma(k)),\ \mathrm{for}\ k\in[N].
\end{eqnarray*}

Let $k\in[N]$. Let
\begin{eqnarray}
\label{dw}w_{i}=\left\{\begin{array}{cc}u_i&\mathrm{if}\ 1\leq i\leq k\\x_i&\mathrm{if}\ k+1\leq i\leq N\end{array}\right.
\end{eqnarray}


\begin{proposition}\label{p437}Let $\{w_i\}_{i\in[N]}$ be given by (\ref{dw}). 
Define
\begin{eqnarray*}
\underline{u}^{(i)}=\left\{\frac{u_j}{x_i}:x_j=x_i\right\}
\end{eqnarray*}
Then we have the following formula
\begin{eqnarray}
&&\label{sws}s_{\lambda}(w_1,\ldots,w_N)\\
&=&\sum_{\ol{\sigma}\in[\Sigma_N/\Sigma_N^X]^r} \left(\prod_{i=1}^{n}x_i^{|\phi^{(i,\sigma)}(N)|}\right)\left(\prod_{i=1}^{n}s_{\phi^{(i,\sigma)}(N)}\left(\underline{u}^{(i)},1,\ldots,1\right)\right)\notag\\
&&\times\left(\prod_{i<j,x_{\sigma(i)}\neq x_{\sigma(j)}}\frac{1}{w_{\sigma(i)}-w_{\sigma(j)}}\right)\notag
\end{eqnarray}
where $\sigma\in \ol{\sigma}\cap \Sigma_N$ is a representative. 
\end{proposition}

\begin{proof}The proposition is a general version of Proposition 3.4 in \cite{ZL18}, where it is required that for each $i\in[n]$ the number of items in $\{x_1,\ldots,x_N\}$ that are equal to $x_i$ is exactly $\frac{N}{n}$. This proposition allows that the number of items equal to different $x_i$'s are different, however, the proof follows from the same arguments as the proof of Proposition 3.4 in \cite{ZL18}.
\end{proof}

\begin{corollary}\label{p438}
\begin{eqnarray}
&&\label{sws1}s_{\lambda}(x_1,\ldots,x_N)\\
&=&\sum_{\ol{\sigma}\in[\Sigma_N/\Sigma_N^X]^r} \left(\prod_{i=1}^{n}x_i^{|\phi^{(i,\sigma)}(N)|}\right)\left(\prod_{i=1}^{n}s_{\phi^{(i,\sigma)}(N)}\left(1,\ldots,1\right)\right)\left(\prod_{i<j,x_{\sigma(i)}\neq x_{\sigma(j)}}\frac{1}{x_{\sigma(i)}-x_{\sigma(j)}}\right)\notag
\end{eqnarray}
where $\sigma\in \ol{\sigma}\cap \Sigma_N$ is a representative. 
\end{corollary}

\begin{proof}The corollary follows from Proposition \ref{p437} by letting $u_i=x_i$ for all $i\in[k]$.
\end{proof}

\subsection{Limit shape of rail-yard graphs with piecewise boundary condition}

Suppose that Assumption \ref{ap41} holds. For each $\epsilon>0$, define
\begin{eqnarray*}
N_{L,-}^{(\epsilon)}:=\left|\left\{j\in\left[l^{(\epsilon)}..r^{(\epsilon)}\right]:a_j^{(\epsilon)}=L,\ b_j^{(\epsilon)}=-\right\}\right|.
\end{eqnarray*}

Let $\pi$ be the bijection from $\left[N_{L,-}^{(\epsilon)}\right]$ to $\{j\in[l^{(\epsilon)}..r^{(\epsilon)}]:a_j^{(\epsilon)}=L,b_j^{(\epsilon)}=-\}$, such that
\begin{itemize}
    \item for any $i,j\in\left[N_{L,-}^{(\epsilon)}\right]$ and $i<j$, $\pi(i)<\pi(j)$.
\end{itemize}

Let $\Sigma_{L,-,\epsilon}$ be the group consisting of all the bijections from $\{j\in[l^{(\epsilon)}..r^{(\epsilon)}]:a_j^{(\epsilon)}=L,b_j^{(\epsilon)}=-\}$ to itself.

Let $\Sigma_{L,-,\epsilon}^X$ be a subgroup of $\Sigma_{L,-,\epsilon}$ preserving the values of $x_i$'s, i.e.
\begin{eqnarray*}
\Sigma_{L,-,\epsilon}^X:=\{\si\in \Sigma_{L,-,\epsilon}:x_{\si(i)}=x_i,\ \forall\ i\in\{j\in[l^{(\epsilon)}..r^{(\epsilon)}]:a_j^{(\epsilon)}=L,b_j^{(\epsilon)}=-\}\}.
\end{eqnarray*}

Note that $\pi\Sigma_{L,-,\epsilon}$, defined by
\begin{eqnarray*}
\pi\Sigma_{L,-,\epsilon}:=\{\pi\xi:\xi\in \Sigma_{L,-,\epsilon}\},
\end{eqnarray*}
consists of all the bijections from $\left[N_{L,-}^{(\epsilon)}\right]$ to $\{j\in[l^{(\epsilon)}..r^{(\epsilon)}]:a_j^{(\epsilon)}=L,b_j^{(\epsilon)}=-\}$.
Let $\sigma_0\in\pi\Sigma_{L,-,\epsilon}$ such that
\begin{eqnarray}
x_{\si_0(1)}\geq x_{\si_0(2)}\geq\ldots\geq x_{\si_0\left(N_{L,-}^{(\epsilon)}\right)}.\label{sz}
\end{eqnarray}

\begin{assumption}\label{ap428}Suppose Assumption \ref{ap41} holds. Let $s\in\left[N_{L,-}^{(\epsilon)}\right]$ such that $s$ is independent of $\epsilon$ as $\epsilon\rightarrow 0$.
Assume there exist positive integers $K_1^{(\epsilon)},K_2^{(\epsilon)},\ldots K_s^{(\epsilon)}$, such that 
\begin{enumerate}
\item $\sum_{t=1}^s K_t^{(\epsilon)}=N_{L,-}^{(\epsilon)}$;
\item Assume the partition corresponding to the left boundary condition is $\lambda^{(l,\epsilon)}$ is given by
\begin{eqnarray*}
\lambda^{(l,\epsilon)}=\left(\lambda_1,\ldots,\lambda_{N_{L,-}^{(\epsilon)}}\right).
\end{eqnarray*}
such that
\begin{eqnarray}
\mu_1^{(\epsilon)}>\ldots>\mu_s^{(\epsilon)}\label{mi}
\end{eqnarray}
are all the distinct elements in $\left\{\lambda_1,\lambda_2,\ldots,\lambda_{N_{L,-}^{(\epsilon)}}\right\}$.
\item
\begin{eqnarray*}
&&\lambda_1=\lambda_2=\ldots=\lambda_{K_s^{(\epsilon)}}=\mu_1^{(\epsilon)};\\
&&\lambda_{K_s^{(\epsilon)}+1}=\lambda_{K_s^{(\epsilon)}+2}=\ldots=\lambda_{K_s^{(\epsilon)}+K_{s-1}^{(\epsilon)}}=\mu_2^{(\epsilon)};\\
&&\ldots\\
&&\lambda_{1+\sum_{t=2}^{s}K_t^{(\epsilon)}}=\lambda_{2+\sum_{t=2}^{s}K_t^{(\epsilon)}}=\ldots=\lambda_{\sum_{t=1}^{s}K_t^{(\epsilon)}}=\mu_s^{(\epsilon)};
\end{eqnarray*}
 \item For $p\in[m]$ and $i\in[n_p]$, let 
\begin{eqnarray}
J_{p,i}=\{t\in[s]:\exists q\in\left[N_{L,-}^{(\epsilon)}\right],\ \mathrm{s.t.}\ x_{\sigma_0(q)}=x_i^{(p)},\mathrm{and}\ \lambda_q=\mu_t^{(\epsilon)}\}\label{ji}
\end{eqnarray}
such that
\begin{enumerate}
\item For any $p_1,p_2\in[m]$ and $i_1\in[n_{p_1}]$, $i_2\in [n_{p_2}]$ with $x_{i_1}^{(p_1)}> x_{i_2}^{(p_2)}$, 
if $\ell\in J_{p_1,i_1}$, and $t\in J_{p_2,i_2}$, then $\ell<t$.
\item For any $p,q$ satisfying $1\leq p\leq s$ and $1\leq q\leq s$, and $q>p$
\begin{eqnarray*}
C_1N_{L,-}^{(\epsilon)} \leq \mu_p^{(\epsilon)}-\mu_q^{(\epsilon)}\leq C_2N_{L,-}^{(\epsilon)}
\end{eqnarray*}
where $C_1$, $C_2$ are sufficiently large constants independent of $\epsilon$ as $\epsilon\rightarrow 0$.
\end{enumerate}
\end{enumerate}
\end{assumption}

\begin{assumption}\label{ap32}Suppose that Assumption \ref{ap41} holds. For $i\in[l^{(\epsilon)}..r^{(\epsilon)}]$, suppose the weights $x_i$ change with respect to $\epsilon$, and denote it by $x_{i,\epsilon}$. Let $\sigma_0$ be defined as in (\ref{sz}). More precisely, we order elements in the set
\begin{eqnarray}
\{x_{i,\epsilon}\}_{p\in[m],i\in[n_p],a_{i,p}=L,b_{i,p}=-}\label{wl-}
\end{eqnarray}
by
\begin{eqnarray*}
x_{\sigma_0(1),\epsilon}\geq x_{\sigma_0(2),\epsilon}\geq \ldots\geq x_{\sigma_0(N_{L,-}^{(\epsilon)}),\epsilon}
\end{eqnarray*}
Assume 
\begin{eqnarray*}
x_{\sigma_0(1),\epsilon}=x_{\sigma_0(1)}>0
\end{eqnarray*}
i.e., the maximal weight in (\ref{wl-}) is a positive constant independent of $\epsilon$.

Suppose that Assumption \ref{ap428} holds. Moreover, assume that
 \begin{eqnarray*}
\liminf_{\epsilon\rightarrow 0} \epsilon\log\left(\min_{x_{i,\epsilon}>x_{j,\epsilon}}\frac{x_{i,\epsilon}}{x_{j,\epsilon}}\right)\geq \alpha>0,
\end{eqnarray*}
where $\alpha$ is a sufficiently large positive constant independent of $\epsilon$.
\end{assumption}

It is convenient to define the index sets $J_i$'s as follows.
\begin{definition}\label{df34}For $p\in[m]$ and $i\in[n_p]$, let $J_{p,i}$ be defined as in (\ref{ji}). Assume
\begin{eqnarray*}
\{J_{p,i}\}_{p\in[m],i\in[n_p],a_{i,p}=L,b_{i,p}=-}=\{J_k\}_{k\in [I]}
\end{eqnarray*}
where $I\in \NN$ is the total number of distinct elements in $\{J_{p,i}\}_{p\in[m],i\in[n_p]}$, such that 
\begin{itemize}
    \item If $J_{k_1}=J_{p_1,i_1}$, $J_{k_2}=J_{p_2,i_2}$ and $k_1<k_2$, then
    \begin{eqnarray*}
    x_{i_1}^{(p_1)}>x_{i_2}^{(p_2)}.
    \end{eqnarray*}
\end{itemize}
In other words, we order all the distinct index set in $\{J_{p,i}\}_{p\in[m],i\in[n_p],a_{i,p}=L,b_{i,p}=-}$ by the corresponding $x_i^{(p)}$'s, and then relabel them from the largest $x_i^{(p)}$ to the least $x_i^{(p)}$. This way we also define a mapping $\psi$ from $\{(p,i)\}_{p\in[m],i\in[n_p],a_{i,p}=L,b_{i,p}=-}$ to $[I]$ such that
\begin{eqnarray*}
\psi(p,i)=k
\end{eqnarray*}
if and only if
\begin{eqnarray*}
J_{p,i}=J_k.
\end{eqnarray*}
Note that the mapping $\psi$ is surjective but not necessarily injective. In particular if there exist $a_{i_1,p_1}=a_{i_2,p_2}=L$, $b_{i_1,p_1}=b_{i_2,p_2}=-$, $(p_1,i_1)\neq (p_2,i_2)$ but $x_{i_1}^{(p_1)}=x_{i_2}^{(p_2)}$, then $\psi(p_1,i_1)=\psi(p_2,i_2)$.
\end{definition}

For $N_{L,-}^{(\epsilon)}\geq 1$, let $\lambda^{(l)}\in \YY_{N_{L,-}^{(\epsilon)}}$ be the left boundary partition satisfying Assumption \ref{ap428}.  Let
\begin{eqnarray*}
\Omega=\left(\Omega_1<\Omega_2<\ldots<\Omega_{N_{L,-}^{(\epsilon)}}\right)=\left(\lambda_{N_{L,-}^{(\epsilon)}}^{(l)}-N_{L,-}^{(\epsilon)},\lambda_{{N_{L,-}^{(\epsilon)}}-1}^{(l)}-N_{L,-}^{(\epsilon)}+1,\ldots,\lambda_1^{(l)}-1\right)
\end{eqnarray*}
Indeed, $\Omega_1,\ldots,\Omega_{N_{L,-}^{(\epsilon)}}$ are the  locations of the $N_{L,-}^{(\epsilon)}$ topmost particles on the left boundary of the rail-yard graph, assuming that the vertical coordinate of the $N_{L,-}^{(\epsilon)}$ highest particles is $-N_{L,-}^{(\epsilon)}$ when the left boundary condition is given by the empty partition.
Under Assumption \ref{ap428}, we may assume
\begin{eqnarray}
\Omega&=&(A_1,A_1+1,\ldots, B_1-1,B_1,\label{abt}\\
&&A_2,A_2+1,\ldots,B_2-1,B_2,\ldots,A_s,A_s+1,\ldots,B_s-1,B_s).\notag
\end{eqnarray}
where 
\begin{eqnarray*}
B_i-A_i+1=K_i.
\end{eqnarray*}
and
\begin{eqnarray*}
\sum_{i=1}^{s}(B_i-A_i+1)=N_{L,-}^{(\epsilon)}.
\end{eqnarray*}
Suppose as $\epsilon\rightarrow 0$,
\begin{eqnarray}
A_i=a_iN_{L,-}^{(\epsilon)}+o(N_{L,-}^{(\epsilon)}),\qquad B_i=b_iN_{L,-}^{(\epsilon)}+o(N_{L,-}^{(\epsilon)}),\qquad \mathrm{for}\ i\in[s],\label{abi}
\end{eqnarray}
and $a_1<b_1<\ldots<a_s<b_s$ are fixed parameters independent of $\epsilon$ and satisfy $\sum_{i=1}^{s}(b_i-a_i)=1$. Under Assumption \ref{ap428}, it is straightforward to check that for $i\in[s]$
\begin{eqnarray*}
b_i&=&\lim_{N\rightarrow\infty}\frac{\mu_{s-i+1}+\sum_{t=1}^{i}K_t}{N_{L,-}^{(\epsilon)}}-1\\
a_i&=&\lim_{N\rightarrow\infty}\frac{\mu_{s-i+1}+\sum_{t=1}^{i-1}K_t}{N_{L,-}^{(\epsilon)}}-1
\end{eqnarray*}

Then we have the following lemmas.

\begin{lemma}\label{cm1}Let $\{J_i\}_{i\in[I]}$ be defined as in Definition \ref{df34}.
Under Assumption \ref{ap428}, for $i\in[I]$, we have
\begin{eqnarray*}
\phi^{(i,\sigma_0)}\left(N_{L,-}^{(\epsilon)}\right)\in \YY_{N_{L,-,i}^{(\epsilon)}}
\end{eqnarray*}
where
\begin{eqnarray*}
N_{L,-,i}^{(\epsilon)}&=&\sum_{(p,j)\in \psi^{-1}(i)}\frac{v_p^{(\epsilon)}-v_{p-1}^{(\epsilon)}}{n_p}\\
&=&|k\in[l^{(\epsilon)}..r^{(\epsilon)}]:a_k=L,b_k=-,x_k=x_j^{(p)}\ \mathrm{for\ some}\ p\in[m],j\in[n_p],\psi(p,j)=i|
\end{eqnarray*}
\end{lemma}
Note that
\begin{eqnarray*}
\sum_{i\in[I]}N_{L,-,i}^{(\epsilon)}=N_{L,-}^{(\epsilon)}.
\end{eqnarray*}
\begin{assumption}\label{ap36}
Suppose that for each $i\in[I]$,
\begin{eqnarray*}
\lim_{\epsilon\rightarrow 0}\frac{N_{L,-,i}^{(\epsilon)}}{N_{L,-}^{(\epsilon)}}=\theta_i;\qquad\mathrm{and}\qquad 
\lim_{\epsilon\rightarrow 0}\frac{N_{L,-}^{(\epsilon)}}{N^{(\epsilon)}}=\rho;
\end{eqnarray*}
where $N^{(\epsilon)}=|r^{(\epsilon)}-l^{(\epsilon)}|.$
\end{assumption}

\begin{lemma}\label{cm}Let $\{J_i\}_{i\in[I]}$ be defined as in Definition \ref{df34}.
Under Assumptions \ref{ap428} \ref{ap36}, assume that 
\begin{eqnarray*}
J_i=\left\{\begin{array}{cc}\{d_i,d_i+1,\ldots,d_{i+1}-1\}& \mathrm{if}\ 1\leq i\leq I-1\\ \{d_I, d_I+1,\ldots,s\}& \mathrm{if}\ i=I \end{array}\right.
\end{eqnarray*}
where $d_1,\ldots,d_I$ are positive integers satisfying
\begin{eqnarray*}
1=d_1<d_2<\ldots <d_I\leq s
\end{eqnarray*}
Let
\begin{eqnarray*}
d_{I+1}:=s+1.
\end{eqnarray*}
For $j\in[s]$, let $a_j,b_j$ be given by (\ref{abi}). 
If $i\in[I]$, for $0\leq k\leq d_{i+1}-d_i-1$, let
\begin{small}
\begin{eqnarray}
\beta_{i,k}
&=&\frac{a_{s-d_i-k+1}-a_1}{\theta_i}\label{dbik}
\end{eqnarray}
and let
\begin{eqnarray}
\gamma_{i,k}&=&\beta_{i,k}+\frac{b_{s-d_i-k+1}-a_{s-d_i-k+1}}{\theta_i}
=\frac{b_{s-d_i-k+1}-a_1}{\theta_i}\label{dcik}
\end{eqnarray}
\end{small}

Then for $i\in[I]$ the counting measures of  $\phi^{(i,\si_0)}\left(N_{L,-}^{(\epsilon)}\right)$ converge weakly to a limit measure $\bm_i$ as $\epsilon\rightarrow 0$. Moreover,
if $i\in[I]$, for $0\leq k\leq d_{i+1}-d_i-1$, the limit counting measure $\mathbf{m}_i$ is a probability measure on $[\beta_{i,1},\gamma_{i,d_{i+1}-d_i-1}]$ with density given by
\begin{eqnarray*}
\frac{d\mathbf{m}_i}{dy}=\left\{\begin{array}{cc}1,& \mathrm{if}\ \beta_{i,k}< y< \gamma_{i,k};\\0,& \mathrm{if}\ \gamma_{i,k}\leq y\leq \beta_{i,k+1}. \end{array}\right.
\end{eqnarray*}
\end{lemma}

As we have seen in Lemma \ref{cm}, under Assumption \ref{ap428}, as $\epsilon\rightarrow 0$, the counting measures of  $\phi^{(i,\si_0)}(N_{L,-}^{(\epsilon)})$ converges weakly to a limit measure $\bm_i$.

Let
\begin{eqnarray}
H_{\mathbf{m}_i}(u)=\int_{0}^{\ln(u)}R_{\mathbf{m}_i}(t)dt+\ln\left(\frac{\ln(u)}{u-1}\right)\label{hmi}
\end{eqnarray}
and $\mathbf{R}_{\mathbf{m}_i}$ is the Voiculescu R-transform of $\mathbf{m}_i$ given by
\begin{eqnarray*}
R_{\bm_i}=\frac{1}{S_{\bm_i}^{(-1)}(z)}-\frac{1}{z};
\end{eqnarray*}
Where $S_{\bm_i}$ is the moment generating function for $\bm_i$ given by
\begin{eqnarray}
S_{\bm_i}(z)=z+M_1(\bm_i)z^2+M_2(\bm_i)z^3+\ldots;\label{smi}
\end{eqnarray}
$M_k(\bm_i)=\int_{\RR}x^k\bm_i(dx)$; and $S_{\bm_i}^{-1}(z)$ is the inverse series of $S_{\bm_i}(z)$.

\begin{proposition}\label{p436}Suppose Assumptions \ref{ap32}, \ref{ap428} and \ref{ap36} hold, and let $\alpha$ be given as in Assumption \ref{ap32}. For each given $\{a_i,b_i\}_{i=1}^{n}$, when $\alpha$ is sufficiently large, for any $\sigma\notin \ol{\si}_0$ we have
\begin{eqnarray}
\label{gep}&&\left|\frac{\left(\prod_{i=1}^{n}x_{i,\epsilon}^{|\phi^{(i,\sigma_0)}(N_{L,-}^{(\epsilon)})|}\right)\left(\prod_{i=1}^{n}s_{\phi^{(i,\sigma_0)}(N_{L,-}^{(\epsilon)})}(1,\ldots,1)\right)}{\left(\prod_{i=1}^{n}x_{i,\epsilon}^{|\phi^{(i,\sigma)}(N_{L,-}^{(\epsilon)})|}\right)\left(\prod_{i=1}^{n}s_{\phi^{(i,\sigma)}(N_{L,-}^{(\epsilon)})}(1,\ldots,1)\right)
}\right|\\
&&\times\left|\frac{\left(\prod_{i<j,x_{\sigma_0(i),\epsilon}\neq x_{\sigma_0(j),\epsilon}}\frac{1}{x_{\sigma_0(i),\epsilon}-x_{\sigma_0(j),\epsilon}}\right)}{\left(\prod_{i<j,x_{\sigma(i),\epsilon}\neq x_{\sigma(j),\epsilon}}\frac{1}{x_{\sigma(i),\epsilon}-x_{\sigma(j),\epsilon}}\right)}\right|\notag\geq e^{C\epsilon^{-2}}\notag
\end{eqnarray}
where $C>0$ is a constant independent of $\epsilon$ and $\si$, and increases as $\alpha$ increases. Indeed, we have
\begin{eqnarray*}
\lim_{\alpha\rightarrow\infty} C=\infty.
\end{eqnarray*}
\end{proposition}

\begin{lemma}\label{l410}Let
 \begin{eqnarray*}
 K\subset\{i\in[l^{(\epsilon)}..r^{(\epsilon)}]:a_i^{(\epsilon)}=L,b_i^{(\epsilon)}=-\}\ \mathrm{s.t.}\ |K|=k.
 \end{eqnarray*} Let
\begin{eqnarray*}
\underline{ux}^{K,\epsilon}:=\{u_ix_{i,\epsilon}:i\in K\}
\end{eqnarray*}
Under Assumptions \ref{ap32}, \ref{ap428} and \ref{ap36}, for each given $\{a_i,b_i\}_{i=1}^{n}$, when $\alpha$ in Assumption \ref{ap32} is sufficiently large, each $u_i$ $i\in K$ are in an open complex neighborhood of $1$, we have
\begin{eqnarray}
&&\lim_{\epsilon\rightarrow 0}\frac{1}{N^{(\epsilon)}}\log \frac{s_{\lambda^{(l,\epsilon)}}(\underline{ux}^{K,\epsilon},\underline{x}^{(L,-)\setminus K,\epsilon})}{s_{\lambda^{(l,\epsilon)}}(\underline{x}^{(L,-),\epsilon})}=\sum_{i\in {K}}[Q_j(u_i)]\label{fc}
\end{eqnarray}
where for $i\in K$, 
\begin{itemize}
\item if $x_i=x_k^{(p)}$ such that $J_{k,p}=J_j$ for some $j\in[I]$, then 
\begin{eqnarray*}
Q_j(u)=\rho\theta_j H_{\mathbf{m}_{j}}(u)-\rho\left[\sum_{g=j+1}^{I}\theta_g\right]\log(u).
\end{eqnarray*}
\end{itemize}
Moreover, the convergence of (\ref{fc}) is uniform when $u_1,\ldots,u_k$ are in an open complex neighborhood of $1$.
\end{lemma}

\begin{proof}By Proposition \ref{p437}, both $s_{\lambda^{(l,\epsilon)}}(\underline{ux}^{K,\epsilon},\underline{x}^{(L,-)\setminus K,\epsilon})$ and $s_{\lambda^{(l,\epsilon)}}(\underline{x}^{(L,-),\epsilon})$ can be expressed as a sum of $\left|[\Si_{L,-,\epsilon}/\Si_{L,-,\epsilon}^X]^r\right|$ terms. Under Assumption \ref{ap32}, we have
\begin{eqnarray*}
\left|[\Si_{L,-,\epsilon}/\Si_{L,-,\epsilon}^X]^r\right|=\frac{\left(N_{L,-}^{(\epsilon)}\right)!}{\prod_{i\in [I]}(N_{L,-,i}^{(\epsilon)})!}
\end{eqnarray*}
By Stirling's formula and Assumption \ref{ap36} we obtain
\begin{eqnarray}
\lim_{\epsilon\rightarrow 0}\left|[\Si_{L,-,\epsilon}/\Si_{L,-,\epsilon}^X]^r\right|^{\frac{1}{N_{L,-}^{(\epsilon)}}}=\frac{1}{\prod_{i\in[I]}\left(\theta_i\right)^{\theta_i}}.\label{Nn}
\end{eqnarray}
By Proposition \ref{p436} and (\ref{Nn}),  when $\alpha$ in Assumption \ref{ap32} is sufficiently large,
\begin{eqnarray*}
&&s_{\lambda^{(l,\epsilon)}}(\underline{x}^{(L,-),\epsilon})
=\left(\prod_{i\in[I]}x_i^{|\phi^{(i,\sigma_0)}(N_{L,-}^{(\epsilon)})|}\right)\\
&&\left(\prod_{i\in[I]}s_{\phi^{(i,\sigma_0)}(N_{L,-}^{(\epsilon)})}(1,\ldots,1)\right)\left(\prod_{i,j\in[N_{(L,-)}^{(\epsilon)}], i<j,x_{\sigma_0(i)}\neq x_{\sigma_0(j)}}\frac{1}{x_{\sigma_0(i)}-x_{\sigma_0(j)}}\right)\left(1+e^{-C\epsilon^{-2}}\right)
\end{eqnarray*}
For $i\in[I]$, define
\begin{eqnarray*}
\underline{v}^{(i)}=\left\{u_j:j\in K,\ x_j=x_i\right\}
\end{eqnarray*}
and for $i\in [l^{(\epsilon)},r^{(\epsilon)}]$ satisfying $a_i^{(\epsilon)}=L$ and $b_i^{(\epsilon)}=-$, let
\begin{eqnarray*}
w_i:=\begin{cases}u_ix_i&\mathrm{if}\ i\in K;\\x_i&\mathrm{otherwise}. \end{cases}
\end{eqnarray*}
Then by Proposition \ref{p437},
\begin{eqnarray*}
&&s_{\lambda^{(l,\epsilon)}}(\underline{ux}^{K,\epsilon},\underline{x}^{(L,-)\setminus K,\epsilon})=\left(\prod_{i\in[I]}x_i^{|\phi^{(i,\sigma_0)}(N_{L,-}^{(\epsilon)})|}\right)\\
&&\left(\prod_{i\in[I]}s_{\phi^{(i,\sigma_0)}(N_{L,-}^{(\epsilon)})}(\underline{v}^{(i)},1,\ldots,1)\right)\left(\prod_{i,j\in\left[N_{(L,-)}^{(\epsilon)}\right],i<j,x_{\sigma_0(i)}\neq x_{\sigma_0(j)}}\frac{1}{w_{\sigma_0(i)}-w_{\sigma_0(j)}}\right)\left(1+e^{-C\epsilon^{-2}}\right)
\end{eqnarray*}
Then the lemma follows from explicit computations and Lemma \ref{p25}.
\end{proof}



\begin{lemma}\label{tm2}Assume (\ref{c151}) holds. Let $\rho_{\epsilon}^t$ be the probability measure for the partitions corresponding to dimer configurations adjacent to the column of odd vertices labeled by $2t-1$ in $RYG(l^{(\epsilon)},r^{(\epsilon)},\underline{a}^{(\epsilon)},\underline{b}^{(\epsilon)})$, conditional on the left and right boundary condition $\lambda^{(l)}$ and $\emptyset$, respectively. 
Let $K$, $\underline{ux}^{K,\epsilon}$ and $\underline{x}^{(L,-,>t)\setminus K}$ be defined as in Lemmas \ref{l410} and \ref{l22}, respectively.  Assume $t\in\left[v_{p_t-1}^{(\epsilon)},v_{p_t}^{(\epsilon)}\right]$ such that (\ref{lmt}) holds.
Under Assumptions \ref{ap32}, \ref{ap428} and \ref{ap36}, for each given $\{a_i,b_i\}_{i=1}^{n}$, when $\alpha$ in Assumption \ref{ap32} is sufficiently large, each $u_i$ $i\in K$ are in an open complex neighborhood of $1$, we have
\begin{eqnarray}
&&\lim_{\epsilon\rightarrow 0}\frac{1}{N^{(\epsilon)}}\log
\mathcal{S}_{\rho^t_{\epsilon},\underline{x}^{(L,-,>t)}}\left(\underline{ux}^{K,\epsilon},\underline{x}^{(L,-,>t)\setminus K,\epsilon}\right)=\sum_{i\in {K}}\left[Q_j\left(u_i\right)+R_{p_t,\alpha_t,j}(u_i)\right]\label{fc1}
\end{eqnarray}
where 
\begin{itemize}
\item If $x_i=x_k^{(p)}$ then $J_{k,p}=J_j$ for some $j\in[I]$;
\item If $x_{i,\epsilon}=x_{\sigma_0(1)}$
\begin{small}
\begin{eqnarray*}
R_{p_t,\alpha_t,i}(u)&=&
\sum_{p=1}^{p_t-1}\sum_{s=1}^{n_p}\frac{(V_p-V_{p-1})\zeta_{R,+,s,p}}{V_m-V_0}\log(1+ux_{\si_0(1)}x_{s}^{(p)})\\
&&+\sum_{s=1}^{n_{p_t}}\frac{\alpha_t(V_{p_t}-V_{{p_t}-1})\zeta_{R,+,s,p_t}}{V_m-V_0}\log(1+ux_{\si_0(1)}x_{s}^{(p_t)})\notag\\
&&-\sum_{p=1}^{p_t-1}\sum_{s=1}^{n_p}\frac{(V_p-V_{p-1})\zeta_{L,+,s,p}}{V_m-V_0}\log(1-ux_{\si_0(1)}x_{s}^{(p)})\\
&&-\sum_{s=1}^{n_{p_t}}\frac{\alpha_t(V_{p_t}-V_{{p_t}-1})\zeta_{L,+,s,p_t}}{V_m-V_0}\log(1-ux_{\si_0(1)}x_{s}^{(p_t)}).
\end{eqnarray*}
\end{small}
\item If $x_{i,\epsilon}<x_{\sigma_0(1)}$,
\begin{eqnarray*}
R_{p_t,\alpha_t,i}(u)=0.
\end{eqnarray*}
\item For each $i\in[I]$, $(p(i),j(i))\in \psi^{-1}(i)$.

\end{itemize}
Moreover, the convergence of (\ref{fc1}) is uniform when $u_1,\ldots,u_k$ are in an open complex neighborhood of $1$.
\end{lemma}

\begin{proof}
The lemma follows from Lemmas \ref{tm2} and \ref{lm212}.
\end{proof}

\begin{theorem}\label{t412}
Let $k\geq 1$ be a positive integer.
Let $\rho^t_{\epsilon}$ be the probability measure for the partitions corresponding to dimer configurations adjacent to the column of odd vertices labeled by $2t-1$ in $RYG(l^{(\epsilon)},r^{(\epsilon)},\underline{a}^{(\epsilon)},\underline{b}^{(\epsilon)})$, conditional on the left and right boundary condition $\lambda^{(l,\epsilon)}$ and $\emptyset$, respectively. Assume $t\in\left[v_{p_t-1}^{(\epsilon)},v_{p_t}^{(\epsilon)}\right]$ such that (\ref{lmt}) holds.
Moreover, assume that
\begin{eqnarray*}
\lim_{\epsilon\rightarrow 0}\epsilon t=\chi.
\end{eqnarray*}
and for $i\in[I]$
\begin{eqnarray*}
&&\lim_{\epsilon\rightarrow 0}
\frac{\left|j\in[t..r^{(\epsilon)}]:a_j^{(\epsilon)}=L,b_j^{(\epsilon)}=-,x_{j,\epsilon}=x_{k}^{(p)},\mathrm{s.t.}\ \psi(p,k)=i\right|}{N^{(\epsilon)}}=\gamma_i\\
&&\lim_{\epsilon\rightarrow 0}
\frac{\left|j\in[t..r^{(\epsilon)}]:a_j^{(\epsilon)}=L,b_j^{(\epsilon)}=-,x_{j,\epsilon}=x_{k}^{(p)},\mathrm{s.t.}\ \psi(p,k)>i.\right|}{N^{(\epsilon)}}=\eta_i
\end{eqnarray*}
Let $\mathbf{m}_{\rho^t_{\epsilon}}$ be the random counting measure on $\RR$ regarding a random partition $\lambda\in \YY_N$ with distribution $\rho_{\epsilon}^t$. Suppose that (\ref{c151}) Assumptions \ref{ap14}, \ref{ap41}, \ref{ap32}, \ref{ap428} and \ref{ap36} hold, for each given $\{a_i,b_i\}_{i=1}^{n}$, 
\begin{eqnarray*}
&&\lim_{\epsilon\rightarrow 0}\int_{\RR}x^{k}\textbf{m}_{\rho_{\epsilon}^t}(dx)=
\frac{1}{2(k+1)\pi \mathbf{i}}\\
&&\sum_{x_j^{(p)}\in S_{p_t}}
\oint_{C_{1}}\frac{dz}{z}\left[\frac{z\left(Q_{\psi(p,j)}'\left(z\right)+R'_{p_t,\alpha_t,\psi(p,j)}(z)\right)+U_{p_t,\alpha_t,\psi(p,j)}(z)}{\phi_{p_t,\alpha_t}}\right]^{k+1},
\end{eqnarray*}
where 
\begin{itemize}
\item 
\begin{eqnarray*}
U_{p_t,\alpha_t,i}(z)&=&\frac{\gamma_i z}{z-1}+\eta_i
\end{eqnarray*}
\item 
\begin{eqnarray*}
S_{p_t}:=\{x_j^{(p)}, \mathrm{s.t.}\ p\in[p_t..m],j\in[n_p]:a_{j,p}=L,b_{j,p}=-\}.
\end{eqnarray*}
\item $C_{1}$ is a simple, closed, positively oriented, contour enclosing 1, and no other singularities of the integrand.
\end{itemize}
\end{theorem}

\begin{proof}The theorem follows from similar arguments as in the proof of Theorem \ref{t33}.
\end{proof}

\subsection{Frozen boundary}

Again let $\mathbf{m}_{p_t,\alpha_t}$ be the limit of $\mathbf{m}_{\rho_{\epsilon}^t}$ as $\epsilon\rightarrow 0$. Define the liquid region $\mathcal{L}$ and the frozen boundary $\partial \mathcal{L}$ as in Definition \ref{df41}.

For $i\in I$, define 
\begin{eqnarray*}
F_{p_t,\alpha_t}^{(i)}(z)=\frac{z\left(Q_{i}'\left(z\right)+R'_{p_t,\alpha_t,i}(z)\right)+U_{p_t,\alpha_t,i}(z)}{\phi_{p_t,\alpha_t}}.
\end{eqnarray*}
Then when
$x$ is in a neighborhood of infinity, we have
\begin{align*}
  \mathrm{St}_{\mathbf{m}_{p_t,\alpha_t}}(x)
  &=\sum_{j=0}^{\infty}x^{-(j+1)} \int_{\RR}y^j\mathbf{m}_{p_t,\alpha_t}(dy)\\
&=\sum_{j=0}^{\infty}
\frac{1}{2(j+1)\pi\mathbf{i}}\sum_{x_j^{(p)}\in S_{p_t}}\oint_{C_{1}}\left(\frac{F_{p_t,\alpha_t}^{(\psi(p,j))}(z)}{x}\right)^{j+1}\frac{dz}{z}\\
&=-\frac{1}{2\pi\mathbf{i}}\sum_{x_j^{(p)}\in S_{p_t}}\oint_{C_{1}}\log\left(1-\frac{F_{p_t,\alpha_t}^{(\psi(p,j))}(z)}{x}\right)\frac{dz}{z}.
\end{align*}
Integration by parts gives
\begin{equation*}
\mathrm{St}_{\mathbf{m}_{p_t,\alpha_t}}(x)=
\frac{1}{2\pi\mathrm{i}}\sum_{x_j^{(p)}\in S_{p_t}}\left[\oint_{C_{1}}\log
z\frac{\frac{d}{dz}\left(1-\frac{F_{p_t,\alpha_t}^{(\psi(p,j))}(z)}{x}\right)}{1-\frac{F_{p_t,\alpha_t}^{(\psi(p,j))}(z)}{x}}dz
-\oint_{C_{1}} 
d\left(\log z\log \left(1-\frac{F_{p_t,\alpha_t}^{(\psi(p,j))}(z)}{x}\right)\right)\right].
\end{equation*}

Because for each $i\in[I]$,  $F_{p_t,\alpha_t}^{(i)}(z)$ has a pole at $1$, 
$F^{(i)}_{p_t,\alpha_t}(z)=x$ has exactly one root in a neighborhood of $1$ when $x$ is in a neighborhood of $\infty$, and thus, we can can find a unique composite inverse Laurent series
given by
\begin{equation*}
G_{p_t,\alpha_t}^{(i)}(w)=1+\sum_{i=1}^{\infty}\frac{\beta_i^{(j)}}{w^i},
\end{equation*}
such that $F_{p_t,\alpha_t}^{(i)}(G_{p_t,\alpha_t}^{(i)}(w))=w$ when $w$ is in a neighborhood of
infinity. Then
\begin{equation}
  \label{eq:rootFG}
z_{i}(x)=G_{p_t,\alpha_t}^{(i)}(x)
\end{equation}
is the unique root of $F_{p_t,\alpha_t}^{(i)}(z)=x$ in a neighborhood of $x_i$.

Since $1-\frac{F^{(i)}_{p_t,\alpha_t}}{x}$ has exactly one zero $z_{i}(x)$ and one pole $1$ in a neighborhood of $1$, we have
\begin{equation*}
\oint_{C_1}d\left(\log z\log \left(1-\frac{F_{p_t,\alpha_t,M}(z)}{x}\right)\right)=0;
\end{equation*}
where $C_1$ is a small positively oriented cycle enclosing $z_i(x)$ and $1$, but no other singularities of the integrand.
Therefore
\begin{equation}
\mathrm{St}_{\mathbf{m}_{p_t,\alpha_t}}(x)=\sum_{x_j^{(p)}\in S_{p_t}}\left[\log(z_{\psi(p,j)}(x))\right]\label{sjl1}
\end{equation}
when $x$ is in a neighborhood of infinity. By the complex analyticity of both
sides of \eqref{sjl1}, we infer that \eqref{sjl1} holds whenever $x$ is outside
the support of $\mathbf{m}_{p_t,\alpha_t}$.

\begin{lemma}\label{lnr}Let $A_i(z,\kappa)$ be defined as in (\ref{kti}). For each $k\in \{0,1,2,\ldots,d_{i+1}-d_i-1\}$, and each $c\in\RR$ $A_i(z,\kappa)=c$ has $|\mathcal{S}_{p_t}|$ roots in $\RR\cup\{\infty\}$, and all these roots are simple.
\end{lemma}

\begin{proof}Note that
\begin{small}
\begin{eqnarray*}
\gamma_i-\rho\theta_i=\lim_{\epsilon\rightarrow 0}
\frac{\left|j\in[t..r^{(\epsilon)}]:a_j^{(\epsilon)}=L,b_j^{(\epsilon)}=-,x_{j,\epsilon}=x_{k}^{(p)},\mathrm{s.t.}\ \psi(p,k)=i\right|}{N^{(\epsilon)}}-\lim_{\epsilon\rightarrow 0}\frac{N_{L,-,i}^{(\epsilon)}}{N^{(\epsilon)}}<0
\end{eqnarray*}
Then explicit arguments show that
\begin{eqnarray}
\frac{dA_i(z,\kappa)}{dz}<0\label{adl}
\end{eqnarray}
whenever $z\in\RR$ and the derivative is well defined. Then $A_i(z,\kappa)$ is strictly decreasing from $+\infty$ to $-\infty$ between any two consecutive singularities of $z$ in $\RR\cup\{\infty\}$. Then the lemma follows.
\end{small}
\end{proof}

\begin{proposition}\label{p72}For any $\kappa\in\RR$, and $i\in[I]$ the equation 
\begin{eqnarray}
F_{p_t,\alpha_t}^{(i)}(z)=\frac{\kappa-(1-\phi_{p_t,\alpha_t})}{\phi_{p_t,\alpha_t}}.\label{p1}
\end{eqnarray} has at most one pair of nonreal conjugate roots.
\end{proposition}

\begin{proof}We only prove the case when $x_{i,\epsilon}=x_{\si_0(1)}$;  the lemma can be proved similarly for other values of $i\in[I]$.

When $x_{i,\epsilon}=x_{\si_0(1)}$, from (\ref{p1}), we obtain
\begin{eqnarray}
\kappa-(1-\phi_{p_t,\alpha_t})&=&\frac{z}{z-1}\left(\gamma_i-\rho\theta_i\right)+\eta_i-\rho\left[\sum_{g=i+1}^{I}\theta_g\right]+\frac{\rho\theta_i}{S_{\bm_i}^{-1}(\ln z)}\label{ke}\\
&&+z\left(\sum_{p=1}^{p_t-1}\sum_{s=1}^{n_p}\frac{(V_p-V_{p-1})\zeta_{R,+,s,p}}{V_m-V_0}\frac{x_{\si_0(1)}x_{s}^{(p)}}{1+zx_{\si_0(1)}x_{s}^{(p)}}\right.\notag\\
&&+\sum_{s=1}^{n_{p_t}}\frac{\alpha_t(V_{p_t}-V_{{p_t}-1})\zeta_{R,+,s,p_t}}{V_m-V_0}\frac{x_{\si_0(1)}x_{s}^{(p_t)}}{1+zx_{\si_0(1)}x_{s}^{(p_t)}}\notag\notag\\
&&+\sum_{p=1}^{p_t-1}\sum_{s=1}^{n_p}\frac{(V_p-V_{p-1})\zeta_{L,+,s,p}}{V_m-V_0}\frac{x_{\si_0(1)}x_{s}^{(p)}}{1-zx_{\si_0(1)}x_{s}^{(p)}}\notag\\
&&\left.+\sum_{s=1}^{n_{p_t}}\frac{\alpha_t(V_{p_t}-V_{{p_t}-1})\zeta_{L,+,s,p_t}}{V_m-V_0}\frac{x_{\si_0(1)}x_{s}^{(p_t)}}{1-zx_{\si_0(1)}x_{s}^{(p_t)}}\right)\notag
\end{eqnarray}
By Lemma \ref{cm}, we obtain that the stieljes transform of $\bm_i$ is given by
\begin{eqnarray*}
\mathrm{St}_{\bm_i}\left(t_i\right)=\log\prod_{k=0}^{d_{i+1}-d_i-1}\frac{t_i-\beta_{i,k}}{t_i-\gamma_{i,k}}
\end{eqnarray*}
If
\begin{eqnarray*}
\frac{1}{S_{\bm_i}^{-1}(\ln z)}=t_i;
\end{eqnarray*}
then
\begin{eqnarray}
z=\prod_{k=0}^{d_{i+1}-d_i-1}\frac{t_i-\beta_{i,k}}{t_i-\gamma_{i,k}}\label{ze}
\end{eqnarray}
Using (\ref{ke}) to solve $t_i$, we obtain
\begin{eqnarray}
\rho\theta_it_i&=&\kappa-\frac{z}{z-1}\left(\gamma_i-\rho\theta_i\right)-\eta_i+\rho\left[\sum_{g=i+1}^{I}\theta_g\right]-(1-\phi_{p_t,\alpha_t})\label{kti}\\
&&-z\left(\sum_{p=1}^{p_t-1}\sum_{s=1}^{n_p}\frac{(V_p-V_{p-1})\zeta_{R,+,s,p}}{V_m-V_0}\frac{x_{\si_0(1)}x_{s}^{(p)}}{1+zx_{\si_0(1)}x_{s}^{(p)}}\right.\notag\\
&&+\sum_{s=1}^{n_{p_t}}\frac{\alpha_t(V_{p_t}-V_{{p_t}-1})\zeta_{R,+,s,p_t}}{V_m-V_0}\frac{x_{\si_0(1)}x_{s}^{(p_t)}}{1+zx_{\si_0(1)}x_{s}^{(p_t)}}\notag\notag\\
&&+\sum_{p=1}^{p_t-1}\sum_{s=1}^{n_p}\frac{(V_p-V_{p-1})\zeta_{L,+,s,p}}{V_m-V_0}\frac{x_{\si_0(1)}x_{s}^{(p)}}{1-zx_{\si_0(1)}x_{s}^{(p)}}\notag\\
&&\left.+\sum_{s=1}^{n_{p_t}}\frac{\alpha_t(V_{p_t}-V_{{p_t}-1})\zeta_{L,+,s,p_t}}{V_m-V_0}\frac{x_{\si_0(1)}x_{s}^{(p_t)}}{1-zx_{\si_0(1)}x_{s}^{(p_t)}}\right)\notag:=A_i(z,\kappa)
\end{eqnarray}
Plugging (\ref{kti}) to (\ref{ze}), we obtain
\begin{eqnarray}
z=\prod_{k=0}^{d_{i+1}-d_i-1}\frac{A_i(z,\kappa)-\rho\theta_i\beta_{i,k}}{A_i(z,\kappa)-\rho\theta_i\gamma_{i,k}}=B_i(z,\kappa)\label{dzb}
\end{eqnarray}
By Lemma \ref{lnr}, for each $k\in \{0,1,2,\ldots,d_{i+1}-d_i-1\}$, $A_i(z,\kappa)=\beta_{i,k}$ (resp.\ $A_i(z,\kappa)=\gamma_{i,k}$) has $|\mathcal{S}_{p_t}|:=m+1$ roots in $\RR\cup\{\infty\}$, and all these roots are simple. Hence we can write
\begin{eqnarray*}
B_i(z,\kappa)=C_i\frac{\prod_{k=0}^{d_{i+1}-d_i-1}\left(\prod_{j=1}^{m+1}(z-u_{k,j})\right)}{\prod_{k=0}^{d_{i+1}-d_i-1}\left( \prod_{j=1}^{m+1}(z-v_{k,j})\right)}
\end{eqnarray*}
where $\{u_{k,j}\}_{j=1}^{m+1}$ (resp.\ $\{v_{k,j}\}_{j=1}^{m+1}$) are roots of $A_i(z,\kappa)=\rho\theta_i\beta_{i,k}$ (resp.\ $A_i(z,\kappa)=\rho\theta_i\gamma_{i,k}$) satisfying
\begin{eqnarray*}
-\infty<u_{k,2}<\ldots<u_{k,m+1}<\infty\\
-\infty<v_{k,2}<\ldots<v_{k,m+1}<\infty
\end{eqnarray*}
and
\begin{eqnarray*}
u_{k,1}\in [-\infty,u_{k,2})\cup (u_{k,m+1},\infty]\\
v_{k,1}\in [-\infty,u_{k,2})\cup (v_{k,m+1},\infty]
\end{eqnarray*}
and $C_i\in \RR$ are constants independent of $z$ given by
\begin{eqnarray*}
C_i=\lim_{z\rightarrow\infty}B_i(z,\kappa).
\end{eqnarray*}
Recall that 
\begin{eqnarray}
\beta_{i,d_{i+1}-d_i-1}<\gamma_{i,d_{i+1}-d_i-1}<\beta_{i,d_{i+1}-d_i-2}<\gamma_{i,d_{i+1}-d_i-2}<\ldots<\beta_{i,0}<\gamma_{i,0}\label{brit}
\end{eqnarray}

Recall that $\mathcal{S}_{p_t}$ consists of all the singular points of $A_i(z,\kappa)$, which 
divide $\RR$ into $|\mathcal{S}_{p_t}|-1$ bounded intervals and two unbounded intervals.

By Lemma \ref{lnr}, $A_i(z,\kappa)$ decreases from $\infty$ to $-\infty$ in each interval of $\RR\cup\{\infty\}$ divided by points in $\mathcal{S}_{p_t}$. Since $\gamma_{i,k}>\beta_{i,k}$, the solution in $z$ of $A_i(z,\kappa)-\rho\theta_i\gamma_{i,k}=0$ is to the left of the solution in $z$ of $A_i(z,\kappa)-\rho\theta_i\beta_{i,k}=0$ in each bounded interval of $\RR$ divided by points in $\mathcal{S}_{p_t}$. More precisely, in the $j$th one (counting from the left) of the $m$ bounded intervals above, there exist $2(d_{i+1}-d_i)$ roots in $\{u_{k,j},v_{k,j}\}_{k\in[d_{i+1}-d_i-1]\cup\{0\},j\in[m+1]}$ given by
\begin{eqnarray*}
v_{0,j+1}<u_{0,j+1}<v_{1,j+1}<u_{1-1,j+1}<\ldots<v_{d_{i+1}-d_i-1,j+1}<u_{d_{i+1}-d_i-1,j+1}.
\end{eqnarray*}
Then $\{v_{k,j}\}_{k\in[d_{i+1}-d_i-1]\cup\{0\},j\in[m+1]}$ divided $\RR$ into $(m+1)(d_{i+1}-d_i)-1$ bounded intervals and two unbounded intervals. By Lemma \ref{l33}, the following cases might occur
\begin{enumerate}
    \item For fixed $\kappa\in \RR$, when $z$ is in each bounded interval of $\RR$ divided by 
    $\{v_{k,j}\}_{k\in[d_{i+1}-d_i-1]\cup\{0\},j\in[m+1]}$, $B_i(z,\kappa)$ increases from $-\infty$ to $\infty$; in $\max\{v_{k,j}\}_{k\in[d_{i+1}-d_i-1]\cup\{0\},j\in[m+1]},\infty)$, $B_i(z,\kappa)$ increases from $-\infty$ to $C_i$; in
    $(-\infty,\min\{v_{k,j}\}_{k\in[d_{i+1}-d_i-1]\cup\{0\},j\in[m+1]})$, $B_i(z,\kappa)$ increases from $C_i$ to $\infty$. In this case, equation (\ref{dzb}) has at least one real solution in each bounded interval of $\RR$ divided by 
    $\{v_{k,j}\}_{k\in[d_{i+1}-d_i-1]\cup\{0\},j\in[m+1]}$; this gives us at least $(m+1)(d_{i+1}-d_i)-1$ real solutions.
    \item For fixed $\kappa\in \RR$, when $z$ is in each bounded interval of $\RR$ divided by 
    $\{v_{k,j}\}_{k\in[d_{i+1}-d_i-1]\cup\{0\},j\in[m+1]}$, $B_i(z,\kappa)$ decreases from $-\infty$ to $\infty$; in $\max\{v_{k,j}\}_{k\in[d_{i+1}-d_i-1]\cup\{0\},j\in[m+1]},\infty)$, $B_i(z,\kappa)$ decreases from $+\infty$ to $C_i$; in
    $(-\infty,\min\{v_{k,j}\}_{k\in[d_{i+1}-d_i-1]\cup\{0\},j\in[m+1]})$, $B_i(z,\kappa)$ decreases from $C_i$ to $-\infty$. In this case, equation (\ref{dzb}) has at least one real solution in each bounded or unbounded interval of $\RR$ divided by 
    $\{v_{k,j}\}_{k\in[d_{i+1}-d_i-1]\cup\{0\},j\in[m+1]}$; this gives us at least $(m+1)(d_{i+1}-d_i)+1$ real solutions.
\end{enumerate}
The equation (\ref{dzb}) has at most $(m+1)(d_{i+1}-d_i)+1$ roots in $\CC$ counting multiplicities. Then the lemma follows.
\end{proof}

The proof of Lemma \ref{p72} also gives us the following corollary.
\begin{corollary}\label{cl415}Assume that equation (\ref{p1}) has one pair of nonreal conjugate roots in $z$, then when $z$ is in each bounded interval of $\RR$ divided by 
    $\{v_{k,j}\}_{k\in[d_{i+1}-d_i-1]\cup\{0\},j\in[m+1]}$, $B_i(z,\kappa)$ increases from $-\infty$ to $\infty$; in $\max\{v_{k,j}\}_{k\in[d_{i+1}-d_i-1]\cup\{0\},j\in[m+1]},\infty)$, $B_i(z,\kappa)$ increases from $-\infty$ to $C_i$; in
    $(-\infty,\min\{v_{k,j}\}_{k\in[d_{i+1}-d_i-1]\cup\{0\},j\in[m+1]})$, $B_i(z,\kappa)$ increases from $C_i$ to $\infty$.
\end{corollary}

\begin{lemma}\label{l70}
Let $i\in[I]$. Assume $\kappa\in \RR$ is such that the equation 
(\ref{p1})
has a pair of nonreal conjugate roots. Let $s_i(x)$ be a real root of (\ref{p1}). Then
\begin{eqnarray*}
\left.\frac{\partial s_1(x)}{\partial x}\right|_{x=x_0}\geq 0.
\end{eqnarray*}
It is equal to 0 if and only if $s_i(x_0)=1$.
\end{lemma}
\begin{proof}We only prove the case when $x_{i,\epsilon}=x_{\si_0(1)}$;  the lemma can be proved similarly for other values of $i\in[I]$.

The derivative $s_i'(\kappa)$ can be computed explicitly from (\ref{dzb}) as follows
\begin{eqnarray*}
s_i'(x)=\frac{\frac{\partial B_i(z,\kappa)}{\partial \kappa}}{1-\frac{\partial B_i(z,\kappa)}{\partial z}}
\end{eqnarray*}
First we claim that $\frac{\partial B_i(z,\kappa)}{\partial \kappa}\leq 0$.
Let 
\begin{eqnarray*}
D_i(z)=\kappa-A_i(z,\kappa)
\end{eqnarray*}
Note that 
\begin{eqnarray*}
B_i(z,\kappa)
&=&\frac{\prod_{k=0}^{d_{i+1}-d_i-1}\left[\kappa-D_i(z)-\beta_{i,k}\right]}{\prod_{k=0}^{d_{i+1}-d_i-1}\left[\kappa-D_i(z)-\gamma_{i,k}\right]}
\end{eqnarray*}
where
\begin{eqnarray*}
&&\beta_{1,k+1}+D_i(z)<\gamma_{i,k+1}+D_i(z)
<\beta_{1,k}+D_i(z).
\end{eqnarray*}
By Lemma \ref{l33}, for each fixed $z$, $B_i(z,\kappa)$ is strictly decreasing in $\kappa$ whenever it is defined. Hence $\frac{\partial B_i(z,\kappa)}{\partial \kappa}< 0$ whenever
\begin{small}
\begin{eqnarray*}
z\notin \left\{1,\left\{\frac{1}{x_{\si_0(1)}x_s^{(p)}}\right\}_{p\in[p_t],s\in[n_p],(a_{s}^{(p)},b_s^{(p)})=(L,+)},\left\{-\frac{1}{x_{\si_0(1)}x_s^{(p)}}\right\}_{p\in[p_t],s\in[n_p],(a_{s}^{(p)},b_s^{(p)})=(R,+)}\right\}:=\Xi_{p_t}
\end{eqnarray*}
\end{small}
and $\frac{\partial B_i(z,\kappa)}{\partial \kappa}= 0$; if $z\in \Xi_{p_t}$.

Now we show that $\frac{\partial B_i(z,\kappa)}{\partial z}> 1$. By Corollary \ref{cl415}, when equation (\ref{p1}) has one pair of nonreal conjugate roots in $z$, then when $z$ is in each bounded interval of $\RR$ divided by 
    $\{v_{k,j}\}_{k\in[d_{i+1}-d_i-1]\cup\{0\},j\in[m+1]}$, $B_i(z,\kappa)$ increases from $-\infty$ to $\infty$. Each real solution is simple and is in a bounded interval of $\RR$ divided by 
    $\{v_{k,j}\}_{k\in[d_{i+1}-d_i-1]\cup\{0\},j\in[m+1]}$. Hence at each real solution of (\ref{dzb}), we have $\frac{\partial B_i(z,\kappa)}{\partial z}> 1$.
    Then the lemma follows.
\end{proof}

Let $i\in[I]$. Consider the following cases
\begin{enumerate}
\item If $i=1$, we consider dimer coverings on a rail-yard graph with left boundary condition given by the partition $\phi^{(i,\si_0)}(N_{L,-}^{(\epsilon)})$, and modified edge weights satisfying
\begin{enumerate}
    \item If the edge is incident to a blue vertex not labeled by $(L,-)$, then the weight is the same as the original weight;
    \item If the edge is incident to a blue vertex labeled by $(L,-)$ and has original weight $x_{\sigma_0(1)}$, then the weight is the same as the original weight;
    \item For all the other cases, give the edge weight 0.
\end{enumerate}
\item If $i>1$, we consider dimer coverings on a rail-yard graph with left boundary condition given by the partition $\phi^{(i,\si_0)}(N_{L,-}^{(\epsilon)})$, and modified edge weights satisfying
\begin{enumerate}
    \item If the edge is incident to a blue vertex not labeled by $(L,-)$, then the weight is the same as the original weight;
    \item If the edge is incident to a blue vertex labeled by $(L,-)$, give the edge weight 0.
\end{enumerate}
\end{enumerate}
Let $t\in[l^{(\epsilon)}..r^{(\epsilon)}]$ satisfying (\ref{lmt}) and $\bm_{\rho_t,\alpha_t,i}$ be the limit counting measure for the partitions on the $(2t-1)$th column. Using the same arguments as before (see also \cite{ZL202}), we obtain that
\begin{eqnarray*}
\mathrm{St}_{\bm_{\rho_t,\alpha_t,i}}\left(\widetilde{\kappa}_i\right)=\log (z_i(\kappa))
\end{eqnarray*}
where $z_i(\kappa)$ is given by (\ref{eq:rootFG}), and
\begin{eqnarray*}
\widetilde{\kappa}_i=\frac{\kappa-(1-\phi_{p_t,\alpha_t})}{\phi_{p_t,\alpha_t}\gamma_i}+\frac{\rho\sum_{g=i+1}^{I}\theta_g-\eta_i}{\phi_{p_t,\alpha_t}\gamma_i}
\end{eqnarray*}
Hence we have
\begin{eqnarray*}
z_i(\kappa)=\mathrm{exp}\left(\int_{\RR}\frac{\bm_{\rho_t,\alpha_t,i}[ds]}{\widetilde{\kappa}_i-s}\right);
\end{eqnarray*}
and 
\begin{eqnarray*}
z_i(\kappa+\mathbf{i}\epsilon)=\mathrm{exp}\left(\int_{\RR}\frac{\left(\widetilde{\kappa}_i-s-\mathbf{i}\epsilon\right)\bm_{\rho_t,\alpha_t,i}[ds]}{\left(\widetilde{\kappa}_i-s\right)^2+\epsilon^2}\right)
\end{eqnarray*}
Therefore $\Im[z_i^{\kappa}(x+\mathbf{i}\epsilon)]<0$ when $\epsilon$ is a small positive number. However, when nonreal roots exist for (\ref{p1}), for real root $s_i(x)$, Lemma \ref{l70} implies that $\Im[s_i(x+\mathbf{i}\epsilon)]>0$ when $\epsilon$ is a small positive number. This implies that when nonreal roots exist for (\ref{p1}), $z_i^{\kappa}(x+\mathbf{i}\epsilon)$ cannot be real. Then we have the following theorem

 
\begin{lemma}\label{l417}Let $i\in I$. The set
\begin{eqnarray}
\mathcal{C}_i:=\left\{(\chi,\kappa)\in [V_0,V_m]\times\RR: F_{p_t,\alpha_t}^{(i)}(z)=\frac{\kappa-(1-\phi_{p_t,\alpha_t})}{\phi_{p_t,\alpha_t}}\ \mathrm{has\ double\ real\ roots}\label{dci}
\right\}
\end{eqnarray}
is a bounded in $\RR^2$.
\end{lemma} 
\begin{proof}The condition that (\ref{p1}) has double real roots gives us the following parametric equation given by (\ref{ke}) and
\begin{eqnarray}
0&=&-\frac{1}{(z-1)^2}\left(\gamma_i-\rho\theta_i\right)+\frac{\rho\theta_i}{S_{\bm_i}^{-1}(\ln z)}+\frac{dt_i}{dz}\label{pe2}\\
&&+\left(\sum_{p=1}^{p_t-1}\sum_{s=1}^{n_p}\frac{(V_p-V_{p-1})\zeta_{R,+,s,p}}{V_m-V_0}\frac{x_{\si_0(1)}x_{s}^{(p)}}{\left(1+zx_{\si_0(1)}x_{s}^{(p)}\right)^2}\right.\notag\\
&&+\sum_{s=1}^{n_{p_t}}\frac{\alpha_t(V_{p_t}-V_{{p_t}-1})\zeta_{R,+,s,p_t}}{V_m-V_0}\frac{x_{\si_0(1)}x_{s}^{(p_t)}}{\left(1+zx_{\si_0(1)}x_{s}^{(p_t)}\right)^2}\notag\notag\\
&&+\sum_{p=1}^{p_t-1}\sum_{s=1}^{n_p}\frac{(V_p-V_{p-1})\zeta_{L,+,s,p}}{V_m-V_0}\frac{x_{\si_0(1)}x_{s}^{(p)}}{\left(1-zx_{\si_0(1)}x_{s}^{(p)}\right)^2}\notag\\
&&\left.+\sum_{s=1}^{n_{p_t}}\frac{\alpha_t(V_{p_t}-V_{{p_t}-1})\zeta_{L,+,s,p_t}}{V_m-V_0}\frac{x_{\si_0(1)}x_{s}^{(p_t)}}{\left(1-zx_{\si_0(1)}x_{s}^{(p_t)}\right)^2}\right),\notag
\end{eqnarray}
where $t_i$ and $z$ satisfy (\ref{ze}).
It suffices to show that when $\chi\in [V_0,V_m]$, $\kappa\neq \notin\infty$. 
Let
\begin{eqnarray*}
\Gamma_{p_t}:=\left\{\left\{\frac{1}{x_{\si_0(1)}x_s^{(p_t)}}\right\}_{s\in[n_{p_t}],(a_{s}^{(p_t)},b_s^{(p)})=(L,+)},\left\{-\frac{1}{x_{\si_0(1)}x_s^{(p)}}\right\}_{s\in[n_{p_t}],(a_{s}^{(p)},b_s^{(p)})=(R,+)}\right\}
\end{eqnarray*}
From (\ref{ke}) we see that $\kappa\in \{\pm\infty\}$ only when one of the following three conditions holds:
\begin{enumerate}
\item $z\in \left[\Xi_{p_t}\setminus \Gamma_{p_t}\right]$; or
\item $z\in \left[\Xi_{p_t}\cap \Gamma_{p_t}\right]$; or
\item $z\in \left[\Gamma_{p_t}\setminus \Xi_{p_t}\right] $.
\end{enumerate}
In Case (1), $\kappa\in\{\pm \infty\}$ and $\alpha_t=\{\pm \infty\}$, which contradicts the assumption that $\alpha_t\in [0,1]$. 

In Case (2), for each $\zeta\in \left[\Xi_{p_t}\cap \Gamma_{p_t}\right]$, let $p_{\zeta}$ be the coefficient for $\frac{1}{z-\zeta}$ in $F^{(i)}_{p_t,\alpha_t}$; let $q_{\zeta}$ be the coefficient for $\frac{\alpha_t}{z-\zeta}$ in  $F^{(i)}_{p_t,\alpha_t}$, then by (\ref{pe2}),
\begin{eqnarray*}
\lim_{z\rightarrow \zeta\ \mathrm{along}\ \mathcal{C}_i}\alpha_t=-\frac{p_{\zeta}}{q_{\zeta}}
\end{eqnarray*}
Then in $F^{(i)}_{p_t,\alpha_t}$,
\begin{eqnarray*}
\frac{p_{\zeta}}{z-\zeta}+\frac{q_{\zeta}\alpha_t}{z-\zeta}=0,\ \forall z\neq \zeta.
\end{eqnarray*}
Hence $|\kappa|<\infty$, as $z\rightarrow \zeta$ along $\mathcal{C}_i$.

In Case (3), for each $\xi\in\left[\Gamma_{p_t}\setminus \Xi_{p_t}\right]$, we have
\begin{eqnarray*}
\frac{|\alpha_t(z)|}{|z-\xi|^2}<\infty,\ \mathrm{as}\ z\rightarrow \xi\ \mathrm{along}\ \mathcal{C}_i;
\end{eqnarray*}
which implies
\begin{eqnarray*}
\frac{|\alpha_t(z)|}{|z-\xi|}\rightarrow 0,\ \mathrm{as}\ z\rightarrow \xi\ \mathrm{along}\ \mathcal{C}_i.
\end{eqnarray*}
which implies that $\kappa\notin \{\pm\infty\}$. Then the lemma follows.
\end{proof} 

\begin{lemma}\label{l25}
  Assume that 
  \begin{enumerate}
  \item for any $\kappa\in \RR$,
  (\ref{p1}) has at most one pair of complex conjugate roots. 
  \item  Let $\chi$ be given by (\ref{dfc}). For any $i,j\in [I]$ and $i\neq j$, $\mathcal{C}_i\cap\mathcal{C}_j=\emptyset$.
    \end{enumerate}
  Then a point $(\chi,\kappa)$ lies on the frozen boundary if and only if there exists $i\in I$ such that the equation
  (\ref{p1}) has double roots.
\end{lemma}
\begin{proof}By~\eqref{sjl1}, we can compute
  the continuous density $f_{\mathbf{m}_{p_t,\alpha_t}}(x)$ of the
  measure $\mathbf{m}_{\kappa}(x)$ with respect to the Lebesgue measure as follows
  \begin{eqnarray*}
    f_{\mathbf{m}_{p_t,\alpha_t}}(x)&=&-\lim_{\epsilon\rightarrow 0+}\frac{1}{\pi}\Im
    [\mathrm{St}_{\mathbf{m}_{p_t,\alpha_t}}(x+\mathbf{i}\epsilon)]\\
 &=&-\lim_{\epsilon\rightarrow 0+}\frac{1}{\pi}\mathrm{Arg}\left(\prod_{x_j^{(p)}\in S_{p_t}} z_{\psi(p,j)}(x+\mathbf{i}\epsilon)\right).
  \end{eqnarray*}
  Then it is straightforward to check the conclusion given the assumptions of the lemma.
\end{proof}

\begin{theorem}\label{t419}Consider dimer coverings on a rail yard graph $RYG(l^{(\epsilon)},r^{(\epsilon)},\underline{a}^{(\epsilon)},\underline{b}^{(\epsilon)})$, conditional on the left and right boundary condition $\lambda^{(l,\epsilon)}$ and $\emptyset$, respectively. Suppose that (\ref{c151}) and Assumptions \ref{ap14}, \ref{ap41}, \ref{ap32}, \ref{ap428} and \ref{ap36} hold. Then the frozen boundary is given by $\cup_{i\in [I]}\mathcal{C}_i$, where for $i,j\in [I]$ and $i\neq j$, $\mathcal{C}_i\cap \mathcal{C}_j=\emptyset$.
\end{theorem}

\begin{proof}It suffices to check conditions (1) and (2) of Lemma \ref{l25}. Condition (1) follows from Proposition \ref{p72}. Condition (2) follows from Lemma \ref{l417} and Assumption \ref{ap428} (4b) when $C_1$ is sufficiently large. Then the theorem follows.
\end{proof}

\subsection{$m=1$} When $m=1$, let
\begin{eqnarray*}
V_0=0;\qquad V_1=1;\qquad n_1=n;\qquad x_s^{(1)}=x_s,\ \forall s\in[n].
\end{eqnarray*}
By (\ref{dfc}) we obtain $\chi=\alpha_t$. 
Write $F_{p_t,\alpha_t}^{(i)}=F_{\chi}^{(i)}$.
For $i\in[I]$, note that 
\begin{eqnarray*}
&&\gamma_i-\rho\theta_i=-\chi\rho\theta_i;\\
&&\eta_i-\rho\left[\sum_{g=i+1}^{I}\theta_g\right]=-\chi
\rho\left[\sum_{g=i+1}^{I}\theta_g\right].
\end{eqnarray*}
equation (\ref{p1}) gives
\begin{align*}
\kappa&=\frac{\rho\theta_i}{S_{\bm_i}^{-1}(\ln z)}-\rho\left[\sum_{g=i+1}^{I}\theta_g\right]\chi+1-\phi_{1,\chi}\\
&+z\chi\left(\sum_{s=1}^{n}\zeta_{R,+,s,p_t}\frac{x_{\si_0(1)}x_{s}}{1+zx_{\si_0(1)}x_{s}}+\sum_{s=1}^{n}\zeta_{L,+,s,p_t}\frac{x_{\si_0(1)}x_{s}}{1-zx_{\si_0(1)}x_{s}}-\frac{\rho\theta_i}{z-1}\right)\notag
\end{align*}
Note that equation (\ref{p1}) has double root in $z$ if and only if
\begin{eqnarray}
0=\frac{d F_{\chi}^{(i)}}{dz}.\label{de0}
\end{eqnarray}
Define
\begin{eqnarray*}
\Phi_i(t):&=&\prod_{k=0}^{d_{i+1}-d_i-1}\frac{t-\beta_{i,k}}{t-\gamma_{i,k}};
\end{eqnarray*}
and
\begin{small}
\begin{eqnarray*}
J_1(t):=\Phi_1(t)\left(\sum_{s=1}^{n}\zeta_{R,+,s,p_t}\frac{x_{\si_0(1)}x_{s}}{1+\Phi_1(t)x_{\si_0(1)}x_{s}}+\sum_{s=1}^{n}\zeta_{L,+,s,p_t}\frac{x_{\si_0(1)}x_{s}}{1-\Phi_1(t)x_{\si_0(1)}x_{s}}-\frac{\rho\theta_i}{\Phi_1(t)-1}\right)-\rho\left[\sum_{g=2}^{I}\theta_g\right];
\end{eqnarray*}
\end{small}
and for $i\in[I]$ and $i>1$,
\begin{small}
\begin{eqnarray*}
J_i(t):=\Phi_i(t)\left(-\frac{\rho\theta_i}{\Phi_i(t)-1}\right)-\rho\left[\sum_{g=i+1}^{I}\theta_g\right];
\end{eqnarray*}
\end{small}
Then the condition that (\ref{p1}) has double roots gives the parametric equation for $(\chi,\kappa)$ as follows:
\begin{eqnarray*}
\chi&=&-\frac{\rho\theta_i}{J_i'(t)};\\
\kappa&=&\rho\theta_i t-\frac{\rho\theta_iJ_i(t)}{J_i'(t)}+1-\phi_{1,\chi}.
\end{eqnarray*}

\begin{proposition}\label{p420}Let $i\in [I]$, and $\mathcal{C}_i$ be given by (\ref{dci}).  Then $\mathcal{C}_i$ is a cloud curve of rank 
\begin{enumerate}
\item $|J_1|\cdot|\Xi_p|$, if $i=1$;
\item $|J_i|$, if $i>1$.
\end{enumerate}
\end{proposition}
\begin{proof}The proposition follows from similar arguments as the proof of Proposition 3.16 in \cite{ZL202}.
\end{proof}

\begin{theorem}\label{t421}Consider dimer coverings on a rail yard graph $RYG(l^{(\epsilon)},r^{(\epsilon)},\underline{a}^{(\epsilon)},\underline{b}^{(\epsilon)})$, conditional on the left and right boundary condition $\lambda^{(l,\epsilon)}$ and $\emptyset$, respectively. Suppose that (\ref{c151}) and Assumptions \ref{ap14}, \ref{ap41}, \ref{ap32}, \ref{ap428} and \ref{ap36} hold. Assume $m=1$. Then the frozen boundary is given by $\cup_{i\in [I]}\mathcal{C}_i$, where for $i,j\in [I]$ and $i\neq j$, $\mathcal{C}_i\cap \mathcal{C}_j=\emptyset$, and $\mathcal{C}_i$ is a cloud curve for all $i\in [I]$.
\end{theorem}

\begin{proof}The theorem follows from Theorem \ref{t419} and Proposition \ref{p420}.
\end{proof}

\begin{example}\label{ex4}Consider a rail-yard graph with $m=1$ and $n_1=4$. Let
\begin{eqnarray*}
V_0=0;\qquad V_1=1.
\end{eqnarray*}
Let
\begin{eqnarray*}
(a_1,b_1)=(L,-),\qquad (a_2,b_2)=(R,+),\qquad (a_3,b_3)=(L,+),\qquad (a_4,b_4)=(L,-).
\end{eqnarray*}
and
\begin{eqnarray*}
x_{1,\epsilon}^{(1)}=1,\qquad x_{2,\epsilon}^{(1)}=1/2,\qquad x_{3,\epsilon}^{(1)}=1/3;\qquad 0<x_{4,\epsilon}^{(1)}<e^{-\alpha N^{(\epsilon)}}
\end{eqnarray*}

Assume $N_{L,-}^{(\epsilon)}$ is an integer multiple of $12$. Assume the boundary partition $\lambda^{(l,\epsilon)}$ satisfies
\begin{eqnarray*}
&&\lambda_1=\lambda_2=\ldots=\lambda_{\frac{N_{L,-}^{(\epsilon)}}{4}}=\mu_1^{(\epsilon)}\approx 6N_{L,-}^{(\epsilon)}\\
&&\lambda_{\frac{N_{L,-}^{(\epsilon)}}{4}+1}=\lambda_{\frac{N_{L,-}^{(\epsilon)}}{4}+2}=\ldots=\lambda_{\frac{N_{L,-}^{(\epsilon)}}{2}}=\mu_2^{(\epsilon)}\approx 5N_{L,-}^{(\epsilon)}\\
&&\lambda_{\frac{N_{L,-}^{(\epsilon)}}{2}+1}(N)=\lambda_{\frac{N_{L,-}^{(\epsilon)}}{2}+2}=\ldots=\lambda_{\frac{2N_{L,-}^{(\epsilon)}}{3}}=\mu_3^{(\epsilon)}\approx 2N_{L,-}^{(\epsilon)}\\
&&\lambda_{\frac{2N_{L,-}^{(\epsilon)}}{3}+1}=\lambda_{\frac{N_{L,-}^{(\epsilon)}}{2}+2}=\ldots=\lambda_{\frac{5N_{L,-}^{(\epsilon)}}{6}}=\mu_4^{(\epsilon)}\approx N_{L,-}^{(\epsilon)}\\
&&\lambda_{\frac{5N_{L,-}^{(\epsilon)}}{6}+1}=\lambda_{\frac{5N_{L,-}^{(\epsilon)}}{6}+2}=\ldots=\lambda_{N_{L,-}^{(\epsilon)}}=\mu_5^{(\epsilon)}=0.
\end{eqnarray*}

Then
\begin{align*}
\Phi_1(t):&=\frac{(t-11)(t-13.5)}{(t-11.5)(t-14)};\\
\Phi_2(t):&=\frac{t\left(t-\frac{7}{3}\right)\left(t-\frac{14}{3}\right)}{\left(t-\frac{1}{3}\right)\left(t-\frac{8}{3}\right)(t-5)};
\end{align*}
\begin{align*}
\phi_{1,\chi}=\frac{1}{2}(1-\chi)
\end{align*}
and
\begin{align*}
J_1(t):&=\frac{\Phi_1(t)}{4}\left(\frac{2}{2+\Phi_1(t)}+\frac{3}{3-\Phi_1(t)}-\frac{1}{\Phi_1(t)-1}\right)-\frac{1}{4};\\
J_2(t):&=-\frac{\Phi_2(t)}{4(\Phi_2(t)-1)}-\frac{1}{4}.
\end{align*}

Then the frozen boundary is a union $\mathcal{C}_1\cup \mathcal{C}_2$, where $\mathcal{C}_i$ has parametric equations given by 
\begin{eqnarray}
\chi_i(t_i)&=&-\frac{1}{4J_i'(t)}.\label{cix}\\
\kappa_i(t_i)&=&\frac {t_i}{4}-\frac{J_i(t)}{4J_i'(t)}+\frac{1}{2}+\frac{\chi_i(t_i)}{2}.\label{ciy}
\end{eqnarray}
See Figure \ref{fig:fbr} for a picture of the frozen boundary.

\begin{figure}
 \includegraphics[width=0.6\textwidth]{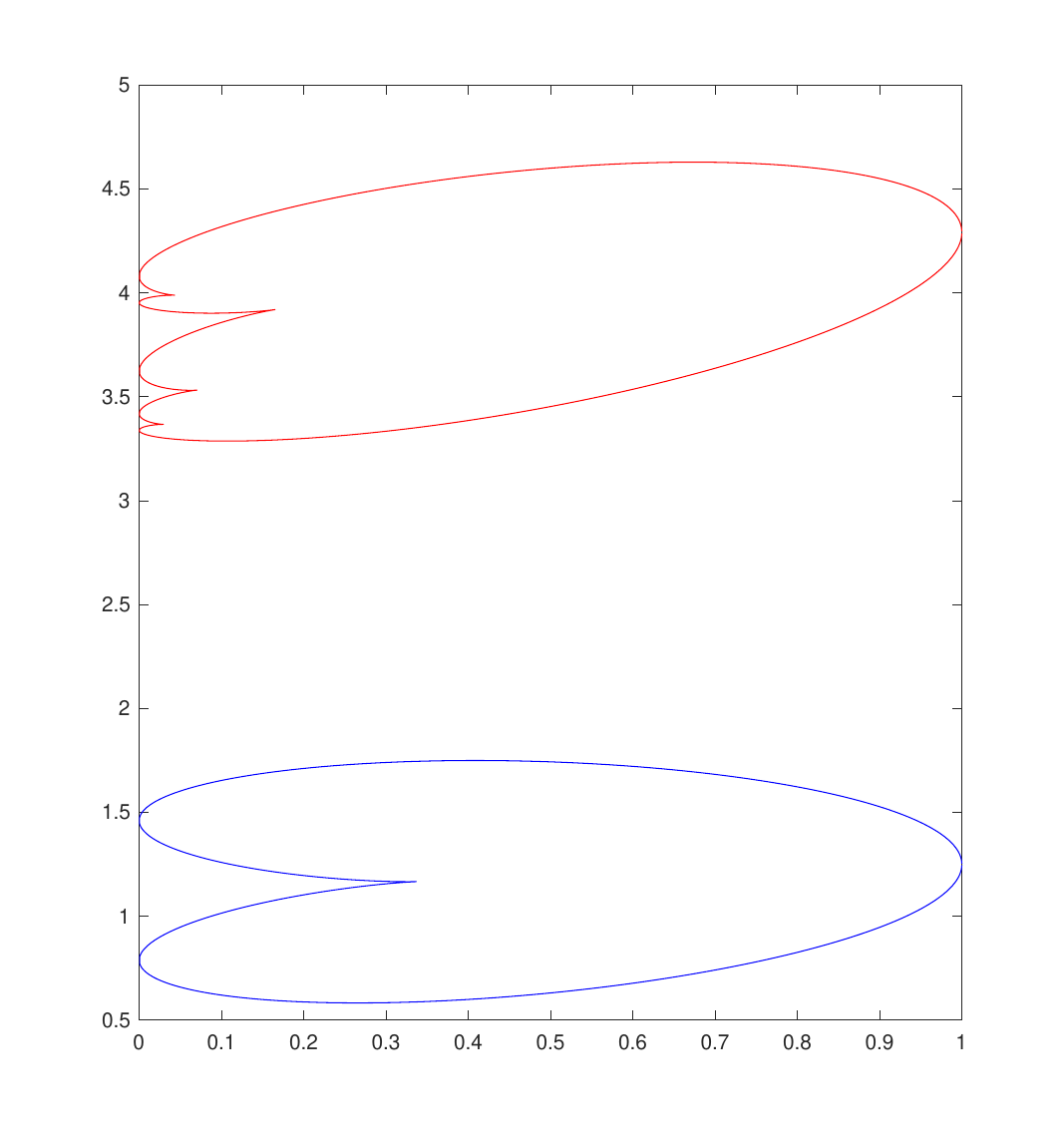}
\caption{Frozen boundary of Example \ref{ex4}}
\label{fig:fbr}
\end{figure}
\end{example}
\appendix

\section{Technical Results}\label{sect:AA}

\begin{lemma}\label{la1}Let $n$ be a positive integer and let $g(z)$ be an analytic function defined in a neighborhood of 1. Then
\begin{eqnarray*}
\lim_{(z_1,\ldots,z_n)\rightarrow (1,\ldots,1)}\left(\sum_{i=1}^n\frac{g(z_i)}{\prod_{j\in[n],j\neq i}(z_i-z_j)}\right)=\left.\frac{\partial^{n-1}}{\partial z^{n-1}}\left(\frac{g(z)}{(n-1)!}\right)\right|_{z=1}
\end{eqnarray*}
\end{lemma}

\begin{proof}See Lemma 5.5 of \cite{bg}.
\end{proof}

\begin{lemma}
  \label{p25}
  Assume that $(\lambda(N)\in \YY_N$, $N=1,2,\ldots)$ is a sequence of partitions, such that there exists a piecewise continuous function $f(t)$ and a constant $C>0$ such that
  \begin{equation*}
    \lim_{N\rightarrow\infty} \frac{1}{N}
    \sum_{j=1}^{N}
    \left|\frac{\lambda_j(N)}{N}-f\left(\frac{j}{N}\right)\right| = 0,
    \quad
    \text{and}
    \quad
    \sup_{1\leq j\leq N}
    \left|\frac{\lambda_j(N)}{N}-f\left(\frac{j}{N}\right)\right|<C\ \text{for
    all $N\geq1$}.
  \end{equation*}
Moreover, assume that
\begin{equation*}
\lim_{N\rightarrow\infty}m(\lambda(N))=\bm.
\end{equation*}

Then for each fixed $k\in\NN$, there is a small open complex neighborhood of
$(1,\ldots,1)\in\mathbb{C}^k$, such that the following
convergence occurs uniformly in this neighborhood:
\begin{equation*}
\lim_{N\rightarrow\infty}
\frac{1}{N}
\log\left(%
\frac{%
  s_{\lambda(N)}\left(u_1,\ldots,u_k,1,\ldots,1\right)
}{%
  s_{\lambda(N)}\left(1,\ldots,1\right)
}
\right)
=H_{\bm}(u_1)+\cdots+H_{\bm}(u_k).
\end{equation*}
\end{lemma}

\begin{proof}See Theorem 4.2 of \cite{bg}.
\end{proof}

\begin{lemma}
  \label{l33}
  \begin{enumerate}
  \item
  Let
  \begin{equation*}
    R(z)=\frac{(z-u_1)(z-u_2)\cdots(z-u_h)}{(z-v_1)(z-v_2)\cdots(z-v_h)},
  \end{equation*}
  where $\{u_i\}$ and $\{v_i\}$ are two sets of real numbers, and $h$ is a positive integer.
  \begin{itemize}
    \item If $\{u_i\}$ and $\{v_i\}$ satisfy
      \begin{equation*}
	v_1<u_1<v_2<u_2<\cdots<v_h<u_h.
      \end{equation*}
      Then $R(z)$ is monotone increasing in each one of the following intervals
      \begin{equation*}
	(-\infty,v_1), (v_1,v_2),\ldots, (v_{h-1},v_h),(v_h,\infty).
      \end{equation*}
    \item If $\{u_i\}$ and $\{v_i\}$ satisfy
      \begin{equation*}
	u_1<v_1<u_2<v_2<\cdots<u_h<v_h.
      \end{equation*}
      Then $R(z)$ is monotone decreasing in each one of the following intervals
      \begin{equation*}
	(-\infty,v_1), (v_1,v_2),\ldots, (v_{h-1},v_h),(v_h,\infty).
      \end{equation*}
  \end{itemize}
\item   Let
  \begin{equation*}
    R(z)=\frac{(z-u_1)\cdots(z-u_{h-1})}{(z-v_1)\cdots(z-v_h)}
    \text{with}\ v_1<u_1<\cdots<u_{h-1}<v_h
  \end{equation*}
  Then $R(z)$ is monotone decreasing in each one of the following intervals
  \begin{equation*}
	(-\infty,v_1), (v_1,v_2),\ldots, (v_{h-1},v_h),(v_h,\infty).
      \end{equation*}
\item Let  
  \begin{equation*}
    R(z)=\frac{(z-u_1)\cdots(z-u_{h+1})}{(z-v_1)\cdots(z-v_h)}
    \text{with}\ u_1<v_1<\cdots<v_{h}<u_{h+1}.
  \end{equation*}
   Then $R(z)$ is monotone increasing in each one of the following intervals
      \begin{equation*}
	(-\infty,v_1), (v_1,v_2),\ldots, (v_{h-1},v_h),(v_h,\infty).
      \end{equation*}
  \end{enumerate}
\end{lemma}

\bigskip
\bigskip

\noindent\textbf{Acknowledgements.} This material is based upon work supported by the National Science Foundation under Grant No. DMS-1928930 while ZL participated in a program hosted by the Mathematical Sciences Research Institute in Berkeley, California, during the Fall 2021 semester.  ZL acknowledges support from National Science Foundation grant 1608896 and Simons Foundation grant 638143. The author thanks anonymous reviewers for valuable comments on improving the readability of the paper.

\bibliography{fpmm2}
\bibliographystyle{plain}

\end{document}